\documentclass[11pt,letterpaper,oneside]{amsart}

\usepackage{amsmath,amsthm,amssymb,graphics,color}
\usepackage{hyperref} 
\usepackage{xcolor}
\usepackage[english]{babel}
\usepackage{tikz}
\usepackage{tikz-cd}
\usepackage[]{geometry}
\usepackage{mathrsfs}
\usepackage{soul}

\usetikzlibrary{shapes,calc,arrows,through,intersections}
\usetikzlibrary{decorations.pathmorphing}
\usetikzlibrary{quotes,angles}
\usetikzlibrary{calc}
\usetikzlibrary{shapes}
\usetikzlibrary{shapes.geometric}
\tikzset{>=latex} 

\usepackage{subcaption}
\usepackage{paralist}
\usepackage{graphicx}
\usepackage{esvect}
\usepackage{float}
\usepackage{array}
\usepackage{tensor}
\usepackage{amsthm}
\usepackage{amsfonts}
\usepackage{bbm}

\newcommand{\C}{\mathbb{C}}
\newcommand{\Q}{\mathbb{Q}}
\newcommand{\Z}{\mathbb{Z}}
\newcommand{\N}{\mathbb{N}}

\newcommand{\R}{\mathbb{R}}

\newcommand{\T}{\mathbb{T}}

\newcommand{\Id}{\text{Id}}
\newcommand{\im}{\text{Im}}

\newcommand{\supp}{\text{supp}}

\newcommand{\inj}{\text{inj}}
\newcommand{\eu}{\mathrm{e}}

\renewcommand{\phi}{\varphi}
\renewcommand{\theta}{\vartheta}
\renewcommand{\epsilon}{\varepsilon}

\newtheorem{theo}{Theorem}[section]
\newtheorem{prop}[theo]{Proposition}
\newtheorem{quest}[theo]{Question}
\newtheorem{coro}[theo]{Corollary}
\newtheorem{exa}[theo]{Example}
\newtheorem{lemm}[theo]{Lemma}
\newtheorem{conj}[theo]{Conjecture}
\theoremstyle{definition}
\newtheorem{def1}[theo]{Definition}
\theoremstyle{remark}
\newtheorem{rema}[theo]{Remark}

\newcommand{\nwc}{\newcommand}
\nwc{\Oph}{\operatorname{Op}_\hbar}
\nwc{\la}{\langle}
\nwc{\ra}{\rangle}

\newcommand{\Spec}{\text{Spec}}

\nwc{\mf}{\mathbf} 
\nwc{\blds}{\boldsymbol} 
\nwc{\ml}{\mathcal} 

\renewcommand{\Re}{\operatorname{Re}}

\newcommand{\eps}{\varepsilon}
\newcommand{\tr}{\text{Tr}\,}
\newcommand{\inv}{^{-1}}

\newcommand{\wt}{\widetilde}
\newcommand{\wh}{\widehat}

\renewcommand{\d}{\partial}

\newcommand{\vol}{{\operatorname{Vol}}}

\newcommand{\black}[1]{\color{black}}

\newcommand{\I}{\mathcal{I}}

\newcommand{\const}{\text{const.}}

\newcommand{\dt}{\delta}

\newcommand{\J}{\mathcal{J}}

\newtheorem{problem*}{Problem}

\newtheorem{theorem}{Theorem}[section]

\newtheorem{proposition}[theorem]{Proposition}

\newtheorem{conjecture*}{Conjecture}

\theoremstyle{definition}

\theoremstyle{remark}
\newtheorem*{remark}{Remark}

\newcommand{\CL}{\mathcal{L}}

\newcommand{\CD}{\mathcal{D}}

\newcommand{\CT}{\mathcal{T}}

\newcommand{\CO}{\mathcal{O}}

\newcommand{\CU}{\mathcal{U}}

\DeclareMathOperator{\sgn}{sgn}

\DeclareMathOperator{\rot}{rot}
\DeclareMathOperator{\osc}{osc}
\DeclareMathOperator{\singsupp}{sing supp}
\DeclareMathOperator{\leb}{Leb}

\newcommand\restr[2]{{
  \left.\kern-\nulldelimiterspace 
  #1 
  \vphantom{\big|} 
  \right|_{#2} 
  }} 

\newcommand{\normm}[1]{{\left\vert\kern-0.15ex\left\vert\kern-0.15ex\left\vert #1 
    \right\vert\kern-0.15ex\right\vert\kern-0.15ex\right\vert}}

\newcommand{\et}{\quad \text{and} \quad}

\newcommand{\dif}{\mathrm{d}}

\usetikzlibrary{cd}

\allowdisplaybreaks

\title{Spectral rigidity of two-dimensional Liouville tori}
\date{\today}

\author{Joscha Henheik}
\address{Department of Mathematics, University of Geneva, Geneva, Switzerland}
\email{joscha.henheik@unige.ch}
\author{Vadim Kaloshin}
\address{Institute of Science and Technology Austria, Klosterneuburg, Austria}
\email{vadim.kaloshin@gmail.com}
\author{Yunzhe Li}
\address{Institute of Science and Technology Austria, Klosterneuburg, Austria}
\email{yunzhe.li@ist.ac.at}
\author{Amir Vig}
\address{Department of Mathematics, University of Toronto, Toronto, Canada} \email{amir.vig@utoronto.ca}

\begin{document}

\begin{abstract}
It is a folklore conjecture that Liouville metrics are the only Riemannian metrics with integrable geodesic flow on the two-dimensional torus. In this paper, we study Laplace-isospectral deformations, within a fixed conformal class, of generic Liouville metrics and prove two rigidity results: First, any isospectral deformation that is affine in $\epsilon$ within the fixed conformal class is trivial. Second, any isospectral deformation that is analytic in the perturbation parameter and whose Taylor coefficients are trigonometric polynomials in the spatial variables remains Liouville and is obtained by componentwise rearrangement of the original metric. For the first result, the proof combines a wave-trace noncancellation result, which allows us to recover suitable length data from the Laplace spectrum, with a second-variation analysis of the energy functional along closed geodesics. For the second result, we prove that, within the Liouville class considered here, the relevant length-isospectrality condition is equivalent to componentwise rearrangement.
\end{abstract}

\maketitle

\tableofcontents

\section{Introduction and main results}\label{sec: Introduction and Main Results} 
Let $(X, g)$ be a compact Riemannian manifold and $\Delta_g$ its associated Laplace-Beltrami operator. We denote by $\Spec(-\Delta_g)$ the set of eigenvalues of $-\Delta_g$, i.e.
\begin{equation*}
   \Spec(-\Delta_g) = \{\lambda^2 \geq 0 : \exists u \in C^\infty(X) \text{ with }  -\Delta_g u = \lambda^2 u\}
\end{equation*}
regarded as a multiset, with eigenvalues repeated according to multiplicity. 
Suppose $(g_\epsilon)_{\epsilon \in (-\epsilon_0, \epsilon_0)}$ is a one-parameter family of Riemannian metrics on $X$.
We say that the family $(g_\epsilon)$ is \textbf{isospectral} if $\Spec(-\Delta_{g_\epsilon}) = \Spec(-\Delta_{g_0})$ for all $\epsilon \in (-\epsilon_0, \epsilon_0)$. 

\medskip

In this paper, we study spectral rigidity of \textbf{Liouville metrics} on the two-dimensional torus. That is, we consider the case $X = \T^2$, and take $g_0$ to be a Riemannian metric whose line element is of the form 
\begin{equation}\label{eq:deform}
    \dif s^2 = \big(1 + f_1(x_1) + f_2(x_2)\big) (\dif x_1^2 + \dif x_2^2).
\end{equation}
It is a well-known folklore conjecture that Liouville metrics are the only integrable metrics on $\T^2$, which can be thought of as a closed manifold analogue of the famous Birkhoff conjecture (see Section~\ref{sec:dynamics} for the definition of integrability and Section~\ref{sec:background} for a comparison with billiards).

\medskip

We now formulate
our first main result, whose proof is given in Section \ref{subsec:proof1} below.

\begin{theo}[First main result: Rigidity of affine conformal deformations] \label{thm:main2}
  Let $(g_\epsilon)$ be a family of Riemannian metrics on $\T^2$ of the form 
  \begin{equation}
      \label{eq:deformLiouLin}
        \dif s^2_\epsilon = \big(1 + f_1(x_1) + f_2(x_2) + \epsilon U(x_1, x_2)\big) (\dif x_1^2 + \dif x_2^2) \,, 
  \end{equation}
  where the pair $(f_1, f_2)$ belongs to an open and dense set of $C^\infty(\T) \times C^\infty(\T)$\footnote{\label{ftn:generic}This set is specified by the \emph{non-coincidence condition} \eqref{eq:INCC}; see also Appendix \ref{app:NCC}.} and $U \in C^\infty(\T^2)$ is arbitrary, such that, for $|\epsilon|$ small enough, the conformal factor in \eqref{eq:deformLiouLin} is positive.\footnote{The positivity shall henceforth always be understood implicitly, although not explicitly spelled out.} If the family $(g_\epsilon)$ is isospectral for small $|\epsilon|$, then $U \equiv 0$.
\end{theo}

Next, we formulate our second main result. Here, we assume that the perturbation $U(\epsilon, x_1, x_2)$ admits a series expansion $U(\epsilon, x_1, x_2) = \sum_{n \ge 1} \epsilon^n U^{(n)}(x_1, x_2)$, where each $U^{(n)} \in C^\infty(\T^2)$ is a trigonometric polynomial with $\max_{(x_1, x_2) \in \T^2} |U^{(n)}(x_1, x_2)| \le C^n $ for some $C > 0$. The proof of Theorem~\ref{thm:mainsecond} is given in Section \ref{subsec:proof2} below. 

\begin{theo}[Second main result: Analytic deformations and rearrangements] \label{thm:mainsecond}
  Let $(g_\epsilon)$ be a family of Riemannian metrics on $\T^2$ of the form 
  \begin{equation}
      \label{eq:deformLioutrig}
        \dif s^2_\epsilon = \big(1 + f_1(x_1) + f_2(x_2) +  U(\epsilon, x_1, x_2)\big) (\dif x_1^2 + \dif x_2^2) \,, 
  \end{equation}
  where the pair $(f_1, f_2)$ belongs to a generic set\footref{ftn:generic} of analytic functions in $C^\omega(\T) \times C^\omega(\T)$ and $U$ is as above, additionally satisfying that, for $|\epsilon|$ small enough, the conformal factor in \eqref{eq:deformLioutrig} is positive. Then, if the family $(g_\epsilon)$ is isospectral for small $\epsilon$, $U$ is separable, i.e.
  \begin{equation} \label{eq:separablemain}
    U(\epsilon, x_1, x_2)  = U_1(\epsilon, x_1) + U_2(\epsilon, x_2)
  \end{equation}
  and for any $\epsilon$ there exists $c_\epsilon \in \R$ such that $f_1+ U_1(\epsilon, \cdot)$ is a rearrangement of $f_1 + c_\epsilon$ and $f_2+ U_2(\epsilon, \cdot)$ is a rearrangement of $f_2 - c_\epsilon$. 

  Moreover, whenever $f_i \equiv \mathrm{const.}$ for a given $i \in \{1,2\}$, then the trigonometric-polynomial assumption on the $x_i$-dependence of the Taylor coefficients of $U$ may be replaced by analyticity in $x_i$.
\end{theo}

Recall that for measurable $f, \tilde{f}: \T\to \R$, $\tilde f$ is a \emph{rearrangement} of $f$, if the Lebesgue measures of their superlevel sets agree, i.e.~$|\{f > t\}| = |\{ \tilde f > t \}|$ for any $t \in \R$.\footnote{It is a separate interesting question whether a general function $f$ on $\T$ admits a non-trivial rearrangement of the form $f(x) + \sum_{n \ge 1} \epsilon^n U^{(n)}(x)$ where all the $U^{(n)}$'s are trigonometric polynomials.} 
See Figure~\ref{fig:rearrangement} for an illustration.
\begin{figure}
    \centering
    \includegraphics[scale=0.8]{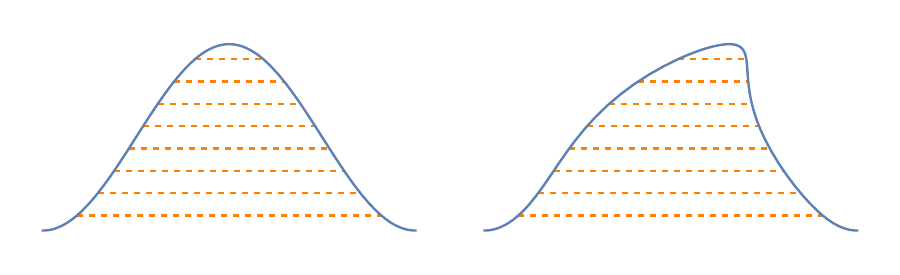}
    \caption{A rearrangement preserves the superlevel measure. The blue curves are the graphs of a function before and after a rearrangement.}
    \label{fig:rearrangement}
\end{figure}

\medskip

In the remainder of this introduction, we discuss our main results (Section~\ref{subsec:discuss}).

\subsection{Discussion of our main result} \label{subsec:discuss}
To our knowledge, Theorems~\ref{thm:main2}--\ref{thm:mainsecond} are among the first Laplace spectral rigidity results for closed manifolds whose geodesic flow is integrable. They follow a result on dynamical rigidity by the first author in \cite{henheik2025deformational}.

\subsubsection{Future directions} 
Recently, much of the literature on inverse spectral problems has focused on chaotic systems with or without boundary as well as integrable billiard dynamics in the plane. In future work, we hope to extend our rigidity result above to the following nondeformational setting:

\begin{conj}\label{conj:localquantum}
    Let $g_0$ be a Liouville metric on $\T^d$. There exists $\dt >0$ such that if $\|g_0 - g\|_{C^r} < \dt$ and $g$ is (Laplace) isospectral to $g_0$, then $g$ must be isometric to $g_0$.
\end{conj}

Conjecture~\ref{conj:localquantum} is a Laplace spectral type analogue of a closed manifold version of the (local) Birkhoff conjecture (see Conjecture~\ref{conj:localclassical} below). In the classical setting, work by the second author together with Corsi (\cite{CorsiKaloshin2018}), in addition to the results of Agapov-Bialy-Mironov \cite{ABM}, suggest that there may indeed exist non-Liouville metrics with integrable geodesic flow. Hence, we instead ask the following \textit{local} question regarding dynamical rigidity:
\begin{quest}\label{conj:localclassical}
    Let $g_0$ be a Liouville metric on $\T^d$. Is it true that there exists $\eps >0$ such that whenever $\|g_0 - g\|_{C^r} < \eps$ and the geodesic flow for the metric $g$ is integrable, $g$ must also be Liouville?
\end{quest}

While an answer to Question~\ref{conj:localclassical} remains out of reach for now, it may be more feasible if one restricts to the class of \textit{polynomially} integrable metrics (see Section~\ref{sec:background} for more details).

\begin{rema}
In subsequent work, the authors plan to investigate what happens to the dynamics and length spectrum of symmetric metrics such that the stable and unstable manifolds continue to coincide, both in the perturbative and nonperturbative regimes. In view of the work of Katok-Krikorian \cite{KatokKrikorian} and Piftankin-Treschev \cite{PiftankinTreschev}, we expect that accumulation of invariant tori at the separatrix will persist under small perturbations. The Katok-Krikorian construction also gives the possibility of the coincidence of stable and unstable manifolds such that invariant tori do not accumulate at the separatrix. We also raise the possibility of the persistence of rational invariant tori near the separatrix under non-Liouville perturbations. From a Laplace spectral viewpoint, it is also interesting to investigate the degenerate Bohr-Sommerfeld levels associated to the separatrix.
\end{rema}

\subsubsection{On the linearity assumption in Theorem~\ref{thm:main2}}
While the linearity assumption in Theorem~\ref{thm:main2} may appear to be an artificial limitation, it might in fact be necessary to obtain triviality of the deformation. This is indicated by the following example, which is a special case of a \emph{rearrangement} as discussed in Theorem \ref{thm:mainsecond}.

\begin{exa}[Two rivers] \label{ex:two rivers}
{For $c \in \R$,} consider a metric on $\T^2$ of the form $$g_c = (1 + f(x_2 - c) + f(x_2)) (\dif x_1^2 + \dif x_2^2)$$  where $f$ is a smooth nonnegative function supported in a small proper subinterval of $\T$.
Choose $c$ such that $f$ and $f( \cdot - c)$ have disjoint support so that $g_{{c}}$ is flat outside two disjoint horizontal strips, which we call rivers (see Figure~\ref{fig:tworivers}).

For this metric, it is not hard to see, for example by applying the Maupertuis principle (Section~\ref{subsubsec:Maupertuis} and Lemma~\ref{lem:length-liouville}), that all closed geodesics belong to one of two types: firstly, there are orbits which pass through both rivers at a nonzero angle, swim across in a symmetric but curved fashion, and then exit back into the flat region, where they continue in a straight line; secondly, there are the orbits which swim downstream along the supports of $f(x_2)$ or $f(x_2 - c)$. In particular, the nonnegativity of $f$ ensures that there are no orbits trapped in between the two rivers.

Observe now that, if one changes the parameter $c$, it is clear that neither type of orbit changes its length, but for different $c \mod 1$ corresponding to disjoint rivers, the corresponding metrics are not isometric! 
\end{exa}

\begin{figure}
    \centering
    \includegraphics[width=0.75\linewidth]{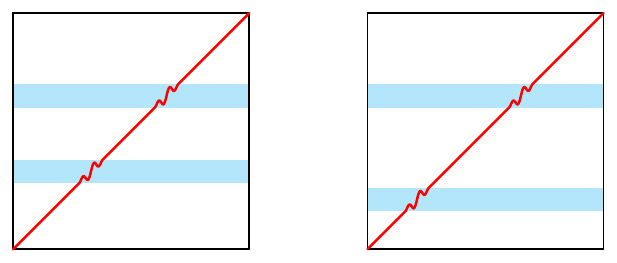}
    \caption{Two rivers (Example \ref{ex:two rivers}): Perturbing the flat metric along horizontal blue strips (the \emph{rivers}), we can then translate one of these rivers while preserving the length spectrum of the metric. More precisely, by varying the distance between the two blue regions, we produce nonisometric perturbations which are nonetheless length isospectral.}
    \label{fig:tworivers}
\end{figure}

Since the Laplace spectrum is typically accessed via the length spectrum (see also Section~\ref{sec:background}), this example suggests that there may exist length -- but not Laplace -- isospectral metrics in the family {$(g_{c+\epsilon})_\epsilon$, which is clearly deformed nonlinearly in $\epsilon$}. However, this example might be ruled out as something `nongeneric.'

\medskip

The construction of two rivers clearly generalizes to any finite number of (narrow) rivers. As we show in Theorem~\ref{thm:rearrangement-rotMLS} in Section \ref{sec:length-liouville} below, preservation of the length spectrum within a family of Liouville metrics is equivalent to it being a family of rearrangements (see also Theorem \ref{thm:liourearrange}). The idea to construct length isospectral metrics via rearrangements is inspired by recent work of Abbondandolo and Mazzucchelli (\cite{AbbondandoloMazzucchelli2026}). However, we should note that length isospectrality does not necessarily imply Laplace isospectrality, even under generic assumptions. In fact, we believe that the metrics in Example~\ref{ex:two rivers} are not Laplace isospectral. See the remark below the proof of Theorem \ref{thm:liourearrange} in Section \ref{sec:length-liouville} for  more details.

\begin{rema}
    {We mention that in this work we need to restrict to the two-dimensional setting. This is for two main reasons: First, Proposition \ref{prop:geom RT} does not have a known generalization to higher dimensions since the proof uses a transversality argument for planar curves. Second, Theorem \ref{SogasTheorem} (due to Soga \cite{Soga})  does not apply in dimensions $d \ge 3$.}
\end{rema}

\subsection{Outline} 
The rest of this paper is structured as follows. First, in Section \ref{subsec:proof}, we provide the proofs of Theorems \ref{thm:main2}--\ref{thm:mainsecond} by reducing the argument to the main steps and ideas required for it, essentially relying on four Theorems \ref{thm:isospec implies RI}, \ref{thm:RT rig}, \ref{thm:RI implies liouv}, and \ref{thm:liourearrange}. Then, in Section~\ref{sec:background}, we discuss the relevant background and results related to our main theorems. Afterwards, we enter the proof of the main ingredients for the proof of our main results: In Section~\ref{sec:dynamics}, we collect a few dynamical preliminaries, which are relevant in all following sections. In Section~\ref{sec:trace}, we establish a microlocal trace formula which allows us to deduce length isospectrality (for the relevant minimal rotational lengths) from Laplace isospectrality. This completes the proof of Theorem~\ref{thm:isospec implies RI}. Then, in Section~\ref{sec:secondorder}, we provide the proof of Theorem~\ref{thm:RT rig} (which together with Theorem \ref{thm:isospec implies RI} yields Theorem \ref{thm:main2}). Afterwards, in Sections \ref{sec:RIliouv}--\ref{sec:length-liouville} we give the proofs of the last two ingredients needed for the proof of our second main result (Theorem~\ref{thm:mainsecond}): In Section \ref{sec:RIliouv} we provide the proof of Theorem \ref{thm:RI implies liouv}. Finally, in Section~\ref{sec:length-liouville} we prove Theorem \ref{thm:liourearrange}. 
Genericity of the noncoincidence condition~\eqref{eq:INCC} appearing as the ``generic'' condition in Theorems~\ref{thm:main2}--\ref{thm:mainsecond} is proved in Appendix~\ref{app:NCC}. 

\subsection*{Acknowledgments}
We are grateful to Jacopo de Simoi and Ke Zhang for their interest in this work and helpful comments. We also thank Daniel Holmes, Leonid Polterovich and Vukašin Stojisavljević for discussions. Moreover, we gratefully acknowledge (partial) financial support by the ERC Grants \emph{ProbQuant} (jointly with the Swiss State
Secretariat for Education, Research and Innovation) and \emph{RMTBeyond \#101020331} (JH), and \emph{SPERIG \#885707} (VK, YL).

\section{Proofs of Theorems \ref{thm:main2} and \ref{thm:mainsecond} assuming the main ingredients} \label{subsec:proof}
In this section we prove Theorems~\ref{thm:main2} and~\ref{thm:mainsecond}, assuming the main ingredients established in the later sections.

\subsection{Proof of Theorem \ref{thm:main2}} \label{subsec:proof1}
The proof of our first main result in Theorem~\ref{thm:main2} is conducted in two main steps.

\subsubsection*{Step 1: Noncancellation of the wave trace} The first step asserts that isospectrality guarantees the preservation of rational invariant tori for an integrable metric $g_0$. We impose the generic assumption (genericity is proven in Appendix~\ref{app:NCC}) that $g_0$ satisfies a \emph{noncoincidence condition} (NCC), which posits that the lengths of geodesics with and without conjugate points do not coincide. See Definition~\ref{def:RI}.

\begin{theo}[Isospectrality and rational integrability]\label{thm:isospec implies RI}
    Let $(g_\epsilon)$ be an isospectral family of Riemannian metrics on $\T^2$ which depends continuously on $\epsilon$ in the $C^2$ norm.{Assume $g_0$ is rationally integrable}.
    \begin{enumerate}
        \item Suppose further that $g_0$ is rationally integrable and satisfies the noncoincidence condition~\eqref{eq:INCC} over some interval $I$ containing the length of periodic orbits in a given graph rational torus $\CT_0$. Then there exists a one-parameter family of rational tori $\CT_\eps \subset T^*X$ whose foliating orbits have the same length as those in $\CT_0$.
        \item If $g_0$ satisfies the noncoincidence condition~\eqref{eq:INCC} on $[0,T]$, then $g_\epsilon$ is rationally integrable up to time $T$ for all $|\epsilon|$ sufficiently small, i.e., for each $t \in \CL(\T^2,g_\eps)$ with $t \leq T$ corresponding to an $(m,n)$ periodic orbit and $mn \neq 0$, $\CL^{m,n}(\T^2,g_\eps) =\{t\}$ and the lifts to $T^*\T^2$ of periodic geodesics having length $t$ foliate a smooth, Lagrangian torus which is invariant under the geodesic flow.
    \end{enumerate}
\end{theo}
Theorem~\ref{thm:isospec implies RI} follows from a noncancellation result on the wave trace, which shows that under the local noncoincidence condition~\eqref{eq:INCC}, Laplace isospectrality implies length isospectrality with no assumption on the initial metric being Liouville; see Theorem~\ref{thm:laplacetolength} and the discussion immediately after it. Integrability then follows easily from length isospectrality. For each $x \in \T^2$ and each homology class $[\gamma] \in H_1(\T^2,\Z)$ corresponding to orbits \textit{without} conjugate points, we have a function  $\psi_\gamma(x)$ -- called the $\gamma$-loop function -- which returns the length of the \textit{unique} geodesic loop based at $x$ and belonging to the homology class $[\gamma]$ (see Corollary~\ref{cor:loop fn} and Definition~\ref{def:loop functions}). Integrability at $\eps = 0$ implies that $\psi_\gamma$ is constant on $\T^2$, and length isospectrality preserves this constancy for small $\eps > 0$, at least for $[\gamma]$ corresponding to graph rational tori (no conjugate points). As our definition of rational integrability concerns only graph tori, we conclude {rational} integrability.

\medskip

To show length isospectrality, we construct an explicit microlocal parametrix for the wave propagator $\cos t \sqrt{-\Delta}$ on $\T^2$ when the underlying Riemannian metric is a small perturbation of an integrable one (see Theorem~\ref{parametrix}). To do this, we follow a WKB algorithm with the machinery of Fourier integral operators. This parametrix is a closed manifold analogue of the ones which appeared in the fourth author's previous works \cite{Vig21} and \cite{Vig22} and uses off-diagonal extensions of the ``$\gamma$-loop functions'' (see Definition~\ref{def:loop functions}), which are analogous to those of the $q$-loop functions used in the study of convex billiards. By integrating the kernel of our parametrix over the diagonal, we then obtain a Marvizi-Melrose style trace formula (see Theorem~\ref{MM style parametrix}) which is versatile enough that it can be adapted to potentially irregularly distributed clusters in the length spectrum without requiring the individual lengths in the relevant cluster to be isolated or nondegenerate. It is also related to a parametrix constructed by Popov in \cite{PopovDegenerate}, although we avoid the use of fiber variables in our Fourier integrals so as to escape the risk of cancellations posed by nontrivial Maslov indices. Armed with this parametrix, we then follow the approach of Hezari-Zelditch (\cite{HeZe19}) by using a theorem of Soga on the decay (or lack thereof) of oscillatory integrals applied to the regularized resolvent trace (see~\eqref{resolvent}) to deduce length isospectrality. For more details, see Theorem~\ref{thm:laplacetolength} and its proof in Section~\ref{subsec:laplacetolength}. The first part of the rest of this paper is dedicated to the proof of Theorem~\ref{thm:isospec implies RI}, which relies on dynamical and microlocal arguments in Sections~\ref{sec:dynamics} and~\ref{sec:trace} respectively.

\subsubsection*{Step 2: A second order variational approach} Having Theorem~\ref{thm:isospec implies RI} at hand, we need to show that the preservation of rational tori and the preservation of lengths forces the perturbation to be trivial, $U \equiv 0$. In the following theorem, we show that the preservation of a \emph{single} rational torus is already sufficient for this conclusion, assuming that the perturbation depends on $\epsilon$ only \textit{linearly}, that is $U(\epsilon, x_1, x_2) = \epsilon U(x_1, x_2)$.

{\begin{theo}[Rigidity of one rational torus]\label{thm:RT rig}
     Let $(g_\epsilon)$ be a conformal family of Riemannian metrics on $\T^2$ of the form
     \begin{equation*}
         \dif s_\epsilon^2
         =
         \bigl(\Lambda(x_1,x_2)+\epsilon U(x_1,x_2)\bigr)
         (\dif x_1^2+\dif x_2^2),
     \end{equation*}
     where $\Lambda,U:\T^2\to\R$ are smooth. Assume that, for all sufficiently small $\epsilon$, the metric $g_\epsilon$ admits a graph rational torus of a fixed homology class, and that the common length of the closed geodesics in this torus is independent of $\epsilon$. Then $U\equiv0$.
\end{theo}}

The basic idea underlying the proof of Theorem~\ref{thm:RT rig} is to consider the \emph{energy} $E_\epsilon(\gamma_\epsilon)$ of the unique closed geodesic $\gamma_\epsilon$ in the preserved rational torus which starts and ends at a fixed $x\in\T^2$, as a function of $\epsilon$:
\begin{equation*}
 \epsilon\mapsto    E_\epsilon (\gamma_\epsilon) 
\end{equation*}
The isospectrality forces the above function to be a constant.
Since $\gamma_\epsilon$ is a geodesic of $g_\epsilon$, it is a critical point of the energy functional $E_\epsilon$, i.e., we have $\dif_{\gamma_\epsilon} E_\epsilon \equiv 0$.
Expanding the conditions $E_\epsilon(\gamma_\epsilon) = \text{const}$ and $\dif_{\gamma_\epsilon} E_\epsilon \equiv 0$ in powers of $\epsilon$ and combining the expansions (see \eqref{eq:E expansion}--\eqref{eq:dE expansion}), we obtain a second order condition:
\begin{equation}\label{eq:hess condition}
 \dif^2_{\gamma^{(0)}}E^{(0)}(\gamma^{(1)}, \gamma^{(1)}) = 0,
\end{equation}
where we have used the notation 
\begin{equation*}
    \gamma_\epsilon = \gamma^{(0)} + \epsilon \gamma^{(1)} + \epsilon^2 \gamma^{(2)} + \CO(\epsilon^3), \qquad E_\epsilon = E^{(0)} + \epsilon E^{(1)}.
\end{equation*}
This is the \emph{only}, but crucial, use of the linearity assumption of the deformation, because a general deformation will produce a nontrivial term on the left side of \eqref{eq:hess condition} involving $E^{(2)}$.
On the other hand, by minimality of $\gamma^{(0)}$ established in Proposition~\ref{prop:geom RT}, the Hessian $\dif^2_{\gamma^{(0)}}E^{(0)}$ is positive semi-definite.\footnote{The relevant closed geodesics are \emph{globally} minimizing in their homology class.} 
This in turn forces $\dif_{\gamma^{(0)}} E^{(1)} \equiv 0$. 
Comparing the Euler-Lagrange equations resulting from $\dif_{\gamma^{(0)}} E^{(0)} \equiv 0$ and $\dif_{\gamma^{(0)}} E^{(1)} \equiv 0$, we conclude that $U = 0$ along the unperturbed geodesic $\gamma^{(0)}$. 

\medskip

Finally, we use the assumed existence of a rational torus to find, for each $x \in \T^2$, a family of geodesics $\gamma_\eps$ such that $\gamma^{(0)}$ contains $x$, and apply the preceding argument to this family to deduce $U(x) = 0$. Varying $x$, we obtain $U \equiv 0$ on $\T^2$. The full proof of Theorem~\ref{thm:RT rig} is given in Section~\ref{sec:secondorder}.

\medskip 

Now, combining Theorems~\ref{thm:isospec implies RI} and~\ref{thm:RT rig}, we obtain our first main result, Theorem~\ref{thm:main2}.

\begin{proof}[Proof of Theorem~\ref{thm:main2}] 
Choose a graph rational invariant torus of \(g_0\) whose common length lies in a (finite) interval \(I\) on which the noncoincidence condition holds.\footnote{By Proposition~\ref{prop:NCC gen}, the $I$-noncoincidence condition holds for an open and dense set of $g_0$, which is precisely the open and dense set of $f_1$ and $f_2$ alluded to in Theorem~\ref{thm:main2}.} By Theorem~\ref{thm:isospec implies RI}, Laplace isospectrality preserves the corresponding graph rational torus and its length for all sufficiently small \(\epsilon\). Applying Theorem~\ref{thm:RT rig} with \(\Lambda=1+f_1+f_2\), we conclude \(U\equiv0\).
\end{proof}

\subsection{Proof of Theorem \ref{thm:mainsecond}} \label{subsec:proof2}
The proof of our second main result, Theorem~\ref{thm:mainsecond}, is conducted in three main steps.

\subsubsection*{Step 1: Noncancellation of the wave trace} As the first step, we again apply Theorem \ref{thm:isospec implies RI}, showing that isospectrality guarantees the preservation of rational invariant tori for the Liouville metric $g_0$. 
  
\subsubsection*{Step 2: Rationally integrable deformations of Liouville metrics}
As the second step, we use a variant of a result from one of us in \cite{henheik2025deformational}, showing that rationally integrable deformations of Liouville metrics are again Liouville metrics. 
{
\begin{theo}[Integrable deformations of Liouville metrics are Liouville metrics]\label{thm:RI implies liouv}
    Let $(g_\epsilon)$ be a family of Riemannian metrics on $\T^2$ as in Theorem \ref{thm:mainsecond} with $f_1, f_2$ belonging to a set of generic analytic functions and $U$ being a trigonometric polynomial as described above Theorem \ref{thm:mainsecond}. Then: 
\begin{itemize}
    \item[(i)]     Suppose that $g_\epsilon$ is rationally integrable up to time $T = \infty$ and the lengths of all graph rational invariant tori are preserved for every sufficiently small $\epsilon$. 
    Then, the perturbation $U(\epsilon, x_1, x_2)$ is \emph{separable}, i.e.~it satisfies \eqref{eq:separablemain}. In particular, $(g_\epsilon)$ is a family of Liouville metrics. 
    \item[(ii)]     Moreover, for any fixed $n \in \N$, there exists $T_n < \infty$ such that, if $(g_\epsilon)$ is rationally integrable up to time $T_n$ and the lengths of all graph rational invariant tori with length  at most $T_n$ are preserved for every sufficiently small $\epsilon$. Then there exist a constant $C_n> 0$ and trigonometric polynomials $U_i^{(j)} \in C^\infty(\T)$, $i \in\{1,2\}$, $j \in \{1,...,n\}$, such that 
    \begin{equation} \label{eq:orderbyorderstatement}
        \sup_{(x_1, x_2) \in \T^2} \left|U(\epsilon, x_1, x_2) - \sum_{j=1}^n \epsilon^j \big( U_1^{(j)}(x_1) + U_2^{(j)}(x_j)\big)\right| \le C_n \epsilon^{n+1} \,. 
    \end{equation}
    
\end{itemize}
\end{theo}
}
The basic mechanism underlying the proof of Theorem \ref{thm:RI implies liouv} relies on the idea from \cite{henheik2025deformational}, that the preservation of the $(m_1, m_2)$ rational torus leads to the annihilation of the Fourier coefficients $U_{k_1, k_2}$ with indices $(k_1, k_2) \in \{(m_1, m_2)\}^\perp$ of 
\begin{equation*}
    U(x_1, x_2) = \sum_{(k_1, k_2) \in \Z^2} U_{k_1, k_2} \mathrm{e}^{\mathrm{i} 2 \pi (k_1 x_1 + k_2 x_2)} \,. 
\end{equation*}
More precisely, we show non-degeneracy of a (possibly infinite) system of linear equations for the non-zero Fourier coefficients of $U$. This is obtained by considering the first order variations of the (by assumption constant) length functionals for different rational invariant tori, leading to relations of the form
\begin{equation*}
    \int_{\gamma_0} U \dif s_0 = 0
\end{equation*}
where $\gamma_0$ is an unperturbed geodesic with line element $\dif s_0^2$.

\medskip

The proof of Theorem \ref{thm:RI implies liouv}, relying on the arguments of \cite{henheik2025deformational}, is given in Section~\ref{sec:RIliouv}. 

\subsubsection*{Step 3: Length isospectral deformations within the class of Liouville metrics.} As the third step, we show that a smooth length-isospectral family of Liouville metrics must be a family of rearrangements. The proof of Theorem \ref{thm:liourearrange} is given in Section \ref{sec:length-liouville}. 

\begin{theo}[Length isospectral families of Liouville metrics] \label{thm:liourearrange}
    Let $(g_\epsilon)$ be a family of Liouville metrics with line element
\begin{equation}
    \dif s_\epsilon^2
= \bigl(1+f_1(\epsilon, x_1)+f_2(\epsilon, x_2)\bigr)\,(\dif x_1^2+\dif x_2^2)
\end{equation}
    and $f_1$, $f_2$ depending smoothly on $\epsilon$, $x_1$ and $x_2$. 
    Assume that the length of the geodesics of each rational torus of $g_\epsilon$ is independent of $\epsilon$. Then, for each $\epsilon$, there exists $c_\epsilon$ such that $f_1(\epsilon, \cdot)$ is a rearrangement of $f_1(0, \cdot) + c_\epsilon$ and $f_2(\epsilon, \cdot)$ is a rearrangement of $f_2(0, \cdot) - c_\epsilon$. 
\end{theo}

Theorem~\ref{thm:liourearrange} is a consequence of the more general Theorem~\ref{thm:rearrangement-rotMLS} in Section~\ref{sec:length-liouville}, which asserts that, roughly speaking, the \emph{marked length spectrum} of a Liouville metric determines the metric up to rearrangements.

\medskip

As pointed out in the introductory section, our results are similar in spirit to the prior results of Abbondandolo-Mazzucchelli~\cite{AbbondandoloMazzucchelli2026}. 
In the cited work, the authors proved, among other results, that the marked length spectrum determines the sphere of revolution with a single equator up to a rearrangement of the \emph{profile function}.
The authors of~\cite{AbbondandoloMazzucchelli2026} consider a symplectic twist map given by a return map associated with the unique equator, and use its marked spectral data to determine Mather's beta function which, in turn, determines the profile function via Abel's transform.

\medskip

The method of~\cite{AbbondandoloMazzucchelli2026} does not immediately extend to the setting of Liouville metrics, even in the rather similar setting with $f_2 = 0$.
By the method of~\cite{AbbondandoloMazzucchelli2026}, the spectral data of geodesics that oscillate around a fixed elliptic closed geodesic (corresponding to a local maximum of $f_1$) can, in principle, determine the metric over a subregion of $\T^2$ bounded by certain hyperbolic geodesics (corresponding to a local minimum of $f_1$).
However, unless $f_1$ has only two critical points, one needs to partition the length spectrum appropriately according to the oscillation intervals of such geodesics. But this information is generally unavailable from the usual definition of the length spectrum for $\T^2$ with or without marking.

\medskip

We shall modify the idea of~\cite{AbbondandoloMazzucchelli2026} and adopt a more ``global'' approach which uses only geodesics of homology classes $(m,n)$ with $mn \neq 0$.
In particular, the issues with multiple critical points and with nonconstant $f_1$ and $f_2$ can be handled in a unified way.

\medskip

The key observation is that, via a separable Hamiltonian obtained from the \emph{Maupertuis principle} (Section~\ref{sec:dynamics}), the marked length-spectral data of a Liouville metric $(1+f_1(x_1)+f_2(x_2))(\dif x_1^2+\dif x_2^2)$ can be encoded by a real-analytically embedded curve $\Gamma_{f_1, f_2} \subset (\R_{>0})^2$ that is given by the image of the map
\begin{equation}\label{eq:Gam curve intro}
    (-\min f_1, 1 + \min f_2) \ni \ e \longmapsto \bigg(\int_{\T}\sqrt{e + f_1(x)} \dif x, \int_{\T}\sqrt{ 1- e + f_2(x)} \dif x \bigg) \ \in \R^2
\end{equation}
It is straightforward to see that individual rearrangements of $f_1$ and $f_2$ do not change the curve $\Gamma_{f_1, f_2}$, and consequently, are always isospectral deformations.
Conversely, if $(1+\tilde f_1(x_1)+\tilde f_2(x_2))(\dif x_1^2+\dif x_2^2)$ is another, marked-length-isospectral, Liouville metric, then $\Gamma_{f_1, f_2} = \Gamma_{\tilde f_1, \tilde f_2}$.
In particular, we show that the analogous map of~\eqref{eq:Gam curve intro} with $f_j$ replaced by $\wt f_j$ must coincide with the map~\eqref{eq:Gam curve intro} up to precomposing a real-analytic reparametrization $\varphi:\R\to \R$.
A closer analysis of $\varphi$ shows that $z \mapsto \varphi(z)/z$ can be analytically extended to $\C$ as a bounded function, forcing it to be constant, and then the asymptotics at $|z| \to \infty$ shows $\varphi(z)=z$.
It then follows from an application of Laplace transform that the function $\tilde f_j$ rearranges $f_j$ up to an additive constant for $j = 1,2$.
The detailed proof is given in Section~\ref{sec:length-liouville}. 

\medskip

Now, combining Theorem \ref{thm:isospec implies RI} with Theorems \ref{thm:RI implies liouv}--\ref{thm:liourearrange}, we obtain Theorem~\ref{thm:mainsecond}. 

\begin{proof}[Proof of Theorem \ref{thm:mainsecond}]
     {As in the proof of Theorem \ref{thm:main2}, by Theorem~\ref{thm:isospec implies RI}, Laplace isospectrality implies integrability and preservation of lengths on any interval $I$ for which the noncoincidence condition~\eqref{eq:INCC} holds.\footnote{By Proposition~\ref{prop:NCC gen}, the noncoincidence condition holds for a generic set of $g_0$ for $T = \infty$. The intersection of this set with the generic set of $(f_1, f_2)$ in Theorem \ref{thm:RI implies liouv} is precisely the generic set alluded to in Theorem \ref{thm:mainsecond}.} Now, applying Theorem \ref{thm:RI implies liouv}~(ii) order by order, yields the separability of $U$, leaving us with a family of Liouville metrics with preserved lengths of rational invariant tori. By application of Theorem~\ref{thm:liourearrange} we then conclude the stated componentwise rearrangement conclusion, claimed in Theorem~\ref{thm:mainsecond}.}
\end{proof}

\section{Background and related results}\label{sec:background}

Inverse spectral problems arise in many contexts. Two of the most well studied are:
\begin{enumerate}
	\item \label{Kac} \emph{Can you hear the shape of a drum?} This question, posed by Mark Kac in 1966 \cite{kac1966can}, asks whether or not one can recover the underlying geometry of a Riemannian manifold (drum) $(X,g)$, from knowledge of its ostensibly audible Laplace spectrum:\footnote{Here, $B$ is a boundary operator encoding Dirichlet, Neumann, Robin or mixed boundary conditions. If the manifold is closed, there is of course no boundary operator needed.}
	\begin{align*}
		\begin{cases}
			- \Delta_g u = \lambda^2 u,\\
			Bu|_{\d X} = 0,
		\end{cases}
	\end{align*}

	\item \label{inverselsp} Can you recover the geometry of a manifold (with or without boundary) from knowledge of the lengths of closed geodesics (which we denote by $\CL(X,g)$)?
\end{enumerate}

\medskip 

Each of these questions has a number of close relatives. For example, recovering a potential energy function from knowledge of its associated Schr\"odinger spectrum is formally analogous to \eqref{Kac}. As another example, accompanying the length spectrum in \eqref{inverselsp} is the so called \textit{marked} length spectrum, which encodes the lengths of periodic orbits together with their topological data (e.g., winding number or homotopy class).

\medskip

In physics, when one solves the wave or Schr\"odinger equations to obtain the wave functions for a quantum particle, separation of variables leads one to study eigenvalues and eigenfunctions of the Laplacian. For this reason, the class of inverse problems which resemble \eqref{Kac} are said to be ``quantum'' inverse problems. The famous Bohr correspondence principle asserts that in the high-energy regime (large eigenvalue or small Planck's constant), the behavior of quantum wave functions should resemble that of the underlying classical dynamics, i.e., solutions to Hamilton's equations for the relevant classical Hamiltonian. Inverse problems pertaining to the length spectrum, as in \eqref{inverselsp}, are hence said to be ``classical.''

\medskip

In the remainder of this section, we discuss the necessary background in addition to several related results on quantum and classical inverse spectral problems. For further details and additional references, we refer the reader to the surveys \cite{ZelditchSurvey2, datchev2011inverse}. 

\subsection{The Poisson relation}

There is a close mathematical relation between \eqref{Kac} and \eqref{inverselsp} of semiclassical nature, which, to some extent, formalizes the Bohr correspondence principle. This is most evident through the \textit{Poisson relation}, which dictates that

\begin{align}\label{PoissonRelation}
	\text{SingSupp} \,\,\tr \left(\cos t \sqrt{- \Delta}\right) \subset \pm \overline{\CL(X,g)} \cup \{0\}.
\end{align}

The left-hand side of \eqref{PoissonRelation} is the distributional trace of the wave group
\begin{align}\label{eq:wave trace}
	w(t): = \sum \cos t \lambda_j \in \mathcal{D}'(\R),
\end{align}
where $\lambda_j$ are eigenvalues of the Laplacian. This is an entirely spectral quantity, while the length spectrum on the right-hand side of~\eqref{PoissonRelation} is purely geometric.

\medskip

To tackle the inverse Laplace spectral problem, one often uses the Poisson relation~\eqref{PoissonRelation} together with its more refined counterpart, the Poisson summation formula, which we now describe. Let $L \in \CL(X,g)$ be isolated and assume that all periodic geodesics $\gamma$ of length $L$ are nondegenerate in the sense that $\det(\Id - P_\gamma) \neq 0$, where $P_\gamma$ is the linearized Poincar\'e map associated to $\gamma$. Denote by $\wh \rho(t)$ a compactly supported test function which is $1$ in a neighborhood of $L$ and contains no other lengths of closed geodesics in its support. The Poisson summation formula then reads
\begin{align}\label{eq: PSF}
	\int_0^\infty e^{i t \lambda} \wh \rho(t) w(t) \dif t \sim \sum_{\gamma: \text{length}(\gamma) = L} 
	\frac{L_\gamma^\# i^{\sigma_\gamma} }{|\det(I-P_\gamma)|^{1/2}} \sum_{j = 0}^\infty B_{\gamma, j} \lambda^{-j} \quad \text{as}\,\, \lambda \to \infty,
\end{align}
where $\sigma_\gamma$ is the Morse index of $\gamma$ and $L_\gamma^\#$ is its primitive period. The coefficients $B_{\gamma,j}$ contain geometric information about the underlying manifold in the form of integrals of polynomials in the metric and its derivatives over each closed geodesic of length $L$. The Poisson summation formula~\eqref{eq: PSF} is a generalization of the Selberg trace formula, which was developed for hyperbolic surfaces in 1956 (\cite{Sel56}). It was followed by the Gutzwiller trace formula in the physics literature, which relates the eigenvalues of a Schr\"odinger operator to the underlying classical Hamiltonian dynamics (\cite{Gutzwiller}). Around the same time, Balian and Bloch \cite{BB1, BB2, BB3} treat the case of bounded domains. In the mathematical literature, the trace formula was first proved for Riemannian manifolds by Colin de Verdi\`ere (\cite{CdVsllgp1}, \cite{CdVsllgp2}), Chazarain (\cite{Chazarain1974}), and Duistermaat-Guillemin (\cite{DuGu75}). It was later extended by Chazarain \cite{Ch76} and Guillemin-Melrose (\cite{GuMe79a}) to the case of manifolds with boundary.

\medskip

The invariants $B_{\gamma, j}$ in \eqref{eq: PSF} are algebraically equivalent to the \textit{wave invariants}, which are the coefficients of a direct expansion of $w(t)$, as opposed to its Fourier transform. If all geodesics are isolated and nondegenerate, the wave trace is asymptotic to
\begin{align*}
	 w(t) \sim \begin{cases}
	 	\sum_{k = 0}^\infty \Re \left(a_{0, k-d} (t +i0)^{k - d} \right) & \text{near}\,\, t= 0,\\
	 	\sum_{k = 0}^\infty \Re \left(a_{L, k-1} (t - L +i0)^{k - 1} \log(t - L + i 0) \right) & \text{near}\,\, t= L \in \CL(X,g) \backslash \{0\},
	 \end{cases}
\end{align*}
modulo $C^\infty$ functions of $t$, with each coefficient
\begin{align}\label{eq: sum over orbits}
	a_{L, j} = \sum_{\text{length}(\gamma) = L} a_{\gamma, j}
\end{align}
accounting for the contributions of different geodesics when the length spectrum is multiple. We note that there are more general formulas when periodic orbits are degenerate but belong to a clean submanifold of fixed points for the geodesic flow.

\subsection{Related results} Generically, the wave invariants $a_{L,j}$ are nonzero and one can read off the length spectrum from the Laplace spectrum. If the length spectrum is multiple, different periodic orbits of the same length could have identical wave invariants except for Maslov indices which differ by $2 \mod 4$, in which case their contributions would cancel \cite{CdVsllgp1}, \cite{DuGu75}, \cite{KV24}.
\begin{figure}
    \centering
    \includegraphics[width=0.75\linewidth]{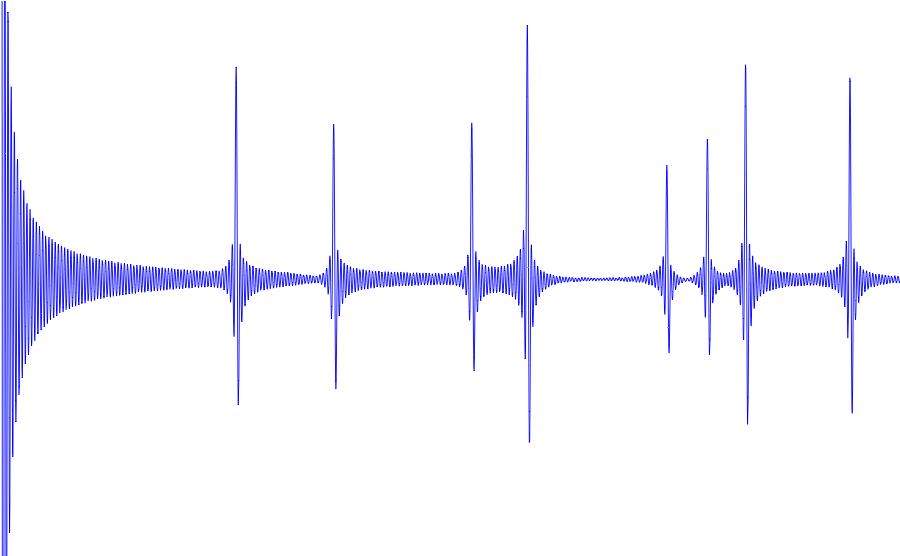}
    \caption{An approximation of the wave trace on $\R^2/\Z^2$ with the Euclidean metric, made with finitely many eigenvalues.}
    \label{fig:placeholder}
\end{figure}
One important case in which there is no risk of cancellation is in the setting of negative curvature. Since there are no conjugate points, all Morse indices are trivial and Laplace isospectral domains must also be length isospectral. Guillemin and Kazhdan exploited this relation to show deformational spectral rigidity of negatively curved surfaces (\cite{GK1}, \cite{GK2}). If the length spectrum is constant along a deformation $(X,g_\eps)$, then
\begin{align}\label{Xray}
	\int_{\gamma} \dot{g} \, \dif s = 0
\end{align}
for each closed geodesic $\gamma$, where the first variation $\dot{g} = \frac{d}{d \eps} \big|_{\eps = 0} g_\eps$ of the metric is a symmetric $2$-tensor. Rigidity then follows from injectivity of the so called \textit{X-ray transform}~\eqref{Xray}. A discrete analogue of this was used in \cite{KaDSWe17} with the X-ray transform replaced by the linearized isospectral operator. In general, existence of conjugate points complicates the analysis of both the wave trace and the X-ray transform, often jeopardizing its injectivity. For example, on the sphere, any odd function is in the kernel of the X-ray transform. For integrable systems, there are almost always conjugate points.

\medskip

In the papers \cite{KKV} and \cite{KV24}, it was shown that the length spectrum can indeed be obscured by the wave trace if distinct orbits conspire against one another to cause cancellations in the sum~\eqref{eq: sum over orbits}. We also mention the work of Schüth \cite{Schueth}, who constructed nonisometric two step nilmanifolds with identical Laplace spectra, but one having integrable geodesic flow and the other not. They do, however, have the same length spectrum, when multiplicities are ignored. She also produces a continuous family of isospectral but nonisometric two step nilmanifolds in \cite{Schueth2}. Despite these {caveats and} limitations, there has been much recent progress on quantum inverse spectral problems. Hezari and Zelditch recently improved on the results of Kac by showing that ellipses of small eccentricity are spectrally determined among all smooth planar domains (\cite{HeZe19}). Additionally, there are nonexplicit examples discovered by Marvizi-Melrose (\cite{MM}) and Watanabe (\cite{Watanabe1}, \cite{Watanabe2}). Zelditch has also shown that generic $\Z_2$ symmetric analytic domains are spectrally determined within that class (\cite{Zel09}). This was later extended to generic centrally symmetric analytic domains (\cite{HezariZelditchCentrallySymmetric}).

\medskip

The interior problem was complemented by the inverse resonance problem for obstacle scattering in \cite{Zel0res}. In \cite{MelroseIsospectralCompactness, POS1, POS2}, $C^\infty$-compactness of Laplace isospectral sets was shown, while spectral rigidity was shown in \cite{HeZe12}, \cite{Zelditch5}, \cite{Vig21}, \cite{PopTop12}, \cite{PopovTopalovKAM}, and \cite{HezariZelditchEigenfunctionAsymptotics}. In the boundaryless case, Zelditch also showed that analytic metrics of revolution on $S^2$ are generically spectrally determined (\cite{Zelditch2}), which improved on prior results assuming mirror symmetry (see \cite{BruningHeintze}). There are several other inverse spectral results which instead consider the \textit{joint} spectrum of $-\Delta$ and the generator of rotational symmetry (\cite{BerardSOR}, \cite{Gurarie}, \cite{KacReflection}). 

\medskip

On the dynamical side, as in the Laplace spectral setting, there are four settings: integrable or chaotic, each either with or without boundary. In the setting of convex billiard tables in the plane, the second author together with Sorrentino showed in \cite{KaSo16} that ellipses are isolated among integrable billiard tables, verifying a version of the famous \emph{Birkhoff-Poritsky conjecture}; see \cite{birkhoff1927periodic, Poritsky} for the classical references, and also \cite{KaAvDS16}, \cite{koval2023local}, \cite{KovalGevrey},  \cite{KaDSWe17}, and \cite{KaSo16}. In \cite{KaDSWe17, KaSo16}, the authors show dynamical spectral rigidity of ellipses and nearly circular domains. In \cite{BialyMironov}, Bialy and Mironov show that the Birkhoff-Poritsky conjecture holds among centrally symmetric domains for which the interval of rotation numbers $(0,1/4]$ corresponds to a foliation by caustics near the boundary. In the context of the \textit{marked} length spectrum, the fourth author demonstrated $C^\infty$ compactness of isospectral sets in \cite{VigBetaCompactness}. {More recently, Abbondandolo and Mazzucchelli \cite{AbbondandoloMazzucchelli2026} extended the work of Zelditch to the dynamical setting and proved that metrics of revolution on the 2-sphere with identical marked length spectra have geodesic flows which are smoothly conjugate. They also show that each isospectral class has a unique mirror symmetric representative which is uniquely determined among all such metrics.} For the chaotic boundaryless case, we refer to the celebrated works of Otal \cite{otal1990spectre} and Croke \cite{croke1990rigidity}, who proved marked length spectral rigidity for manifolds with non-positive curvature.

\medskip

In the boundaryless case, there has been much recent work on the dynamical \emph{spectral} rigidity of negatively curved metrics (see, e.g., \cite{GuillarmouLefeuvre}), but comparatively little in the integrable setting. One restrictive class of examples was presented by Zelditch in \cite{Zelditch2}, who showed that under suitable nondegeneracy conditions, the wave trace invariants give the so called ``quantum Birkhoff normal form,'' which determines (but is not determined by) the \textit{classical} Birkhoff normal form of an elliptic geodesic on a surface of revolution. Under a mirror symmetry assumption, the classical Birkhoff normal form in turn determines the germ of the profile function there; in general, it only determines the even part. See also \cite{BruningHeintze} for an alternative proof and \cite{CdVbilliards} for an analogous result in the billiards setting. However, Zelditch explicitly excludes Liouville tori from his study and instead, considers Laplace (as opposed to length) isospectrality of a more general class of (not necessarily mirror symmetric) metrics of revolution on $S^2$. 

\medskip

In view of the above mentioned results, the first main result of our paper is in some ways, a closed manifold analogue of \cite{KaDSWe17}, although we only {(but potentially necessarily -- cf. Example~\ref{ex:two rivers})} consider linear in $\epsilon$ deformations in Theorem~\ref{thm:main2} and employ second (as opposed to first) variational methods to conclude rigidity from the preservation of a single rational torus, without utilizing the entire length spectrum. On the dynamical side, we also remark on an analogy between our proof and the recent work \cite{ArnaudMassettiSorrentino2023} by Arnaud-Massetti-Sorrentino, who prove (see \cite[Corollary~1.15]{ArnaudMassettiSorrentino2023}) dynamical rigidity of integrable \emph{analytic} twist maps, assuming that only \emph{one} rational torus is preserved. Notably, our result does not require analyticity. We also mention that, within the proof of our second main result, Theorem \ref{thm:mainsecond}, in particular Theorem \ref{thm:liourearrange}, we show that there exist non-trivial families of Liouville metrics, obtained by rearrangement, having the same marked length spectrum (for rotation numbers with $nm \neq 0$, i.e.~for graph rational invariant tori). This contrasts the result of Otal \cite{otal1990spectre} and Croke \cite{croke1990rigidity} on marked length spectral rigidity for negatively curved manifolds.

\subsection{Liouville metrics on $\T^2$}
The simplest manifold featuring geodesic flow with nontrivial behavior is the two-dimensional torus $\T^2$. Analogous to the above mentioned Birkhoff conjecture, it is a folklore belief that the only metrics on $\T^2$ which have \emph{integrable} geodesic flow are Liouville metrics with line-element of the form~\eqref{eq:deform}. Examples of integrable geodesic flows on manifolds besides tori can be found in \cite{PaternainSpatzier1990}.

\medskip

Although the correctness of the folklore belief appeared plausible for a long time, there is rather strong evidence for it being in fact \emph{false} in its very general form: In \cite{CorsiKaloshin2018}, the second author of the present paper together with Corsi constructed a Hamiltonian counterexample that is \emph{locally} integrable in a $p$-cone in the cotangent bundle. This means that on a fixed energy level, there exists an analytic change of variables, which transforms the Hamiltonian with non-Liouville potential to the standard form $(p_1^2 + p_2^2)/2$, however only for momenta $p_i$ in a certain cone in $\R^2$. We also mention examples of \emph{magnetically} integrable metrics which are not of Liouville type, constructed by Agapov-Bialy-Mironov in \cite{ABM}. 

\medskip

However, despite these delicate examples, certain suitably weakened conjectures are still believed to be true. These are supported by a variety of partial results obtained in this direction, starting with the classical works of Dini \cite{Dini43}, Darboux \cite{Darboux37}, and Birkhoff \cite{Birkhoff},  and further developed in \cite{BabNek95,kiyohara1991,kolokoltsov}. In particular, several works by Bialy and Mironov \cite{bialy1987, bialymironov1, bialymironov2, bialymironov3}, Denisova, Kozlov, and Treshev \cite{denisovakozlov1, denisovakozlov2, denisovakozlov3, kozlovtreshchev, denisovakozlovtreshev2012}, Mironov \cite{mironov}, and other authors \cite{babenkoneko1995, kolokoltsov, agapov, taimanov}, strongly indicate the validity of the following (yet unproven) conjecture: \emph{Every polynomially integrable\footnote{Polynomially integrable metrics are those whose second conserved quantity is a polynomial in the momentum variables.} metric $g$ on $\T^2$ is of Liouville type.}
We refer to \cite{probl1, probl2} for recent surveys on open problems and questions concerning geodesics and integrability of finite-dimensional systems. 

\medskip

The first author of the current paper, in \cite{henheik2025deformational}, studied a version concerning deformations of the folklore conjecture and showed that the class of Liouville metrics is \emph{deformationally rigid against a fairly large class of integrable conformal perturbations}, i.e., integrable deformations of Liouville metrics are again Liouville metrics. The proof proceeds by showing that possibly infinitely many linear equations for the Fourier coefficients of the deformation coming from the X-ray transform on the torus with a Liouville metric are nondegenerate. The structure of these equations was analyzed as an infinite Gram matrix. 

\medskip

Despite the considerable efforts towards resolving the folklore conjecture, there is only very sparse literature on the quantum version of it. In the 1990s, Kosygin-Minasov-Sinai \cite{KosyginMinasovSinai} and later Bleher-Kosygin-Sinai \cite{BleherKosyginSinai} and Vorobets \cite{Vorobets} studied the remainder term in Weyl's law (which is supposed to be sensitive to the integrable/nonintegrable nature of the geodesic flow) for the eigenvalues of the Laplace-Beltrami operator associated to a Liouville metric. Roughly speaking, they showed that the eigenvalue counting function $N(x) = \#\{\text{eigenvalues of } -\Delta_g \ \le x\}$ satisfies the asymptotics
\begin{equation}
    N(x) = \frac{1}{4 \pi} \mathrm{Area}_g(\T^2) x + x^{1/4} \theta(\lambda^{1/2}) \quad \text{as} \quad x  \to \infty
\end{equation}
for a large class of functions $f_1, f_2$ in \eqref{eq:deform}, where $\theta$ is  some Besicovitch almost periodic function. We refer to \cite{DAPopov} for a more recent survey. Our paper appears to be the first to study the very natural problem of (deformational) spectral rigidity of Liouville metrics.

\section{Dynamical preliminaries} \label{sec:dynamics}

In this section, we collect several preliminary dynamical results and definitions needed for the rest of this paper.

\subsection{Geodesic flows and connection to Hamiltonian systems}\label{sec:geod-hamilton}

The Riemannian metric $g$ determines an $\R$-action on the tangent bundle $T X$ of $X$ called the \textbf{geodesic flow}, which we shall denote by $\Phi^t_g: TX \to TX$ for $t \in \R$.
The geodesic flow restricts to a flow on the unit tangent bundle of $X$ with respect to the metric $g$, which we shall denote by $SX$.

\subsubsection{Hamiltonian formalism} \label{subsubsec:Hamiltonian} 

To relate the geodesic flow to the wave equation, we first convert it to a flow on the cotangent bundle, which we view as a Hamiltonian system. Denote by $(x, \xi)$ symplectic coordinates on $T^*X$, with $\xi = g_x(v,\cdot)$ for $v \in TX$, and consider the Hamiltonian $H = \frac{1}{2} |\xi|_{g^{-1}}^2$. It is easy to check, either in coordinates or by using Hamiltonian-Lagrangian correspondence for fiberwise convex Lagrangians via the Legendre transform, that projections of integral curves of the Hamiltonian vector field onto the configuration space $X = \T^2$ are geodesics\footnote{In the microlocal framework, it is actually standard to use $|\xi|_{g^{-1}}$ on the punctured cotangent bundle $\dot{T}^*X =T^* X - \{0\}$, which has the advantage of being homogeneous of degree one, corresponding to the order of $\sqrt{-\Delta}$}. That is, they are solutions of the system 
\begin{align*}
    \begin{cases}
        \dot{x} =  \frac{\d H}{\d \xi}\\
        \dot{\xi} = - \frac{\d H}{\d x}.
    \end{cases}
\end{align*}
We will denote by $\Phi_H^t(x_0,\xi_0) = (x(t), \xi(t))$ the flow map and by $\gamma$, individual geodesic lifts. Since $H$ is itself a first integral, that is $ H \circ \Phi_H^t = H$ for all $t \in \R$, it is convenient to restrict our qualitative study of the dynamics to the energy surface $S^*(X):= \{(x,\xi) \in T^*X : H(x,\xi) = 1\}$, which we will call the cosphere bundle\footnote{That $|\xi|^2 = 2$ on $S^*X$ is immaterial. Furthermore, if we let $H_1(x,\xi) = H^{1/m}(x,\xi)$, then the Hamiltonian vector fields of $H$ and $H_1$ are simply rescalings of one another: $V_H = m V_{H_1}$ on $S^*X$, so the dynamics are indistinguishable up to reparametrization}.

\subsubsection{Maupertuis principle (see, e.g., \cite{bolsinov1995maupertuis})} \label{subsubsec:Maupertuis}  
We recall the \textit{Maupertuis principle}, a useful tool for describing geodesics of Liouville metrics. Take a compact Riemannian manifold $(M,\widetilde g)$ and let 
\begin{equation} \label{eq:natHamilton}
\widetilde{H}(x,\xi) = \frac{1}{2}\sum_{i,j} \widetilde{g}^{ij}(x)\xi_i\xi_j - V(x),
\end{equation}
be a mechanical Hamiltonian on $T^*M$, where $V \in C^2(M)$ is a potential function. Moreover, let $T_h = \{\widetilde{H}(x,\xi) = h\}$ be an isoenergy submanifold for some $h > -\min_x V(x)$. We then note that $T_h$ is an isoenergy submanifold for another system with Hamiltonian 
\begin{equation} \label{eq:H-1}
H(x,\xi) = \frac{1}{2}\sum_{i,j} \frac{\widetilde g^{ij}(x)}{h+V(x)}\xi_i\xi_j\,,
\end{equation}
as well. That is, $T_h = \{H(x,\xi) = 1\}$. The \emph{Maupertuis principle} states that the integral curves of the Hamiltonian vector fields generated by $\widetilde{H}$ and $H$ on the fixed isoenergy submanifold $T_h$ coincide up to reparametrization. 
Finally, we note that the vector field generated by $H$ corresponds, via the Hamiltonian formalism (see Section~\ref{subsubsec:Hamiltonian}), to the geodesic flow of the Riemannian metric $g$ with
\begin{equation*}
    g_{ij}(x) = (h+V(x))\widetilde{g}_{ij}(x)\,. 
\end{equation*} 
We will use this for
\begin{equation} \label{eq:Maupertuis}
M=\T^2,\qquad \widetilde{g}_{ij}=\delta_{ij},\qquad h=1,\qquad V(x)=f_1(x_1)+f_2(x_2).
\end{equation}
In this case, the orbits are determined up to time reparametrization by the mechanical Hamiltonian 
\begin{equation}\label{eq:H separable}
    \widetilde H(x, \xi) = \frac{|\xi|^2}{2} - f_1(x_1) - f_2(x_2),
\end{equation}
restricted to the energy surface $T_h = \{ \widetilde{H}(x,\xi) = 1\}$.
A standard proof of the integrability of the geodesic flow of a Liouville metric consists of observing that the Hamiltonian $\widetilde H$ is \emph{separable}, and hence admits a second independent first integral, which can be taken to be 
\begin{equation*}
    \frac{\xi_1^2}{2} - f_1(x_1).
\end{equation*}
We refer to~\cite{henheik2025deformational} for a more detailed discussion of the Maupertuis principle and its application to Liouville metrics on $\T^2$.

\subsection{Length spectrum and the noncoincidence condition}
Recall that to any Riemannian manifold $(X,g)$, we define the length spectrum $\CL(X,g)$ to be the collection of lengths of closed geodesics. 
We now consider the case $X = \T^2$. 
Every closed geodesic of $(\T^2, g)$ can be uniquely associated with a homology class parametrized by the homology group $H_1(\T^2, \Z) \simeq \Z^2$.
Therefore, one can partition the length spectrum into $(m,n)$-\textbf{length spectra} $\CL^{m,n}(\T^2,g) \subset \R_{\geq 0}$, which are the lengths of all closed geodesics belonging to the homology class $(m,n) \in \Z^2$.
We remark for clarity that, if $(m,n) \neq (0,0)$ is nonprimitive, i.e., we have $(m,n) = (km',kn')$ for a pair $(m',n')$ and an integer $k$ with $|k| \geq 2$, then $\CL^{m,n}(\T^2,g)$ contains the lengths of $|k|$-fold iterates of closed geodesics in the homology class $(m',n')$.

\medskip

To state the exact condition on $f_1$ and $f_2$ in Theorem~\ref{thm:main2}, as well as the noncoincidence condition in Theorem~\ref{thm:isospec implies RI}, we shall consider two particular sets of spectral data $\CL^{\osc}(\T^2,g)$ and $\CL^{\rot}(\T^2,g)$, defined as follows.

\begin{def1} \label{def:osc Lsp}
    Consider a Riemannian metric $g$ on $\T^2$. We define
    \begin{equation}\label{osc Lsp}
        \CL^{\rot}(\T^2,g) := \bigcup_{\substack{(m,n) \in \Z^2\\ mn \neq 0}} \CL^{m,n}(\T^2,g)
        \et
        \CL^{\osc}(\T^2,g) := \bigcup_{\substack{(m,n) \in \Z^2\\ mn = 0}} \CL^{m,n}(\T^2,g)
    \end{equation}
    to be the \textbf{rotational} and \textbf{oscillatory length spectra}, respectively. 
\end{def1}

\begin{def1}\label{def:NCC}
    Let $g$ be a Riemannian metric on $\T^2$ and let $I \subset \R_{\geq 0}$ be a nonempty closed interval. 
    We say $g$ satisfies the \textbf{noncoincidence condition (NCC)} on $I$ if
    \begin{align}\label{eq:INCC}
     I \cap \CL^{\rot}(\T^2,g) \cap   \CL^{\osc}(\T^2,g) = \varnothing.
    \end{align}
    If $I = [0,T]$, we say $(\T^2,g)$ satisfies the noncoincidence condition up to time $T$.
\end{def1}

\begin{rema}
    The rationale for introducing ${\CL^{\osc}} (\T^2,g)$, ${\CL^{\rot}} (\T^2, g)$ and the noncoincidence condition~\eqref{eq:INCC} is the following. We will see in Proposition~\ref{prop:geom RT} below that for a Liouville metric on $\T^2$, the invariant tori whose foliating orbits belong to the homology class $(m,n)$ with $mn \neq 0$ project locally diffeomorphically onto the base. We call these \textbf{graph (rational) tori}. By contrast, periodic orbits in the homology class $(m,n)$ with $mn = 0$ might either belong to the separatrix region (where the smooth foliation breaks down) or to a {invariant} torus whose projection onto the base drops in rank at some point, corresponding to conjugate points of the underlying geodesics. We call these \textbf{non-graph (rational) tori}. In order to construct a microlocal parametrix for the wave propagator in Section~\ref{sec:trace} below, we consider only regions of phase space along which waves evolve along geodesics without conjugate points. The noncoincidence condition~\eqref{eq:INCC} then ensures that for each $t \in \CL^{\rot} (\T^2,g) \cap I$, no geodesic of length $t$ contains a pair of conjugate points.
\end{rema}

\begin{remark}
    If $g$ is a Liouville metric and $mn \neq 0$, then $\CL^{m,n}(\T^2,g)$ is a single point (see Proposition~\ref{prop:geom RT} below).
\end{remark}

\subsection{Rational integrability}

As a classical result (see, for example,~\cite{henheik2025deformational}), the geodesic flow of a Liouville metric on $\T^2$ is \emph{integrable} in the sense of Arnold-Liouville. 
While we do not need the general notion of Arnold-Liouville integrability in this paper, we shall introduce \emph{rational integrability}, a slightly different notion adapted to our particular situation.

\medskip

Roughly speaking, the geodesic flow on the unit tangent bundle $S\T^2$ is rationally integrable if every relevant closed orbit is contained in a \emph{rational torus}, namely, a family of closed geodesics that smoothly foliates an invariant torus in the unit tangent bundle and projects diffeomorphically onto the base manifold $\T^2$.
An analogous notion has been employed in works on the Birkhoff conjecture such as \cite{KaAvDS16,KaSo16,koval2023local}, where rational tori correspond to \emph{resonant convex caustics} in the case of billiards.

\medskip

Let us now formally define this notion in the setting of geodesic flows on $\T^2$.

\begin{def1}\label{def:RT}
    Let $g$ be a Riemannian metric on $\T^2$ and $s: \T^2 \to S\T^2$ be a continuous section of the unit tangent bundle of $(\T^2, g)$. 
    We call the image $\CT = \im(s)$ a \textbf{rational torus} of $(\T^2, g)$ if $\CT$ is invariant under the geodesic flow $\Phi^t_g$ and there exists $T > 0$ such that every point of $\CT$ is periodic under the geodesic flow with period $T$.
    We define the \textbf{length} of $\CT$ to be the minimal $T > 0$ for this to hold.
    We define the \textbf{homology class} of $\CT$ to be the homology class $[\gamma] \in H_1(\T^2, \Z)$ of any closed geodesic $\gamma:[0,T] \to \T^2$ which corresponds to a length-$T$ closed orbit in $\CT$.
\end{def1}

We emphasize that the definition above in particular requires a rational torus to be a \emph{graph} over $\T^2$. 
This graph property corresponds, in spirit, to the \emph{convexity} of a resonant caustic in billiards.

\medskip

We remark that since all closed geodesics of a fixed rational torus are homotopy equivalent as closed curves in $\T^2$, the homology class of a rational torus $\CT$ specified in Definition~\ref{def:RT} does not depend on the particular choice of the closed orbit in $\CT$.
For convenience, when a rational torus has homology class $(m,n)$, we shall call it an \emph{$(m,n)$-rational torus} or simply an \emph{$(m,n)$-torus}.
It will follow from Proposition~\ref{prop:geom RT} that if an $(m,n)$-rational torus exists then it is always unique.

\medskip

Although we do not \emph{a priori} assume any primitivity of the homology class of the rational torus in Definition~\ref{def:RT}, it will later follow from Proposition~\ref{prop:geom RT} (particularly the minimality property in part (1)) that the homology class of $\CT$ is always a \emph{primitive} element of $H_1(\T^2, \Z)$, i.e., it is not a positive integer multiple of a different element in $H_1(\T^2, \Z)$.
Under the identification $\Z^2 \simeq H_1(\T^2, \Z)$, the primitive elements $(m,n)$ are precisely the coprime pairs of integers.
In fact, if an $(m,n)$-rational torus exists, then all closed geodesics in the homology class $(km,kn)$ are iterates of geodesics of the $(m,n)$-rational torus.

\begin{def1}\label{def:RI}
    Let $T > 0$. 
    We say $(\T^2, g)$ is \textbf{rationally integrable up to $T$} if every closed geodesic which has length $\leq T$ and belongs to a homology class $(m,n) \in \Z^2$ for some coprime pair of integers $(m,n)$ with $mn \neq 0$ must belong to a rational torus of $(\T^2, g)$.
    We say $(\T^2, g)$ is \textbf{rationally integrable} if it is rationally integrable up to $T$ for all $T < \infty$.
\end{def1}

\begin{remark}
    Definition~\ref{def:RI} deliberately excludes closed geodesics in the homology classes $(\pm 1,0)$ and $(0,\pm 1)$. 
    This is because, for a typical Liouville metric, the invariant tori in these homology classes do not project diffeomorphically onto $\T^2$.
    In our later treatment of the deformations of Liouville metrics, it will be convenient to consider these geodesics separately from those in the rational tori in the sense of Definition~\ref{def:RT}.
\end{remark}

\subsection{Geometric properties of the rational torus}

In this section, we collect some properties of rational tori, which will be used in the proof of Theorem~\ref{thm:isospec implies RI}.

\medskip

We point out that the results in this section do not require the Riemannian metric to be a Liouville metric.
In particular, these results apply to \emph{any} Riemannian metric on $\T^2$ admitting a rational torus in the sense of Definition~\ref{def:RT}.

\medskip

On the other hand, the proof of the main result in this section, namely Proposition~\ref{prop:geom RT} below, relies crucially on the dimension of the torus being exactly $2$.
In fact, we do not know if our results can be generalized to $\T^d$ with $d \geq 3$.

\medskip

Proposition~\ref{prop:geom RT} is motivated by several well-known properties of convex caustics in billiards and, more generally, the invariant curves of symplectic twist maps.
For readers' convenience, we recall the following properties for \emph{essential} invariant curves of an exact symplectic twist map.\footnote{We say an invariant curve is \emph{essential} if it is not homotopically trivial, i.e., not a contractible curve in the phase cylinder.}
\begin{enumerate}
    \item all essential invariant curves are Lipschitz graphs over $\T$;
    \item distinct essential invariant curves do not intersect and the set of all such curves is totally ordered by their \emph{rotation numbers};
    \item if an essential invariant curve consists of periodic orbits with period $q > 0$, then the $q$-th iterate of the map has the twist property on the invariant curve and, moreover, the curve is as regular as the map itself.
\end{enumerate}
Statement (1) above is a classical theorem due to Birkhoff. 
Statements (2) and (3) are rather elementary.
See for example~\cite[Lemma 2.1]{PINTO-DE-CARVALHO_RAMIREZ-ROS_2013} for a result in the context of billiard maps.

\medskip

In the higher-dimensional setting, Arnaud-Massetti-Sorrentino~\cite{ArnaudMassettiSorrentino2023} considered \emph{Lagrangian periodic graphs} for symplectic twist maps on $\T^d \times \R^d$ for all $d \geq 2$ and, as intermediate results, established properties analogous to the invariant curves listed above.

\medskip

The statement of Proposition~\ref{prop:geom RT} below mirrors~\cite[Section 2]{ArnaudMassettiSorrentino2023}, but these two results deal with slightly different scenarios.
In particular, the cited work assumes that the symplectic map satisfies a \emph{global twist} condition, while, in our setup, the global twist condition does not hold unless the metric is flat.
As mentioned above, it is also unclear if Proposition~\ref{prop:geom RT} can be generalized to higher dimensions.

\medskip

To state the proposition, let us recall that $\Phi^t_g$ denotes the geodesic flow on the tangent bundle $T\T^2$ associated with $(\T^2, g)$, which naturally restricts to the unit tangent bundle $S\T^2$ with respect to the given metric $g$.

\begin{proposition}\label{prop:geom RT}
    Suppose $\CT = \im(s)$ is a rational torus of $(\T^2, g)$ corresponding to a $C^0$ section $s: \T^2 \to S\T^2$ as in Definition~\ref{def:RT}.
    Then the following properties hold.
    \begin{enumerate}
        \item $\CT$ contains exactly all unit-speed closed geodesics in its homology class. 
        In particular, every closed geodesic in $\CT$ is length minimizing within that homology class.
        \item The section $s$ is Lipschitz; moreover, its Lipschitz constant depends upper semicontinuously on $g$ with respect to the $C^2$ topology.
        \item (``Twist property'') Let $T > 0$ be the common length of orbits in $\CT$. For every $x \in \T^2$, the derivative at $u=0$ of the map
        \begin{equation}\label{eq:twist}
            \begin{aligned}
                T_x\T^2 &\longrightarrow \T^2,\\
                u &\longmapsto \pi\circ \Phi_g^T\big(s(x)+u\big)
            \end{aligned}
        \end{equation}
        has full rank (i.e., rank $2$), where $T_x \T^2$ denotes the tangent space of $\T^2$ at $x$ and $\pi$ denotes the canonical projection $T\T^2 \to \T^2$.
    \end{enumerate}
\end{proposition}

\begin{proof}
    Lift $g$ to the universal cover $\R^2$; we keep the same letter $g$ for the lifted $\Z^2$-periodic metric.
    Then $s$ lifts to a $\Z^2$-periodic section of $S\R^2$.
    For $x\in \R^2$, let $\gamma_x:\R\to \R^2$ denote the (unit-speed) geodesic with $\gamma_x(0)=x$ and $\gamma_x'(0)=s(x)$.
    By the definition of $\CT$, the projection of $\gamma_x$ to $\T^2$ belongs to the rational torus $\CT$, and there exist $T > 0$ and $(m,n) \in \Z^2$ such that $\gamma_x(T) = x + (m,n)$ for all $x \in \R^2$; the class $(m,n)$ is the homology class of $\CT$.

\medskip

    \noindent\underline{\textit{Proof of (1).}} 
    We claim the following: 
    \begin{equation}\label{eq:intersection lemma}
        \text{If a geodesic $\gamma$ of $(\R^2,g)$ intersects $\gamma_x$ at two distinct points, then $\im(\gamma)=\im(\gamma_x)$.}
    \end{equation}
    This statement
    \footnote{We emphasize that $\gamma$ is said to \textit{intersect} $\gamma_x$ as long as $\im(\gamma) \cap \im(\gamma_x) \neq \varnothing$, even though $\gamma$ and $\gamma_x$ may not reach the point of intersection at the same time.}
    will imply part (1) of the proposition because, if $\gamma:\R \to \R^2$ lifts a geodesic of $(\T^2, g)$ of the same homology class as $\CT$, then, for any $x \in \im(\gamma)$, the curve $\gamma$ intersects $\gamma_x$ at $x$.
    Since both $\im(\gamma)$ and $\im(\gamma_x)$ are invariant under the translation by $(m,n)$, they intersect infinitely many times; hence by \eqref{eq:intersection lemma} they coincide up to reparametrization, so the closed geodesic belongs to $\CT$. 

\medskip

    We now prove \eqref{eq:intersection lemma}. Let us suppose that a geodesic $\gamma$ intersects $\gamma_{x_0}$ more than once for some $x_0 \in \R^2$.
    If $\im(\gamma) \neq \im(\gamma_{x_0})$, then the two curves are not tangent at their points of intersection, and by dimensionality their intersections must be transverse.
    We may hence assume that two adjacent intersections occur at $\gamma(T_1)$ and $\gamma(T_2)$ for some $T_1 < T_2$. 
    By our choice, the segment $\gamma((T_1, T_2))$ does not intersect $\im(\gamma_{x_0})$.
    Consider the following subsets of $\R^2$.
    \begin{equation*}
        H_1 = \{x \in \R^2 \mid \gamma_{x} \text{ intersects } \gamma|_{(T_1, T_2)}\} \quad \text{and} \quad 
        H_2 = \{x \in \R^2 \mid \gamma_{x} \text{ intersects } \gamma|_{[T_1, T_2]}\}.
    \end{equation*}
    By continuous dependence of geodesics on initial data and stability of transverse intersections, $H_1$ is open and $H_2$ is closed; hence $\overline{H_1}\subset H_2$.
    Moreover, $H_2\setminus H_1=\im(\gamma_{x_0})$: indeed, if $x\in H_2\setminus H_1$, then $\gamma_x$ meets $\gamma([T_1, T_2])$ at an endpoint, and since $\gamma_x$ and $\gamma_{x_0}$ correspond to the same invariant torus $\CT$, they cannot intersect transversely; thus they are tangent and coincide up to reparametrization, so $\im(\gamma_x)=\im(\gamma_{x_0})$.
    Hence $H_1$ is both open and closed in $\R^2\setminus \im(\gamma_{x_0})$, so it must be a union of connected components of $\R^2\setminus \im(\gamma_{x_0})$.
    But in each of the two connected components of $\R^2\setminus \im(\gamma_{x_0})$ there exist points $x'$ such that the Riemannian distance between $\im(\gamma_{x'})$ and $\gamma(T_1)$ is greater than $T_2 - T_1$. 
    In particular, we have $\im(\gamma_{x'}) \cap \im(\gamma|_{(T_1, T_2)}) = \varnothing$, contradicting the definition of $H_1$.
    Therefore \eqref{eq:intersection lemma} holds, completing the first half of (1).

\medskip

    By what we have shown, the rational torus $\CT$ contains the minimal closed geodesic in the homology class of $\CT$.
    Since all closed geodesics in $\CT$ by definition have the common Riemannian length $T$, each is length minimizing within the homology class.
    The proof of part (1) is finished.

\medskip

    \noindent\underline{\textit{Proof of (2).}} 
    To prove that the section $s:\T^2 \to S\T^2$ of a rational torus is Lipschitz,
    the guiding heuristic is that the initial directions of geodesics in a rational torus whose starting points are close cannot differ much. 
    If they did, then under bounded curvature of the surface (uniform $C^2$-control of $g$), the geodesics would be forced to intersect, contradicting the fact that a rational torus is a graph over $\T^2$.
    The following proof formalizes this heuristic.
    \begin{figure}
        \centering
        \includegraphics[width=0.7\linewidth,page=1]{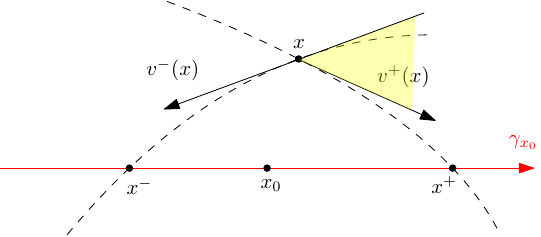}
        \caption{Proving the Lipschitz property of rational tori}
        \label{fig:lipschitz}
    \end{figure}
    Let us denote by $B(x_0, r)$ the geodesic ball centered at $x_0$ with radius $r > 0$.
    By standard results in Riemannian geometry, there exists $\varrho > 0$, depending lower-semicontinuously on $g$ in the $C^2$ topology, such that for all $x_0 \in \R^2$ and all $x, y \in \overline{B(x_0, \varrho)}$, there exists a unique minimizing unit-speed geodesic segment joining $x$ to $y$ and depending smoothly on $(x,y)$.\footnote{For example, it suffices to choose $\varrho < \frac{r_0}{2}$, where $r_0$ is the \textit{injectivity radius} of $(\T^2, g)$. 
    This is related to, but slightly different from the notion of the \emph{convexity radius} of a Riemannian manifold. 
    The latter is defined to be the maximum $\varrho$ such that every geodesic ball $B(x_0, \varrho)$ is \emph{geodesically convex}, i.e., for any pair $x, y \in B(x_0, \varrho)$ there is a unique length-minimizing geodesic segment joining $x$ and $y$ which is contained in $B(x_0, \varrho)$. 
    Our argument, however, does not need the geodesic to be contained in the ball.}

\medskip

    Let us fix an arbitrary $x_0 \in \R^2$ and fix two points $x^\pm := \gamma_{x_0}(\pm \varrho) \in \overline{B(x_0, \varrho)}$.
    Consider the following smooth map
    \begin{equation*}
        \overline{B(x_0, \varrho)} \ni x \mapsto (v^+(x), v^-(x)) \in \R^2 \times \R^2
    \end{equation*}
    where $v^\pm(x)$ is the initial velocity of the unique unit-speed minimizing geodesic segment joining $x$ to $x^\pm$.
    See Figure~\ref{fig:lipschitz} for an illustration.
    Since the geodesic differential equation involves derivatives of $g$ only up to first order, it follows that if $g$ and $g'$ are sufficiently close in the $C^2$-topology (assuming both have the rational torus of the same homology class), then the corresponding functions $v^\pm$ are close in the $C^1$-topology.
    It is clear from the definition that $v^+(x_0) = -v^-(x_0)$.
    
\medskip

    We now claim that for all $x \in B(x_0, \varrho)$, the initial velocity $\gamma_x'(0) \in \R^2$ of the corresponding unit-speed geodesic lifted from $\CT$ lies outside the (positive) cone generated by the tangent vectors $v^+(x)$ and $v^-(x)$.
    Indeed, if this does not hold, then $\gamma_x'(0)$ points towards the interior of the geodesic triangle formed by minimizing geodesic segments with vertices $x, x^-$ and $x^+$.
    However, by uniqueness of minimizing geodesic segments, the geodesic $\gamma_x$ does not intersect again the boundary of this geodesic triangle.
    This implies that $\gamma_x(t)$ is ``trapped'' in the interior of this triangle for all future times $t > 0$, contradicting the periodicity.
    Similarly, by time reversal, the same argument shows that $\gamma_x'(0)$ lies outside the cone generated by $-v^+(x)$ and $-v^-(x)$.
    By continuity and the fact that $\gamma_{x_0}'(0) = v^+(x_0) = -v^-(x_0)$, we rule out the cone generated by $-v^+(x)$ and $v^-(x)$.
    Hence, we deduce that $\gamma_x'(0) \in \R^2$ always lies in the positive cone generated by $v^+(x)$ and $-v^-(x)$ (indicated by the yellow region in Figure~\ref{fig:lipschitz}).

\medskip

    Note that $v^\pm(x)$ depend jointly differentiably on both $x$ and $x_0$. 
    Using $v^\pm(x)$ as bounding directions, the cone they span varies smoothly with $x$ and $x_0$. 
    This provides uniform two-sided bounds on the components of $s(x)=\gamma_x'(0)$. 
    By $\Z^2$-periodicity (compactness), these local bounds yield a global Lipschitz estimate for the map 
    \[
    \R^2 \ni x \longmapsto s(x)=\gamma_x'(0) \in S\R^2
    \]
    with a Lipschitz constant depending upper semicontinuously on the $C^1$ norms of $v^\pm$.
    Since the $C^1$ norms of $v^\pm$ depend continuously on the metric $g$ in the $C^2$-topology, the Lipschitz constant of $s$ depends upper semicontinuously on $g$.
    The proof of part (2) is finished.

\medskip

    \noindent\underline{\textit{Proof of (3).}} 
    Before entering the detailed argument, let us informally explain the guiding heuristic of the proof.
    We vary the initial direction of a geodesic in a rational torus and let the perturbed geodesic (the red curve in Figure~\ref{fig:twist}) intersect a nearby closed geodesic of the rational torus (the blue curve). 
    By the observation in part (1), after this intersection the red geodesic and the original geodesic must lie on opposite sides of the blue one. 
    Using the Lipschitz regularity established in part (2), we then show that the red geodesic deviates sufficiently in the transverse direction. 
    In the limit, this implies that the derivative of the map~\eqref{eq:twist} at $u=0$ in the transverse direction is nonvanishing, establishing the ``twist property''.
    \begin{figure}
        \centering
        \includegraphics[width=\linewidth, page=2]{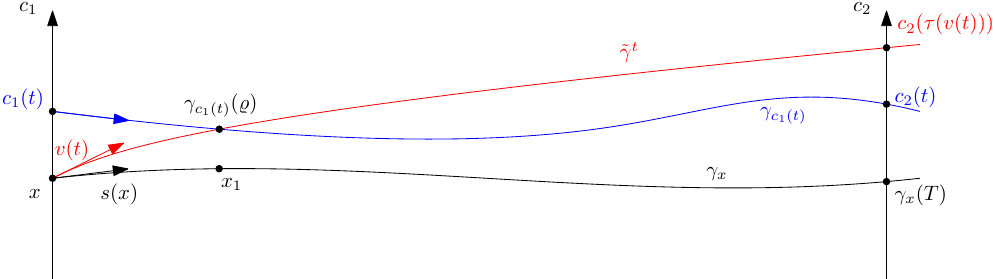}
        \caption{Proving the twist property of rational tori}
        \label{fig:twist}
    \end{figure}

\medskip

    We now present the details. Infinitesimally, varying $u$ in \eqref{eq:twist} corresponds to perturbing the initial direction $s(x)$ of the geodesic $\gamma_x$ while keeping the base point fixed. 
    Because $\gamma_x$ is periodic, the vector $s(x)$ is an eigenvector of the derivative of~\eqref{eq:twist} at $u=0$ with a strictly positive eigenvalue.
    Hence it suffices to study variations of $u$ transverse to $s(x)$.

\medskip

    Fix $x\in \R^2$ and let $c_1:\R\to \R^2$ be a $C^1$ curve intersecting $\gamma_x$ transversely at $c_1(0)=x$.
    Define $c_2(t):=c_1(t)+(m,n)$; then $c_2$ intersects $\gamma_x$ transversely at $c_2(0)=\gamma_x(T)$.
    For $v\in S_x\R^2$ sufficiently close to $s(x)$, the geodesic with initial velocity $v$ meets $c_2$ transversely at a unique point $c_2(\tau(v))$ with $\tau(v) \in \R$ near $0$; we must show that $v\mapsto \tau(v)$ has nonzero derivative at $v=s(x)$. Take $\varrho>0$ smaller than the injectivity radius and set $x_1=\gamma_x(\varrho)$.
    For $|t|$ small, let $\tilde\gamma^{t}$ be the unique unit-speed minimizing geodesic from $x$ to $\gamma_{c_1(t)}(\varrho)$, and write $v(t)\in T_x\R^2$ for its initial velocity.
    We recall that, by definition, the geodesic $\gamma_{c_1(t)}$ starts at the point $c_1(t)\in \R^2$ with the initial velocity $s(c_1(t))$ defined by the rational torus $\CT$, thus lifting a closed geodesic in $\CT$.
    By (2) and smooth dependence on endpoints, $t\mapsto v(t)$ is Lipschitz:
    \begin{equation}\label{eq:v Lip}
        \|v(t)-v(0)\|\leq L|t| \qquad (|t|\ \text{small}),
    \end{equation}
    for some $L<\infty$, where $\|\cdot\|$ is the norm induced by $g$ on $T_x\R^2$.
    For $|t|$ small, $v(t)$ stays close to $s(x)=v(0)$, so the geodesic with initial velocity $v(t)$ meets $c_2$ at $c_2(\tau(v(t)))$ with $\tau(v(t))$ near $0$.
    The curves $\gamma_{c_1(t)}$ and $\tilde\gamma^{t}$ meet at $\gamma_{c_1(t)}(\varrho)$ and, by part (1), do not meet again; this forces $\tau(v(t))$ and $t$ to have the same sign and $|\tau(v(t))|>|t|$.
    Using \eqref{eq:v Lip},
    \begin{equation*}
        \frac{|\tau(v(t))|}{\|v(t) - v(0)\|} > \frac{|t|}{\|v(t) - v(0)\|} \geq \frac{1}{L} > 0.
    \end{equation*}
    It is easy to show using the intermediate value theorem that $t \mapsto v(t)$ is surjective onto an open neighborhood of $s(x) = v(0)$ in $S_{x}\R^2$.
    The above inequality then implies 
    \begin{equation*}
        |\tau(v)| > \frac{1}{L} \|v - s(x)\|
    \end{equation*}
    for $v$ sufficiently close to $s(x)$.
    This in turn shows that $S_{x}\R^2 \ni v \mapsto \tau(v)$ has nonzero derivative at $s(x)$.
    The proof of (3) is finished.
\end{proof}

\subsection{Loop functions}

Let us apply Proposition~\ref{prop:geom RT} part (3) to study what happens to the rational torus under a perturbation of the metric.
We show that, although in general the rational torus may not survive the perturbation, there nevertheless exists a family of ``connecting orbits'' joining any pair of sufficiently close points $x$ and $y$ in $\T^2$, and that these orbits depend smoothly on the endpoints $(x,y)$ and on the deformation parameter.
In particular, for every $x \in \T^2$, there exists a unit-speed geodesic segment with respect to the perturbed metric joining $x$ back to itself and having the same homology class as $\CT$.
In general, this geodesic segment is \emph{not} periodic, in contrast to the case of a rational torus.

\begin{coro}\label{cor:loop fn}
    Let $(g_\eps)_{|\eps| \leq \eps_0}$ be a smooth family of Riemannian metrics on $\T^2$ such that $(\T^2, g_0)$ has a rational torus $\CT$.
    Then there exists a neighborhood $\CU$ of the diagonal of $\T^2 \times \T^2$ such that, for all $\eps$ in a sufficiently small neighborhood of zero, there exist smooth functions $s_\eps: \CU \to S\T^2$ and $\psi^\eps: \CU \to \R_{>0}$ depending smoothly on $\eps$ and satisfying
    \begin{equation}\label{eq:loop fn}
        \pi \circ s_\eps(x,y) = x
        \et
        \pi\circ \Phi^{\psi^\eps(x, y)}_{g_\eps}(s_\eps(x, y)) = y
        \quad \text{for all } (x,y) \in \CU,
    \end{equation}
    with $\CT = \{s_0(x,x) \mid x \in \T^2\}$.
\end{coro}

\begin{proof}
    Let $T$ be the length of $\CT$. By definition of rational tori, the map 
    \begin{equation}
        x \mapsto \pi \circ \Phi^T_{g_0}(s(x))
    \end{equation}
    is the identity map on $\T^2$. 
    By Proposition~\ref{prop:geom RT} part (3), we may apply the implicit function theorem to obtain a unique tangent vector $u(\eps, x, y)$ in a neighborhood of $0$ in $T_x\T^2$, for $|\eps|$ sufficiently small and $y$ sufficiently near $x$, such that $\pi\circ \Phi^T_{g_\eps}(s(x) + u(\eps,x, y)) = y$.
    Moreover, the vector $u(\eps, x,y)$ depends smoothly on $\eps$, $x$ and $y$, and $u(0, x,x) = 0$. By rescaling the tangent vector $s(x) + u(\eps, x,y)$ to a vector $s_\eps(x,y)$ of unit length and defining $\psi^\eps(x,y)$ to be $T$ times the Riemannian norm of $s(x) + u(\eps, x,y)$, we obtain~\eqref{eq:loop fn} as desired.
    It is clear from the construction that $s(x) = s_0(x,x)$. It follows that $\CT = \{s_0(x,x) \mid x \in \T^2\}$.
\end{proof}

\begin{remark}
    The conclusion of Corollary~\ref{cor:loop fn} in particular implies that a rational torus of $(\T^2, g)$, if it exists, is automatically a smooth submanifold of $T\T^2$.
\end{remark}

\begin{def1}[Loop functions]\label{def:loop functions}
    Let $\CT$ be an $(m,n)$-rational torus and $\psi_{m,n}^\eps(x,y)$ the flow times in Corollary \ref{cor:loop fn}. We denote by $\psi_{m,n}(x)$ its restriction to the diagonal, which we call the \textbf{$(m,n)$-loop function}, and refer to $\psi_{m,n}(x,y)$ as its off-diagonal extension.
\end{def1}

Since geodesics are parametrized with unit speed, the flow times $\psi_{m,n}^\eps(x,y)$ coincide with the lengths of such orbits. Recall that a geodesic \textit{loop} is a geodesic whose start and endpoints in $\T^2$ agree, even if the corresponding orbit is not periodic, i.e., the initial and final tangent vectors may differ. Corollary~\ref{cor:loop fn} says that $(m,n)$-loops always exist when $mn \neq 0$ and $\psi_{m,n}(x)$ is the length of such a loop based at $x$.

\begin{rema}
    A simple calculation shows that $-\nabla_x \psi_{m,n}(x,y)$ and $\nabla_y \psi_{m,n}(x,y)$ are the positively oriented tangent vectors at $x$ and $y$ to the geodesic $\gamma$ as in Corollary~\ref{cor:loop fn}. It then follows that $x$ is a critical point of the loop function, i.e., $\nabla_x \psi_{m,n}(x) = 0$, if and only if $\gamma$ is a \textit{closed} geodesic through $x$ in the homology class $(m,n)$.
\end{rema}

\section{A microlocal trace formula and noncancellation: Proof of Theorem~\ref{thm:isospec implies RI}} \label{sec:trace}

In this section, we demonstrate noncancellation of the wave trace and prove Theorem~\ref{thm:isospec implies RI}. We show that the wave trace $\tr \cos t \sqrt{-\Delta}$ must be singular at the majority of the length spectrum, namely the rotational lengths under a noncoincidence condition. In other words, we show that there is an \textit{equality} at certain lengths in the Poisson relation. The exceptional lengths which may possibly disappear arise from rational tori which do not project diffeomorphically onto the base, corresponding to periodic orbits with nontrivial Morse indices (due to conjugate points). The possibility of cancellations in the wave trace was recently confirmed by Koval together with the second and fourth authors \cite{KKV}, \cite{KV24}.

\medskip

We first assume that we have a Laplace isospectral deformation $g_\eps$ of an integrable Riemannian metric $g_0$ and show that the part of the length spectrum which is associated to graph rational tori, whose corresponding orbits have no conjugate points (called the rotational length spectrum, cf. \eqref{osc Lsp}) is constant in $\eps$. The part of the length spectrum that is associated to \textit{nongraph} rational tori, whose projections onto $\T^2$ drop in rank at some point (called the oscillatory length spectrum) may change. However, the noncoincidence condition~\eqref{eq:INCC} ensures that the rotational and oscillatory length spectra do not overlap. A Marvizi-Melrose type trace formula (Theorem~\ref{MM style parametrix}) then shows that the wave trace, a Laplace spectral quantity, is singular at the rotational length spectrum. We do not impose any clean intersection type hypotheses on periodic orbits. Our trace formula is accompanied by an off-diagonal microlocal parametrix for the wave propagator (Theorem~\ref{parametrix}), in close connection with the parametrices constructed in \cite{Vig21} and \cite{Vig22}.

\begin{theo} \label{thm:laplacetolength}
    If $g_\eps$ is a Laplace isospectral deformation of a rationally integrable metric $g_0$ on $\T^2$, $|\eps| \leq \eps_0$, and the noncoincidence condition~\eqref{eq:INCC} holds on some bounded interval $I$, then there exists $\eps_1>0$ such that $g_\eps$ is also rotationally length isospectral on $I$ for $|\eps| < \eps_1$, i.e., $\CL^{\rot} (\T^2,g_\eps) \cap I$ does not depend on $\eps$.
\end{theo}

Theorem~\ref{thm:isospec implies RI} now follows easily from Theorem~\ref{thm:laplacetolength}, whose proof is given at the end of Section~\ref{subsec:laplacetolength} below.

\begin{proof}[Proof of Theorem~\ref{thm:isospec implies RI}]
    Since $g_0$ is required to satisfy the noncoincidence condition on some interval $I$ in Theorem~\ref{thm:isospec implies RI}, by Proposition \ref{prop:NCC gen}, there exists $\eps_0 > 0$ such that  $g_\eps$ satisfies the noncoincidence condition on $I$ for $|\eps| \leq \eps_0$. 
    It then follows from Theorem \ref{thm:laplacetolength} that $g_\eps$ is necessarily rotationally length isospectral on $I$ for sufficiently small $\eps$. Now let $\psi_{m,n}^\eps$ be the $(m,n)$-loop functions for $g_\eps$, as in Definition~\ref{def:loop functions}. For $(m,n)$-geodesics of the metric $g_0$ with length in the interior of the interval $I$, consider their deformations as $\eps$ increases. Since critical points of $\psi_{m,n}^\eps$ correspond to closed geodesics in the homology class $(m,n)$, it is clear that $\max \psi_{m,n}^\eps$ and $\min \psi_{m,n}^\eps$ belong to the length spectrum for all $\eps$, so they must in fact coincide, forcing $\psi_{m,n}^\eps$ to be constant. This means that for each $x \in \T^2$, there exists a unique periodic $g_\eps$-geodesic through $x$ which is homotopic to an $(m,n)$-geodesic loop. If $\psi_{m,n}^\eps(x,y)$ is the off-diagonal extension of $\psi_{m,n}^\eps(x)$ as in Definition~\ref{def:loop functions}, $\{(x,d_x \psi_{m,n}^\eps(x,y)|_{y = x} ): x \in \T^2\} \subset T^*\T^2$ parametrizes a one parameter family of periodic orbits for the geodesic flow which form a rational torus. This proves part (1) of Theorem~\ref{thm:isospec implies RI} and part (2) then follows by taking $I = [0,T]$.
\end{proof}

While Theorems~\ref{thm:main2} and~\ref{thm:laplacetolength} apply only to $\T^2$, we will construct our parametrix more generally on a compact Riemannian manifold $X$ of dimension $d$ and specialize to the cases $X = \T^2$ with $g_0$ integrable (or Liouville) as necessary. Recall from \eqref{eq:wave trace} that
\begin{align*}
    w(t) = \tr \cos t \sqrt{-\Delta} \in \mathscr{S}'(\R),
\end{align*}
interpreted in the sense of distributions. That $w(t)$ is well defined follows from the functional calculus. For any $\rho \in \mathscr{S}(\R)$ (the Schwartz space), we have
\begin{align}\label{eq:functional calculus}
    \rho(\sqrt{-\Delta}) = \frac{1}{2\pi} \int_{-\infty}^{\infty} e^{i t \sqrt{-\Delta}} \check{\rho}(t) dt,
\end{align}
where $\check{\rho}$ is the inverse Fourier transform of $\rho$. From Weyl's law on the asymptotic distribution of eigenvalues, it is clear that the left-hand side of \eqref{eq:functional calculus} is trace class. Taking the trace of both sides in \eqref{eq:functional calculus} defines the pairing of $\tr e^{i t \sqrt{-\Delta}}$ and $\check{\rho} \in \mathscr{S}(\R)$, making the former a tempered distribution by duality.

\medskip

The Poisson relation dictates that $\singsupp \,\tr \big(\cos t \sqrt{-\Delta}\big) \subset \pm \overline{\CL(X, g)} \cup \{0\}$. Our strategy will be to construct an approximation (parametrix) of $\cos {t \sqrt{-\Delta}}$ and integrate its Schwartz kernel over the diagonal; if $K(x,y)$ is the Schwartz kernel of an integral operator $K$ which is smoothing\footnote{$\cos t \sqrt{-\Delta}$ is not smoothing, but a regularization of it will be.}, has discrete spectrum and can be diagonalized with eigenvalues $\mu_j$ and eigenfunctions $\phi_j$, we have
\begin{align}
    K(x,y) = \sum_j \mu_j \phi_j(x) \overline{\phi_j}(y),
\end{align}
and hence
\begin{align*}
    \tr K = \int_X K(x,x) dx.
\end{align*}
By choosing a test function $\rho$ such that $\check{\rho}$ is a bump function which is compactly supported near some part of the length spectrum, we can detect the contribution of those lengths to the wave trace by studying the Schwartz kernel of the operator on the right-hand side of \eqref{eq:functional calculus} and integrating it over the diagonal in $X \times X$. To find a formula for $\cos {t \sqrt{-\Delta}}$, we will follow a Fourier integral operator variant of the WKB method.

\subsection{Wave propagators as FIOs}
	
It is well known, going back to the work of H\"ormander, that $e^{i t \sqrt{- \Delta}}$ is a Fourier integral operator (FIO) which quantizes the time $t$ geodesic flow. In particular, so are its real and imaginary parts. The propagators
\begin{align}\label{propagators}
    E(t) := \cos t \sqrt{- \Delta}, \qquad S(t) = \frac{\sin t \sqrt{-\Delta}}{\sqrt{- \Delta}}
\end{align}
are solution operators, defined by the functional calculus, for the wave equation. We will denote by $E(t,x,y)$ and $S(t,x,y)$ the Schwartz kernels of $E$ and $S$ respectively:
\begin{align*}
    \begin{cases}
        (\d_t^2 - \Delta) E(t,x,y) = 0,\\
        E(0,x,y) = \delta(x-y),\\
        \frac{\d}{\d t} E(0,x,y) = 0,
    \end{cases}
    \qquad 
    \begin{cases}
        (\d_t^2 - \Delta) S(t,x,y) = 0,\\
        S(0,x,y) = 0,\\
        \frac{\d}{\d t} S(0,x,y) = \delta(x-y).
    \end{cases}
\end{align*}
For $t \in \R$, $E(t)$ belongs to the class $I^{-1/4}(\R \times X, X; C)$ of Fourier integral operators from $\R \times X$ to $X$ with canonical relation $C = \cup_\pm C_\pm$, where
\begin{align}
\label{eq:CRLagrangians}
\begin{split}
    C_\pm =& \{((t, x, \tau, \xi), (y, - \eta)) \in T^*(\R \times X) \times T^*X: (t, x,y, \tau, \xi, \eta) \in \Lambda_\pm\},\\
    \Lambda_\pm =& \left \{ (t, x, y, \tau, \xi, \eta): t \in \R, \tau = \mp |\xi|, (x,\xi) = \Phi_H^t(y,\eta) \right\}.
\end{split}
\end{align}
In other words, $C_\pm$ are, up to a sign in one of the fiber variables, the graphs of the geodesic flow on $T^* X$ and are to be interpreted as corresponding to Lagrangian submanifolds $\Lambda_\pm$ in $T^*(\R \times X \times X)$. The Schwartz kernel of $E$ is a Lagrangian distribution associated to $\Lambda_\pm$ and the two branches $(\pm)$ correspond to time reversal. If $t = t_0 \in \R$ is instead considered as a fixed parameter, then $E(t_0) \in I^{0}(X,  X; C_{t_0})$, where
\begin{align*}
    C_{t_0} = \left\{(x, \xi), (y,\eta): \Phi_H^{t_0}(y, \eta) = (x,\xi)\right\}.
\end{align*}
From here on, we will assume that $t$ is considered as a separate variable so that $E$ has order $-1/4$. We now follow the WKB construction of the wave kernel confined to pairs of points $(x,y)$ whose distance is less than the injectivity radius.

\subsection{A Hadamard-Riesz parametrix for small time on $\wt X$}
Denote by $\wt X \xrightarrow{\pi} X$ the universal cover of a compact Riemannian manifold $(X, g)$, $\wt g = \pi^* g$ the corresponding lifted Riemannian metric on $\wt X$ and $\inj(\wt X, \wt g)$ its injectivity radius. In the case of Theorem \ref{thm:main2}, $X = \T^2, \wt X = \R^2$, and $\wt g$ is periodic with respect to some lattice isomorphic to $\Z^2$. In contrast to Section~\ref{sec:dynamics} where we used $\gamma$ to denote a geodesic curve, we will now use $\gamma \in \Gamma$ to denote an element of the fundamental group of $X$, which acts on $\wt X$ via deck transformations. The associated homology class will be denoted by $[\gamma] \in H_1(X, \Z)$ and we write $\psi_\gamma$ for the ``$\gamma$-loop function'' in place of the $(m,n)$-loop function in Definition~\ref{def:loop functions}. We follow the presentation in \cite{Berard77} to find the Schwartz kernel of the sine propagator $\wt S$ on $\wt X$, which will have the form
\begin{align}\label{kernel}
    \wt S(t,x,y) = \sum_\pm \int_0^\infty e^{i \phi_\pm(\theta, t,x,y)} a_\pm (t,x,y,\theta) \dif \theta, \quad x,y \in \wt X, t \in \R,
\end{align}
for some phase functions $\phi_\pm$ and amplitudes $a_{\pm}$. Denote by $\Box$ the wave operator $\d_t^2 - \Delta_{\wt g}$. Rather than demanding $\Box \wt S = 0$, we will accept errors of the form $\Box \wt S \in C^\infty$ as long as they satisfy the correct initial condition $\wt S(0) = 0$ and $\d_t \wt S(0) = \Id$. The reason for this is that the true propagator, which is uniquely defined, differs from such a parametrix by a smoothing operator with $C^\infty$ kernel and smooth trace, hence playing a secondary role.

\medskip

In fact, we may choose $\phi_\pm$ to be homogeneous of degree $1$ in the fiber variable $\theta$ and $a_\pm$ to be classical symbols in the Kohn-Nirenberg class which have the form
\begin{align*}
    a_\pm(t,x,y, \theta) \sim \theta^{\frac{d-3}{2}} \sum_{k = 0}^{\infty} W_{k, \pm}(x,y) \theta_+^{-k}.
\end{align*}
From here on, we will use notation $b(\theta; z) \sim \sum_{k = 0}^\infty b_k(z) \theta^{-k}$ for an asymptotic expansion to indicate that for all $N \geq 1$, $\exists C_N > 0$ and $\theta_0(N) > 0$ such that, uniformly in all other parameters $z$,
\begin{align*}
    \left|b(\theta, z) - \sum_{k = 0}^N b_k(z) \theta^{-k} \right| \leq C_N \langle \theta \rangle^{-N - 1}, \quad |\theta| \geq \theta_0(N).
\end{align*}
Formally $\Box \wt S \in C^\infty$ is equivalent to
\begin{align*}
    \Box e^{i \phi_\pm(\theta, t,x,y)} \sum_{k = 0}^\infty W_{k, \pm}(x,y) \theta_+^{-k} \sim 0.
\end{align*}
Integrating by parts, one can see that smoothly cutting off $a_\pm(t, x,y,\theta)$ away from $\theta = 0$ modifies the parametrix by a $C^\infty$ kernel, so without loss of generality, we may assume from time to time that $W_{k, \pm} \theta_+^{-k}$ is replaced by $\chi(\theta) W_{k, \pm} \theta_+^{-k}$, where $\chi \in C^\infty(\R_+)$ is zero for $\theta \leq 1$, $1$ for $\theta \geq 2$, and monotonically nondecreasing. Plugging in \eqref{kernel} to the wave equation $\Box \wt S = 0$, we see that to eliminate the top degree terms, $\phi$ must satisfy the Hamilton-Jacobi equation
\begin{align}\label{eq: HJE}
    |\d_t \phi|^2 = |\nabla_x \phi|_g^2.
\end{align}
Two solutions \footnote{There are many choices for a phase function. For example, $\pm \theta(t - \psi(x,y))$ is used in \cite{MM}, \cite{Vig22}, and \cite{HeZe19}, while $\pm \psi^2(x,y)/4t$ is used in \cite{CdVParametrixOndesIntegraleslEspaceChemins}. See also \cite{GuilleminSternberg-SCA} Chapter 5.4.} are then given by $\pm \theta(\psi(x,y) - t)$, where for $x$ and $y$ close to the diagonal, $\psi(x,y)$ is the length of the unique unit-speed geodesic connecting $x$ to $y$. Any expression of this form satisfies
\begin{align*}
    \Box \left(e^{\pm i \theta (t - \psi(x,y))} W_{j, \pm}(x,y) \theta_+^j\right) = O(\theta_+^{j + 2 - 1}),
\end{align*}
where $j \in \Z$ is arbitrary. However, it is actually more convenient to use the phase functions $\pm \theta(\psi^2(x,y) - t^2)$, which are smooth at the diagonal. The function $\psi^2(x,y) - t^2$ does \textit{not} solve the Hamilton-Jacobi equation~\eqref{eq: HJE}, but
\begin{align}\label{eq:toporder}
    \Box e^{i \theta (\psi^2 - t^2)} W_k \theta^{\frac{d-3}{2} - k} = 4 e^{i \theta (\psi^2 - t^2)}  W_k \theta^{\frac{d+1}{2} - k} \left(|\psi \nabla \psi|_g^2 - t^2 \right) + O(\theta^{\frac{d-1}{2} - k}).
\end{align}
It is clear that $|\nabla \psi|^2 = 1$ and the quadratic term vanishes on the light cone $\{|\psi|^2 = t^2\}$, so whenever $\psi(x,y) \neq \pm t$, termwise integration by parts yields
\begin{align}\label{eq:IBP light cone}
\begin{split}
    \int_0^\infty & e^{\pm i \theta(\psi^2(x,y) - t^2)} \sum_{k = 0}^{\infty} W_{k, \pm}(x,y) \theta_+^{\frac{d-3}{2}-k} d \theta\\
    &=\int_0^\infty \left(\frac{1}{\pm i (\psi^2 - t^2)} \frac{d}{d \theta}\right)^j e^{\pm i \theta(\psi^2(x,y) - t^2)} \sum_{k = 0}^{\infty} W_{k, \pm}(x,y) \theta_+^{\frac{d-3}{2}-k} d \theta\\
    & = \int_0^\infty e^{\pm i \theta (\psi^2(x,y) - t^2)} \sum_{k = 0}^{\infty} \frac{\left(\frac{d-3}{2} - k\right) \cdots \left(\frac{d-3}{2} - k - j + 1\right)}{(\pm i)^j (\psi^2 - t^2)^j} W_{k, \pm}(x,y) \theta_+^{\frac{d-3}{2}-k -j } d \theta.
\end{split}
\end{align}
Thus, the integrand in \eqref{eq:IBP light cone} is $O(\theta^{\frac{d- 3}{2}})$ for $(t,x,y)$ on the light cone and $O(\theta^{-\infty})$ off of it.

\medskip

To solve away the $\theta^{\frac{d-1}{2}}$ coefficients ($k = 0$) and higher order terms in \eqref{eq:toporder}, we need to find appropriate amplitudes $W_{k,\pm}$. To describe them, we introduce the following notation. Choose a point $x$ and introduce geodesic normal coordinates $y = \exp_x(v(y))$ based at $x$. We define the function $\Theta(x,y) = \det D_{\exp_x^{-1}(y)} \exp_x$. More specifically, if $v \in T_x \wt X$ is such that $\exp_x(v) = y$, then $\Theta(x,y) = \det D_v \exp_x(v)$. An easy calculation in geodesic normal coordinates shows that for each $x$, $\Theta(x,y) = \sqrt{|\det(g_{ij}(y))|}$ is the density of the Riemannian volume form at $y$. Finally, solving away the first order terms in~\eqref{eq:toporder}, we arrive at the well-known transport equations:
\begin{align}\label{eq:transport}
    \begin{split}
    \frac{\langle \nabla_y \psi, \nabla_y \Theta \rangle}{2\Theta}\, W_0 + \langle \nabla_y \psi, \nabla_y W_0 \rangle &= 0, \\
    4 i \psi \left\{ 
    \left( \frac{k+1}{\psi} + \frac{\langle \nabla_y \psi, \nabla_y \Theta \rangle}{2\Theta} \right) W_{k+1}
    + \langle \nabla_y \psi, \nabla_y W_{k+1} \rangle
    \right\} &= \Delta_y W_k, \qquad k \ge 0,
    \end{split}
\end{align}
(see \cite{BGM}). We introduce the normalizations
\begin{align}
\label{eq:normalized Wk}
    W_{k,\pm}(x,y) = C_0 \exp \left( \mp i \pi \frac{d-1 + 2k}{4}  \right) 4^{-k} w_k(x,y),
\end{align}
where
\begin{align}\label{eq:integratedw}
    \begin{split}
        w_0(x,y) &= \Theta^{-\frac{1}{2}}(x,y),\\
        w_{k+1}(x,y) &= \int_0^1 s^k \Theta^{\frac{1}{2}}(x,x(s)) \Delta_2 w_k(x,x(s)) ds,
    \end{split}
\end{align}
with $x(s)$ being the unique length minimizing geodesic connecting $x$ to $y$, parametrized proportionally to arclength, the constant of proportionality being $\psi(x,y)$. $C_0$ is a universal \textit{real valued} constant and $\Delta_2$ is the Laplacian in the second spatial variable. That $W_k$ can be chosen to be time independent is a reflection of the propagation of singularities, which tells us that $\wt S (t,x,y)$ is singular only along the wavefronts $t = \pm \psi(x,y)$.	
With these coefficients, one sees that
\begin{align*}
    \Box_{t,y} \left(\theta_+^{\frac{d-3}{2}} e^{\pm i \theta (\psi^2(x,y) - t^2)}\sum_{k = 0}^N W_{k, \pm}(x,y) \theta_+^{-k} \right) \in S^{\frac{d-1}{2} - N}(\R \times \wt X \times \wt X).
\end{align*}

Since the true sine kernel satisfies $\wt S(t,x,y) = - \wt S(-t,x,y)$, we extend our construction to be odd in time for $t<0$. To evaluate the integral
\begin{align*}
    \wt S_N(t,x,y) := \sum_{\pm} \mp  \sgn (t) \frac{1}{2i} \int_0^\infty e^{\pm i \theta (\psi^2(x,y)  - t^2)} \theta_+^{\frac{d-3}{2}}  \sum_{k = 0}^N W_{k,\pm}(x,y) \theta_+^{-k} d\theta,
\end{align*}
we use the following formula for the Fourier transform of homogeneous distributions:
\begin{align*}
    \int_0^\infty e^{\pm i \theta s} \theta_+^a d\theta = \Gamma(a + 1) e^{\pm i\pi (a+1)/2} (s \pm i 0)^{-a - 1}
\end{align*}
(see \cite{Ho90}, Sections $3.2$ and $7.1$, for the regularization of the Riesz distributions $\theta_+^a$ together with their Fourier transforms). For $t \neq 0$, the map $\sigma: (t,x,y)  \mapsto \psi^2(x,y) - t^2$ is a submersion, so we may pullback homogeneous distributions $(s \pm i 0)^{a}$ by $\sigma$ for $s \in \R \backslash \{0\}$, $a \in \C\backslash \Z_{< 0}$ and $t \neq 0$. In this case, we have
\begin{align}\label{FT}
    \begin{split}
        \int_0^\infty& e^{\pm i \theta (\psi^2(x,y)  - t^2)} \theta_+^{\frac{d-3}{2}}  \sum_{k = 0}^N W_{k, \pm}(x,y) \theta_+^{-k} d\theta
        \\
        & = C_0 \sum_{k = 0}^N \Gamma\left( \frac{d-1}{2} - k\right) w_{k}(x,y) (-1)^k 4^{-k} \left(\psi^2(x,y)  - t^2 \pm i 0 \right)^{k - \frac{d-1}{2}}.
    \end{split}
\end{align}
Using formulas for the homogeneous distributions $(\psi^2(x,y) - t^2 \pm i 0)^{k - \frac{d-1}{2}}$ in \cite{Ho90}, we obtain
\begin{align*}
    \sum_{k =0}^N \Gamma\left( \frac{d-1}{2} - k\right) \sin \left(\frac{d-1}{2} \pi - k \pi \right) w_k(x,y) (-4)^{-k} (\psi^2(x,y)  - t^2)_-^{k - \frac{d-1}{2}},
\end{align*}
for $d$ even. This corroborates the even dimensional phenomenon whereby waves propagate within the light cone; equivalently, the wave kernel is nonzero only inside $\{|t| \geq \psi(x,y)\}$. We also define
\begin{align*}
    \wt E_{N,\pm}(t,x,y) := \frac{\d}{\d t} \wt S_{N, \pm}(t,x,y)
    =  |t| \sum_{k = 0}^N \int_0^\infty e^{\pm i \theta (\psi^2(x,y)  - t^2)} W_{k,\pm}(x,y) \theta_+^{\frac{d -1}{2} - k} d\theta
\end{align*}
and set
\begin{align}\label{ENplusminus}
    \wt E_N(t,x,y) &= \sum_{\pm} \wt E_{N,\pm} (t,x,y).
\end{align}
Just as we expect $\wt S_N$ to be an approximation for the sine kernel, $\wt E_{N,\pm}$ will be microlocal parametrices for $\cos t \sqrt{- \Delta}$ corresponding to the two modes of propagation $\pm \tau > 0$, where $\tau$ is the symplectic dual variable to $t$.

\medskip

Since we are mainly interested in the case $d = 2$, we now restrict ourselves to even dimensions. In this case, the analytic continuation of homogeneous distributions is much simpler than in odd dimensions.

\begin{theo}[\cite{Berard77}] \label{Berard}
    For $N \geq \lceil d/2 \rceil$, $t \neq 0$, and $\psi(x,y) < \inj(\wt X, \wt g)$, we have
    \begin{align*}
        \Box \wt E_N \in C^{\lfloor N - \frac{d+3}{2}\rfloor}((0, \inj(\wt X,\wt g)) \times \wt X \times \wt X),
    \end{align*}
    and there exists a universal constant $C_0 > 0$, depending only on the dimension $d$, such that for any $x_0 \in \wt X$ and $h \in C_c^\infty(B(x_0,\inj(\wt X, \wt g)))$\footnote{Note that we cannot evaluate $\wt E_N(t)$ at $t= 0$ directly since $\sigma(x,y,t) = \psi^2(x,y) - t^2$ fails to be a submersion there.},
    \begin{align}
        \begin{split}
            \lim_{\substack{t \to 0^+ \\ t \neq 0}} &\int_{\wt X} \wt E_N(t,x,y) h(y) dy = h(x),\\
            \lim_{\substack{t \to 0^+ \\ t \neq 0}} &\int_{\wt X} \frac{\partial \wt E_N}{\partial t}(t,x,y) h(y)dy = 0.
        \end{split}
    \end{align}
\end{theo}

That the $d \theta$ integral can be evaluated explicitly is due to the fact that the Lagrangian manifolds $\Lambda_\pm$ which are associated to the canonical relations $C_\pm$ are in fact horizontal over $\R \times X \times X$ for $|t|$ less than the injectivity radius; i.e., they project locally diffeomorphically onto $\R \times \wt X \times \wt X$. In general, the representation of $E$ as an oscillatory integral is valid even up to times larger than the injectivity radius, albeit with a different phase function. We prefer to use the integrated version of Theorem~\ref{Berard} in order to obtain stationary phase expansions in Section~\ref{sec: wave trace for deformations of an integral metric}.

\medskip

To show that the parametrices \eqref{ENplusminus} approximate the true wave propagator up to a sufficiently smooth remainder, we deviate from \cite{Berard77}, following more closely the energy estimates for solutions to hyperbolic equations as presented in \cite{Ho90}. The reason for this is that these energy estimates work for the off-diagonal kernel, whereas Berard considers a regularization of the pullback by the diagonal embedding at $t = 0$. We are primarily concerned with the singularities of $w(t)$ away from $t=0$, so an off-diagonal extension for positive time is important.

\begin{lemm}[\cite{Ho385}, Lemma 17.5.4]\label{Energy}
    Let $Y$ be a compact manifold with boundary and $P$ be an elliptic, self-adjoint second order differential operator with smooth coefficients. Denote by  \( v \in C^\infty([0,T] \times \overline{Y}) \) a solution of the mixed problem
    \begin{align} \label{eq: Hormander energy estimates}
        \begin{cases}
            (P + \partial^2 / \partial t^2)v = g & \qquad\text{in } [0,T] \times \overline{Y}; \\
            v = 0 & \qquad\text{in } [0,T] \times \partial Y; \\
            v = \partial v / \partial t = 0 & \qquad \text{if } t = 0.
        \end{cases}	
    \end{align}
    Assume that \( \partial^j g / \partial t^j = 0 \) when \( t = 0 \) if \( j < k \). Then it follows that
    \begin{align*}
        \sum_{j = 0}^{k+1} \| D_t^{k+1-j} v(t, \cdot) \|_{H^j(Y)} \leq C_k \left( \int_{0}^{t} \| D_s^k g(s, \cdot) \|_{L^2(Y)} ds + \sum_{j=0}^{k-1} \| D_t^{k-1-j} g(t, \cdot) \|_{H^j(Y)} \right).
    \end{align*}
\end{lemm}

Using the above energy estimates, we may now show that our parametrix $E_N(t)$ is indeed a good approximation of the true wave propagator $E(t)$.
\begin{prop}\label{Cr approximation}
    Denote by $\wt E_N$ the oscillatory integral in \eqref{ENplusminus} obtained by summing the first $N$ solutions to the transport equations. Then for each $N \geq d + 2$, we have
    \begin{align*}
        \wt E_N(t,x,y) - \wt E(t,x,y) \in C^{N - d - 2}\big((0, \inj(\wt X, \wt g) )\times \wt X \times \wt X\big).
    \end{align*}
    In particular, $\text{diag}^*\left(\wt E_N(t,x,y) - \wt E(t,x,y) \right) \in C^{N - d - 2}((0, \inj(\wt X, \wt g) ) \times \wt X)$, where $\text{diag}: \wt X \to \wt X \times \wt X$, $x \mapsto (x,x)$ is the diagonal embedding.
\end{prop}
		
\begin{proof}
    Let $T < \inj (\wt X)$ be arbitrary. Choose any function $f \in C^\infty(\wt X)$ which is compactly supported in a geodesic ball of radius strictly less than the $\inj (\wt X)$ and set
    \begin{align*}
        v_N(t,x) = (\wt E(t) - \wt E_N(t)) f \in C^\infty((0, \inj(\wt X) ) \times \wt X).
    \end{align*}
    We will break up our proof into several steps.
    
    \medskip
    
    \noindent \textbf{Step 1.} We claim that $g_N:= - \Box \wt E_N f$ vanishes to order $2N$ at $t = 0$. In view of Theorem \ref{Berard}, it is clear that $v_N(0,x) = \frac{\d v_N}{\d t}(0,x) = 0$ and
    \begin{align*}
        \Box v_N = - \Box \wt E_N f := g_N.
    \end{align*}
    It is also shown in \cite{Berard77} (pg. 258-259) that for $t \neq 0$ and $\psi(x,y) < \inj(\wt X)$, 
    \begin{align}
        \Box {\wt E}_{N} (t,x,y) = C_0 (-1)^N \Delta_y w_{N} (x,y)|t|
        \frac{\left( \psi^2(x,y) - t^2 \right)_{-}^{N - d/2 - 1/2}}
        {4^N \Gamma \left( N - \tfrac{d}{2} - \frac{1}{2} \right)}.
        \label{eq: Box EN}
    \end{align}
    To find the order of vanishing of $g_N$ at $t=0$, we write
    \begin{align}
     \label{eq:dt box v}
        \Box v_N(t,x) = \int_{\wt X} V_N(x,y) |t| (\psi^2(x,y) - t^2)_{-}^{N - d/2 - 1/2} f(y) dy,
    \end{align}
    where
    \begin{align*}
        V_N(x,y) =
        \frac{C_0 (-1)^N \Delta_y w_{N} (x,y)}
        {4^N \, \Gamma\!\left( N - \tfrac{d}{2} - \tfrac{1}{2} \right)}
    \end{align*}
    is a smooth function on $\wt X \times \wt X$ which is well defined whenever $\psi(x,y) < \inj( \wt X)$. Introducing local geodesic coordinates $w = w(y)$ in place of $y$, i.e., $y = \exp_x w$, we can rewrite \eqref{eq:dt box v} as
    \begin{align*}
        \int_{T_x \wt X} (|w|_g^2 - t^2)_-^{N - d/2 - 1/2} V_N(x, \exp_x w) |t| f(\exp_x w) \Theta (x, \exp_x w) dw.
    \end{align*}
    Using homogeneity and a rescaling of coordinates $\wt w = t^{-1} w$, we may simplify the above integral to
    \begin{align*}
        |t|^{2N - d} \int_{T_x \wt X} (|\wt w|_{g}^2 - 1)_-^{N - d/2 - 1/2} V_N(x, \exp_x t \wt w) f(\exp_x t \wt w) \Theta (x, \exp_x t \wt w) |t|^d d \wt w.
    \end{align*}
    Pulling out the Jacobian scaling factor from $|t|^d dw$, we obtain $|t|^{2N}$ times an integral which depends smoothly on $x$ and $t$. It then follows that
    \begin{align*}
        \frac{\d^j}{\d t^j} \Box \wt E_N f \Big|_{t = 0} = 0, \qquad \text{ for }0 \leq j \leq 2N - 1.
    \end{align*}
    
    \medskip
    
    \noindent \textbf{Step 2.} We now show that $v_N$ itself vanishes to order $2N + 2$ at $t = 0$. We can write
    \begin{align*}
        \d_t^2 v_N = \Box v_N + \Delta v_N
    \end{align*}
    Theorem \ref{Berard} showed us that $v_N \big|_{t= 0} = \d_t v_N \big|_{t = 0} = 0$ and we saw in Step $1$ that $\Box v_N$ vanishes to order $2N - 1$ at $t = 0$, so we can bootstrap to see that
    \begin{align*}
        \frac{\d^{j+2}}{\d t^{j + 2}} v_N \Big|_{t = 0} = \left(\frac{\d^j}{\d t^j} \Box v_N + \Delta \frac{\d^j}{\d t^{j}} v_N \right) \Bigg|_{t = 0} = 0, \qquad \text{ for }0 \leq j \leq 2N - 1.
    \end{align*}
    
    \medskip
    
    \noindent \textbf{Step 3.} We now show that in the left-hand side of H\"ormander's energy estimate \eqref{eq: Hormander energy estimates} applied to $v_N$, we can replace the Sobolev $H^j$ norm by a $C^r$ norm with respect to both space and time for an appropriate $r$ depending on $j, d, N$ and the number of time derivatives $k + 1 - j$. Whenever $k \leq 2N$, Steps 1 and 2 show that the hypotheses of Lemma~\ref{Energy} with $P = - \Delta_{\wt g}$ are satisfied\footnote{$\wt X$ is of course not a compact manifold with boundary. However, by finite speed of propagation, both $E_N f$ and $E f$ are compactly supported whenever $f$ is, so for any fixed $T < \infty$, we can replace $\wt X$ by a sufficiently large ball $Y$ with $f, E_N f,$ and $Ef$ vanishing in a neighborhood of $\d Y$.} and
    \begin{align}\label{Sobolev remainder estimate}
        \begin{split}
            \sum_{j = 0}^{k + 1}& \| D_t^{k + 1 - j} (\wt E(t) - \wt E_N (t)) f \|_{H^j(\wt X)} \leq\\
            &C_N \bigg(\int_0^t \|D_s^{k} \Box \wt E_N (s) f(x)\|_{L^2(\wt X)} ds + \sum_{j = 0}^{k-1} \|D_t^{k - 1 - j} \Box \wt E_N(t) f(x) \|_{H^j(\wt X)} \bigg).
        \end{split}
    \end{align}
    We use the Sobolev embedding $H_{\text{loc}}^{j}(\wt X) \hookrightarrow C_{\text{loc}}^{j - \frac{d}{2} - 1}(\wt X)$ to replace the left-hand side of \eqref{Sobolev remainder estimate} with a $C^r$ norm (recall that for simplicity, we are assuming $d$ is even). Let $\alpha \in \N, \beta \in \N^d$ be multi-indices such that $\alpha + |\beta| \leq k - \frac{d}{2}$. It is clear that
    \begin{align*}
        \|D_t^\alpha D_x^\beta v_N\|_{C^0([0,T] \times \wt X)} \lesssim \| D_t^\alpha v_N\|_{C^{|\beta|}(\wt X)} \lesssim \|D_t^\alpha v_N\|_{H^{|\beta| + \frac{d}{2} + 1}(\wt X)}.
    \end{align*}
    We set $j = |\beta| + \frac{d}{2} + 1$ to be the Sobolev exponent on the right-hand side and choose $\alpha \leq k - \frac{d}{2} - |\beta|$. We see that
    \begin{align}
    \begin{split}
        \label{eq:Sobolev estimate}
        \|v_N\|_{C^{k - \frac{d}{2}}([0,T] \times \wt X)}
        = &
        \sum_{\substack{\alpha, \beta\\ \alpha + |\beta| \leq k - \frac{d}{2}}} \|D_t^\alpha D_x^\beta v_N\|_{C^0([0,T] \times \wt X)}
        \lesssim
        \sum_{\substack{\alpha, \beta\\ \alpha + |\beta| \leq k - \frac{d}{2}}} \|D_t^\alpha v_N\|_{C^0([0,T]; H^{|\beta| + \frac{d}{2} + 1} (\wt X))}
        \\
        \lesssim &
        \sup_{t \in [0,T]} \Bigg ( \sum_{j = \frac{d}{2} + 1}^{k + 1} \|D_t^{k  + 1 - j} v_N\|_{H^{j} (\wt X)}
        + \sum_{j = \frac{d}{2} + 1}^{k} \sum_{\alpha = 0}^{k - j} \|D_t^{\alpha} v_N\|_{H^{j} (\wt X)} \Bigg).
    \end{split}
    \end{align}
    The first sum matches terms in the left-hand side of \eqref{Sobolev remainder estimate}. As for the lower order derivatives in the second sum, we use the vanishing of initial data $D_t^\alpha v_N \big|_{t=0}$ for $\alpha \leq 2N + 1$ to see that
    \begin{align}
    \label{eq: FTOC Sobolev}
    \begin{split}
    \|D_t^\alpha v_N \|_{H^j(\wt X)} = & \Bigg \|\d_t^\alpha v_N \Big|_{t = 0} + \int_0^t \d_s^{\alpha + 1} v_N \,ds\Bigg \|_{H^j(\wt X)}
    =
    \Bigg \|\int_0^t \d_s^{\alpha + 1} v_N \,ds \Bigg \|_{H^j(\wt X)}
    \\ \leq &
    T \sup_{t \in [0,T]} \|\d_t^{\alpha + 1} v_N \|_{H^j(\wt X)}.
    \end{split}
    \end{align}
    If $k \leq 2N$ and $j \leq k$, then for each term $0 \leq \alpha \leq k - j$, we can iterate \eqref{eq: FTOC Sobolev} $k + 1 - j - |\alpha| \leq k + 1$ times to conclude that the second sum can be bounded as follows:
    \begin{align*}
    \sum_{j = \frac{d}{2} + 1}^{k} \sum_{\alpha = 0}^{k - j} \|\d_t^{\alpha} v_N\|_{H^{j} (\wt X)} \leq &
    \sum_{j = \frac{d}{2} + 1}^{k} \sum_{\alpha = 0}^{k - j} T^{k + 1 - j - \alpha} \|\d_t^{k  + 1 - j} v_N\|_{H^{j} (\wt X)}\\
    \leq &
    \left(k - \frac{d}{2}\right) \wt C (T) \sum_{j = \frac{d}{2} + 1}^{k} \|\d_t^{k + 1 - j} v_N\|_{H^{j} (\wt X)},
    \end{align*}
    for some positive constant $\wt C(T)$ which is uniformly bounded when $T < \inj(\wt X)$ (depending on whether $T$ is less than, greater than, or equal to one).
    
    \medskip
    
    \noindent \textbf{Step 4.} We now estimate the right-hand side of~\eqref{Sobolev remainder estimate}. Using Minkowski's inequality, we see that
    \begin{align}
    \label{eq: L2 estimate}
    \begin{split}
    \int_0^T & \left( \int_{\wt X} \left|D_t^{k} \int_{\wt X} \Box \wt E_N(t,x,y) f(y) dy\right|^2 dx \right)^{1/2} dt
    \\
    \leq &
    \int_0^T \left( \int_{\wt X}  \sup_{y \in \wt X} |D_t^{k} \Box \wt E_N(t,x,y)|^2 \|f\|_{L^1(\wt X)}^2 dx \right)^{1/2} dt
    \\
    \leq &
    \int_0^T \sup_{y \in \wt X} \|D_t^k \Box \wt E_N (t,x,y) \|_{L_x^2} \|f \|_{L^1(\wt X)} dt,
    \end{split}
    \end{align}
    and similarly,
    \begin{align}\label{eq: Hj estimate}
    \|D_t^{k - 1 - j} \Box \wt E_N(t)f \|_{H^j(\wt X)} \leq & \sup_{y \in \wt X} \|D_t^{k - 1 - j} \Box \wt E_N(t,x,y)\|_{H_x^j(\wt X)} \|f\|_{L^1(\wt X)}.
    \end{align}
    Both \eqref{eq: L2 estimate} and \eqref{eq: Hj estimate} contain terms of the form $D_t^\alpha D_x^\beta \Box \wt E_N(t,x,y)$, which can be estimated by introducing local geodesic coordinates $x = \exp_y z$ and rescaling as in Step 1, i.e., setting $\wt z = t^{-1} z(x,y)$. There are two quantities to keep track of: (1) the spatial regularity with respect to $x$ and (2) the resulting power of $t$. The leading order singularity occurs when all derivatives fall on $(\psi^2(x,y) - t^2)_-^{N - \frac{d}{2} - \frac{1}{2}}$:
    \begin{align*}
    D_t^\alpha D_x^\beta \Box \wt E_N(t,x,y) =& 
    G_{\alpha, \beta}(\frac{z}{t},y, t) |t|^{2N - d - |\alpha| - |\beta|}, \qquad x = \exp_y(z)
    \end{align*}
    where $G_{\alpha, \beta} (\wt z, y, t)$ is compactly supported in $\wt z$ and conormal along the codimension-$1$ hypersurface $|\wt z|_g = 1$, having the leading order singularity $(|\wt z|_{g}^2 - 1)_-^{N - \frac{d}{2} - \frac{1}{2} - |\alpha| - |\beta|}$ with respect to $\wt z$. The function $G_{\alpha, \beta}$ absorbs combinatorial constants, derivatives of $\psi^2(x,y)$ and derivatives of $V_N$. It is $L^2$-integrable with respect to the rescaled volume density $\Theta (\exp_y t \wt z , y) |t|^d |d \wt z|$ if and only if $|\alpha| + |\beta| < N - \frac{d}{2}$. Since $|\psi| \leq |t|$ on the support of $\Box \wt E_N$, the $L^2 (\wt X)$ norm is controlled by a constant times $|t|^{2N - \frac{d}{2} - |\alpha| - |\beta|}$.
    
    \medskip
    
    For the time integrated $L_x^2$ norm, we set $\alpha = k$ and $\beta = 0$. In that case, $\|D_t^k \Box \wt E_N(t,x,y)\|_{L_x^2(\wt X)}$ is finite provided that $k < N - \frac{d}{2}$ and is bounded by a constant times $t^{2N - \frac{d}{2} - k}$, which is clearly integrable with respect to $t$ on $[0,T]$ since $2N - \frac{d}{2} - k > N$. For the Sobolev terms, we set $\alpha = k - 1 - j$ and $|\beta| \leq j \leq k - 1$. In that case, we see that $|\alpha| + |\beta| \leq k - 1$, which is strictly less than $N - \frac{d}{2}$ whenever $k < N - \frac{d}{2} + 1$. As for the time dependence, the exponent of $t$ is at least $2N - \frac{d}{2} - k + 1 >0$ and hence remains bounded on $[0,T]$, provided that $k < 2N - \frac{d}{2} + 1$. To bound both the time integrated $L^2$ norm and Sobolev norms of $\Box \wt E_N(t) f$, we may choose $k = N - \frac{d}{2} - 1$.
    
    \medskip
    
    \noindent \textbf{Step 5.} Combining Steps 1 through 4, we obtain
    \begin{align*}
    \left \| \left(\wt E(t) - \wt E_N(t) \right) f \right \|_{C_{t,x}^{N - d - 1}([0,T] \times \wt X)} \leq C_{d, N}(T) \|f\|_{L^1(\wt X)},
    \end{align*}
    for another constant $C_{d,N}(T)$ which is uniformly bounded when $T < \inj(\wt X)$. In fact, we may replace $f$ by $D_y^\sigma f$ for any multi-index $\sigma \in \N^d$ whenever $|\sigma| \leq N - d - 1$, which, in view of integration by parts and symmetry of $\wt E$ and $\wt E_N$ with respect to $x$ and $y$, is equivalent to
    \begin{align}
    \label{eq: pointwise Sobolev bound}
    \left \| \int_{\wt X} f(y) D_y^\sigma \left(\wt E(t,x,y) - \wt E_N(t,x,y)\right) dy \right\|_{C_{t,x}^{N - d - 1 - |\sigma|} ([0,T] \times \wt X)} \leq C_{d,N}(T) \|f\|_{L^1(\wt X)}.
    \end{align}
    Since the right-hand side of \eqref{eq: pointwise Sobolev bound} depends only on the $L^1$ norm of $f$, we may replace the integrated estimate with a pointwise one as follows. Let $\phi \in C_0^\infty(B(0,1))$ be nonnegative and identically equal to $1$ in a neighborhood of $0 \in \R^d$, satisfying $\|\phi\|_{L^1(\R^d)} = 1$. Denote by $\phi_\eps(x) = \eps^{-d} \phi(x/\eps)$. For any $y_0 \in \wt X$, introduce geodesic normal coordinates $y = \exp_{y_0} u(y)$ based at $y_0$ and set $f_\eps(y) = \phi_\eps(u(y))$, where $u(y)$ is the preimage of $y$ under the exponential map. We then have
    \begin{align}
    \label{eq:spacetime Cr bound}
    \begin{split}
     \lim_{\eps \to 0} \langle D_y^{\sigma} (\wt E - \wt E_N) (t,x,y),  \phi_\eps(u(y)) \rangle_{L_y^2(\wt X)}
    = D_y^\sigma \left( \wt E - \wt E_N \right)(t,x,y_0),
    \end{split}
    \end{align}
    where the $L_y^2(\wt X)$ inner product is with respect to the volume density $\Theta(y_0, y(u))$ in normal coordinates at $y_0$ and convergence is in the sense of distributions in $(x,t)$. The left-hand side of \eqref{eq:spacetime Cr bound} is uniformly bounded in $C_{t,x}^{N - d - 1 - |\sigma|}([0,T] \times \wt X)$ and the constants $C_{d,N}(T)$ are clearly uniformly bounded in $\eps$. Taking the limit $\eps \to 0$ in \eqref{eq: pointwise Sobolev bound} and applying the Arzel\`a-Ascoli theorem, we see that the right-hand side of \eqref{eq:spacetime Cr bound} belongs to $C^{N - d - 2 - |\sigma|}([0,T] \times \wt X)$, which concludes the proof of the proposition.
\end{proof}

For $t \geq \inj(\wt X, \wt g)$, $\wt E(t,\wt x, \wt y)$ can be immensely complicated. However, we can microlocalize the Hadamard-Riesz parametrix in Theorem~\ref{Berard} away from places where the canonical relation becomes vertical over $\wt X \times \wt X$. This is done by choosing smooth branches of the distance function on $\wt X$ which generate $\Lambda_\pm$, together with corresponding FIO amplitudes (see Theorem \ref{parametrix} below).

\subsection{Microlocal trace formula on the quotient near horizontal geodesic loops}

From now on, we denote by $X = \T^2 = \R^2/ \Z^2$ a torus. We will call its homotopy group $\Gamma = \Z^2 \cong H_1(X, \Z)$ and associate to each free homotopy class $\gamma \in \Gamma$ a corresponding homology class $[\gamma]$ indexed by $(m_\gamma,n_\gamma) \in \Z^2$. For a function $f \in C^\infty(\wt X)$, we will denote by $\gamma^* f(x)$ the function $f(\gamma x) = f(x + (m_\gamma, n_\gamma))$. To obtain a parametrix for the wave propagator on $X$, we will average the parametrices constructed above over $H_1 (X, \Z)$. Let $g_\eps$ be a family of Laplace isospectral metrics on $X$, $0 \leq \eps \leq \eps_0$. We write $(\wt{X}, \wt{g}) = (\R^2, \pi_X^*g_\eps)$ to denote the universal Riemannian covering of $X$, equipped with a locally isometric and $\Z^2$-periodic lift $\wt{g}_\eps$ of the metric $g_\eps$ to $\wt{X}$. 
\begin{prop}
    Let $D \subset \wt X$ be a fundamental domain for the torus and denote by $\wt E(t,x,y)$ the kernel of the wave propagator $\cos t \sqrt{-\Delta}$ on the universal cover with respect to the $\Z^2$-periodic lifted Riemannian metric. Then the locally finite series
    \begin{align*}
        E(t,x,y) : = \sum_{\gamma \in \Gamma} \wt E (t, x, \gamma y), \quad t \in \R, x, y \in D
    \end{align*}
    descends to the quotient $\R \times X \times X$ and defines the kernel of a solution operator for the wave equation which coincides with that in~\eqref{propagators}.
\end{prop}

\begin{proof}
    That $\wt E(t, x, \gamma y )$ solves the wave equation for each $\gamma \in \Gamma$ is clear, since $\Delta$ commutes with isometries. Averaging over $\gamma$ produces a $\Gamma$-invariant function and hence descends to the quotient $X^2 = \wt X / \Gamma \times \wt X / \Gamma$. By finite speed of propagation, $\wt E (t, x, z) = 0$ whenever $|t| < d(x,z)$, so the sum is locally finite and there is no issue of convergence. To see that the initial conditions are correctly satisfied, note that for $|t|$ sufficiently small, only the $\gamma = \Id$ term is nonzero and $\wt E(t,x,y) \to \dt(x-y)$ in the sense of distributions as $t \to 0$. The proposition then follows from uniqueness of the propagator on $X$.
\end{proof}

To compute the wave trace for deformations of an integrable metric in the next section, we are going to integrate our parametrix in Theorem~\ref{parametrix} over the diagonal. Since it was constructed microlocally on the universal cover, we need a way to relate integrals of functions on $X$ to their lifts to $\wt X$.

\begin{lemm}[Lemma 1 in \cite{CdVsllgp1}] \label{CdVlemma}
    There exists a function $\chi \in C^\infty(\wt X, [0,1])$ with compact support such that
    \begin{itemize}
        \item for every continuous $f: X \to \C$,
        \begin{align*}
            \int_X f(x) d \vol (x) = \int_{\wt X} \chi(\wt x) \pi_X^*f(\wt x) d \vol (\wt x).
        \end{align*}
        \item For every submanifold $Y \subset X$,
        \begin{align*}
             \int_Y f(y) d \vol_Y(y) = \int_{\pi_X^{-1}(Y)} \chi( \wt y) \pi_X^*f( \wt y) d \vol_{\pi_X^{-1}(Y)}( \wt y),
        \end{align*}
        \item the diameter of $\supp \chi$ is bounded by $2 \text{diam}(X) + 1$.
    \end{itemize}
\end{lemm}
 
\begin{rema}\label{positivity of chi CdV}
    We sketch the derivation of $\chi$ below without explicitly verifying the three claims above in order to emphasize the following property: for each $x \in X$, there exists at least one $\wt x \in \wt X$ such that $\pi_X(\wt x) = x$ and $\chi (\wt x) > 0$. In addition, for any $\beta \in \Gamma$, $\beta^* \chi$ has the same properties as $\chi$ does.
\end{rema}
\begin{proof}
    Let $\chi_1 \in C^\infty(\wt X, [0,1])$ be such that for every $x \in X$, there exists at least one $\wt{x} \in \wt{X}$ for which $\chi_1(\wt{x}) > 0$. This can be easily achieved by choosing a fundamental domain $D \subset \wt X$ and letting $\chi_1$ be a bump function which is compactly supported and identically equal to $1$ on $D$. Then set
    \begin{align}
        \chi(\wt{x}) = \frac{\chi_1(\wt{x})}{\sum_{\gamma \in \Gamma} \gamma^* \chi_1(\wt{x})}.
    \end{align}
\end{proof}

\subsection{The wave trace for deformations of an integrable metric}\label{sec: wave trace for deformations of an integral metric}
We now show how the parametrices $\wt E_N$ in Proposition~\ref{Cr approximation} descend to the quotient $X = \T^2$ given our dynamical hypotheses on $\Phi_{g_\eps}^t$.

\begin{theo}\label{parametrix}
    Assume that $(g_\eps)_{0 \leq \eps \leq \eps_0}$ is a smooth, one-parameter family of Riemannian metrics on $\T^2$ and suppose that $g_0$ has integrable geodesic flow. Then, for $\eps$ small and $x,y$ sufficiently close to the diagonal, the Schwartz kernel $E(t,x,y; \eps)$ of the wave propagator $\cos t \sqrt{-\Delta_\eps}$ is given microlocally near $g_0$-graph tori by a locally finite (in time) sum of Lagrangian distributions
    \begin{align*}
        E(t, x, y; \eps) = \sum_{\pm} \sum_{\gamma \in \Gamma} E_{\gamma, \pm}(t, x,y; \eps) \mod C^\infty(\R \times X \times X),
    \end{align*}
    each of which belongs to the class $I^{-1/4}(\R \times X \times X; \Lambda_{\gamma,\pm}^\eps)$, where
    \begin{align*}
        \Lambda_{\gamma, \pm}^\eps = & \left \{ (t, x, y, \tau, \xi, \eta): |t| = |\psi_\gamma^\eps(x,y)|,\,\, (\xi, \eta) = - \tau \text{sgn}(t) d_{x,y} \psi_\gamma^\eps, \,\, \mp \tau \text{sgn}(t) > 0 \right\}.
    \end{align*}
    Each Lagrangian distribution has an oscillatory integral representation of the form
    \begin{align*}
        E_{\gamma, \pm}(t,x,y;\eps) = e^{\mp i \pi (d-1)/ 4} \int_0^\infty e^{\pm i \theta ((\psi_{\gamma}^\eps)^2(x,y)  - t^2)} |t| a_{\gamma, \pm} (x,y, \theta; \eps) d\theta,
    \end{align*}
    where the amplitudes $a_{\gamma, \pm}$ are elliptic Kohn-Nirenberg symbols of order $\frac{d-1}{2}$ with positive principal part $a_{\gamma, 0, \pm}$ and $\psi_\gamma^\eps$ is the length of the unique geodesic connecting $x$ to $y$ which is close to a loop in the homology class $[\gamma]$ (cf. Corollary~\ref{cor:loop fn}).
\end{theo}

\begin{proof}
    Throughout the proof, we will fix $\eps$ and suppress the dependence of $E_{\gamma,\pm}(t,x,y;\eps)$, $\psi_\gamma^\eps(x,y)$, $a_{\gamma, \pm}(x,y; \eps)$ and $\Lambda_{\gamma, \pm}^\eps$ on $\eps$. Choose a conic, open set $\mathcal{O} \subset T^*X$ whose evolution under the $g_\eps$-geodesic flow for times $t \in I$ remains disjoint from $g_0$-nongraph tori and the $g_0$-separatrix region, where the foliation is nonsmooth. This can be done by choosing $\CO$ to correspond to a set of actions in $g_0$-action-angle coordinates and for small $\eps$, it is clear that within the region of phase space corresponding to $I \times \CO$, the Lagrangians $\Lambda_\pm$ project locally diffeomorphically onto $\R \times X \times X$. Now choose any zeroth order pseudodifferential operator $Q \in \Psi^0(X)$ which is microlocally elliptic on $\CO$ and has operator wavefront set in a slightly larger set whose $g_\eps$-geodesic flow intersects neither nongraph tori nor the separatrix for $t \in I$. When we say that two distributions $F$ and $G$ agree \textit{microlocally} near graph tori, we mean that
    \begin{align*}
        F Q - G Q \in C^\infty(\R \times X \times X).
    \end{align*}
    We will precompose $\cos t \sqrt{- \Delta}$ with $Q$ and show that it agrees with the sum
    $$
    \sum_\pm \sum_{\gamma \in \Gamma} E_{\gamma, \pm}(t,x,y),
    $$
    modulo a smooth remainder.

    \medskip
    
    Let $I \subset \R$ be a nonempty compact interval and denote by $C(I)$ the set of rotational homology classes for which the corresponding $\gamma$-loop functions take values in $I$, i.e.,
    \begin{align*}
        C(I) = \{[\gamma] \in H_1(X,\Z): \psi_\gamma(\wt X) \cap I \neq \emptyset, m_\gamma n_\gamma \neq 0\}.
    \end{align*}
    It is clear that for $\dt > 0$ small, all horizontal geodesics (with lifts projecting locally diffeomorphically onto their images in $X$) which connect points in a $\delta$-neighborhood of the diagonal and have length in $I$ are $\delta$-close to a $\gamma$-loop with $[\gamma] \in C(I)$. Define $\Gamma_{\CO, I}$ to be the collection of $\gamma \in \Gamma$ for which there exist geodesic $\gamma$-loops whose cotangent loops are contained completely in $\CO$ and have length in $I$. For $t \in I^\circ$, the microlocalized wave operator $\cos t \sqrt{- \Delta} \,Q$ is again a Fourier integral operator whose canonical relation is contained in $T^* (I^\circ \times X) \times \CO$, where the corresponding Lagrangians $\Lambda_{\gamma, \pm} \subset \Lambda_\pm$ project diffeomorphically onto the base; they are parametrized by the phase functions
    \begin{align*}
        \phi_{\gamma, \pm}(t,x,y,\theta) = \pm \theta(\psi_\gamma^2(x,y) - t^2).
    \end{align*}
    For each $\gamma \in \Gamma_{\CO, I}$, we Borel sum the amplitudes $W_{\gamma, k,\pm}$ from \eqref{eq:integratedw} to obtain
    \begin{align*}
         E_{\gamma, \pm}(t,x,y) \sim & |t| \sum_{k = 0}^\infty \int_0^\infty e^{\pm i \theta (\psi_{\gamma}^2(x,y)  - t^2)} W_{\gamma, k, \pm} (x,y) \theta_+^{\frac{d- 1}{2} - k} d\theta,
    \end{align*}
    which we claim are microlocal parametrices for $\cos t \sqrt{-\Delta}$.
    
    \medskip
    
    Indeed, the amplitudes $W_{\gamma, k, \pm}$ are obtained by integrating the transport equations \eqref{eq:transport} along geodesics which connect $x$ to $y$ within $\CO$. The transport equations \eqref{eq:transport} can be lifted to the Lagrangians $\Lambda_\pm$, where they admit global half-density solutions when tensored with an element of the Keller-Maslov line bundle \cite{GuilleminSternberg-SCA}. Whenever $\Lambda_\pm$ project locally diffeomorphically onto $\R \times X \times X$, a parametrization of $\Lambda_\pm$ using the phase functions $\pm \theta(\psi_\gamma^2 - t^2)$ and a trivialization of the Maslov bundle allow us to identify these half-density solutions with functions on $\R \times X \times X$. For $|t|$ small, the Borel sum of half-densities corresponding to $W_{\gamma, k, \pm}$ from \eqref{eq:transport}, when tensored with a suitable section of the Maslov bundle, coincides locally with the full symbol of the true solution operator on $\Lambda_\pm$ \cite{DuGu75}, \cite{DuHo72}. If $\gamma \in \Gamma_{\CO, I}$, then both $\psi_\gamma(x,y)$ and $W_{\gamma, k, \pm}(x,y)$ can be extended beyond the injectivity radius and the lifted half-densities associated to $W_{\gamma, k, \pm}$ continue to coincide with solutions of the transport equations in phase space. It follows that the operators defined by $E_{\gamma, \pm}(t,x,y)$ are microlocal parametrices for $\cos t \sqrt{- \Delta}$.

    \medskip
    
    Since the Maslov factors are given by $e^{\mp i \pi (d-1)/4}$ for $|t| \leq \inj(X,g)$, they remain constant along a tubular neighborhood of any orbit along which $\Lambda_\pm$ remains horizontal. Propagation of singularities tells us that microlocally near $I^\circ \times \CO \times \CO$, there are no orbits $\gamma \notin \Gamma_{\CO,I}$  which contribute to the singularities of $\cos t \sqrt{- \Delta} Q$, from which we see that the sum of oscillatory integrals
    \begin{align*}
        \sum_{\gamma \in \Gamma} \sum_\pm \sum_{k = 0}^\infty \int_0^\infty e^{\pm i \theta (\psi_{\gamma}^2(x,y)  - t^2)} |t|W_{\gamma, k, \pm} (x,y) \theta_+^{\frac{d- 1}{2} - k} d\theta,
    \end{align*}
    is locally finite in time. Ellipticity of the principal symbol follows directly from the formulas \eqref{eq:normalized Wk} and \eqref{eq:integratedw} for $W_{\gamma, 0, \pm}$.
\end{proof}

This is a geodesic flow analogue of the constructions in \cite{Vig21} and \cite{Vig22}, where an off-diagonal parametrix was constructed for the wave propagator microlocally near billiard loops with sufficiently small rotation number. Restricting to billiard loops of small rotation number excludes conjugate points, which is parallel to our consideration of only horizontal (graph) tori in the closed manifold case. The decomposition of $\Lambda_\pm$ into branches $\Lambda_{\gamma, \pm}$ resembles Chazarain's construction within a strictly convex planar domain \cite{Ch76}, where the sum over $\Gamma$ is replaced by a sum over bounce numbers (the number of times a billiard ball strikes the boundary). Guillemin and Melrose also used Chazarain's construction to derive a Poisson summation formula for more general Riemannian manifolds with strictly convex boundary \cite{GuMe79b}. A restriction of this parametrix to the diagonal of the boundary was used by Hezari and Zelditch to prove infinitesimal spectral rigidity of ellipses \cite{HeZe12}.

\begin{rema}\label{rem:microlocalCrapproximation}
    The finite regularity approximation in Proposition \ref{Cr approximation} carries over directly to the microlocalized parametrix construction in Theorem \ref{parametrix} above. Namely, if $Q \in \Psi^0(X)$ has operator wavefront set contained in the union of graph tori, then
    \begin{align*}
        \cos t \sqrt{- \Delta} \, Q - \sum_\pm \sum_{\gamma \in \Gamma} E_{\gamma, N, \pm} \,Q \in C_{\text{loc}}^{N - d - 2}(\R \times X \times X),
    \end{align*}
    where the operator $E_{\gamma, N, \pm}$ has Schwartz kernel
    \begin{align}\label{eq:EgammaNpm}
        E_{\gamma, N, \pm}(t,x,y) = \sum_{k = 0}^N \int_0^\infty e^{\pm i \theta (\psi_{\gamma}^2(x,y)  - t^2)} |t| W_{\gamma, k,  \pm} (x,y) \theta_+^{\frac{d-1}{2}-k} d\theta.
    \end{align}
    The amplitudes $W_{\gamma, k , \pm}$ in \eqref{eq:EgammaNpm} are constructed as in \eqref{eq:transport} by integrating along geodesic segments of an approximate $\gamma$-loop containing the points $x$ and $y$ (cf. Corollary \ref{cor:loop fn}). Similarly, we will denote by
    \begin{align}
        \wt E_{\gamma, N, \pm}(t,\wt x, \wt y) = \sum_{k = 0}^N \int_0^\infty e^{\pm i \theta (\psi_{\gamma}^2(\wt x, \wt y)  - t^2)} |t| \wt W_{\gamma, k,  \pm} (\wt x, \wt y) \theta_+^{\frac{d-1}{2}-k} d\theta
    \end{align}
    the kernel of a finite order parametrix on the universal cover $\R \times \wt X \times \wt X$, where $\wt W_{\gamma,k, \pm}$ are constructed exactly as before. Finally, we set
    \begin{align}\label{eq:Egammasumoverpm}
        E_{\gamma,N} = E_{\gamma,N,+} + E_{\gamma,N,-} \qquad \text{and} \qquad \wt E_{\gamma,N} = \wt E_{\gamma,N,+} + \wt E_{\gamma,N,-}.
    \end{align}
\end{rema}

A Marvizi-Melrose type formula for the wave trace follows immediately from Theorem~\ref{parametrix}.
\begin{theo}\label{MM style parametrix}
    Let $I \subset \R$ be a compact interval which is disjoint from the oscillatory length spectrum (cf. Definition \ref{def:NCC}). Then, the wave trace of $(\T^2, g)$ has an oscillatory integral representation near $I$ of the form
    \begin{align*}
        w(t) = \Re \left \{ e^{- i \pi \frac{d - 1}{4}} \sum_{\gamma \in \Gamma_I} \int_0^\infty \int_{\T^2} e^{i \theta (\psi_{\gamma}^2(x) - t^2)} b_\gamma(x, \theta) dx d\theta \right\} \mod C^\infty(\R),
    \end{align*}
    where $b_\gamma$ is an elliptic Kohn-Nirenberg symbol of order $\frac{d-1}{2}$ with positive principal part $b_{\gamma, 0} > 0$ and $\psi_\gamma(x)$ is the ``$\gamma$-loop function,'' which returns the length of the unique geodesic loop based at $x$ in the homology class $[\gamma]$. The index set $\Gamma_I$ is the collection of all homotopy classes which have periodic orbits of length belonging to the interval $I$.
\end{theo}
A similar microlocal trace formula for isolated but potentially degenerate periodic orbits appears in \cite{PopovDegenerate}, where the integral over the base manifold $X = \T^2$ is replaced by one over a Poincar\'e section with a corresponding symplectic density. In that case, the generating function for the first return map may not coincide with the loop function (which does not necessarily exist) and may depend on additional fiber variables, precisely because of the possibility of conjugate points. In the absence of conjugate points, we found it easier to follow the standard WKB approach so that we could be explicit about the contributions of Maslov indices to the trace formula.
		
\subsection{Noncancellation}

We now move on to show noncancellation in the wave trace. It is actually more convenient to consider decay (or lack thereof) of the resolvent trace, whose absence is equivalent to the existence of singularities in the wave trace. Choose a Schwartz function $\wh \rho$ with compact support in $(0, \infty)$ and define
\begin{align*}
    \I(\lambda) = \int_0^\infty e^{i t \lambda} \wh \rho(t) w(t) dt.
\end{align*}
If we consider spectral parameters in the upper half plane, then $\I(\lambda + i \tau)$ coincides with the trace of the \textit{regularized resolvent}:
\begin{align}\label{resolvent}
    \int_0^\infty e^{i t (\lambda + i \tau)} \wh \rho(t) E(t,x,y) dt = - i \int_{- \infty}^\infty \rho(\lambda - \mu) (\mu + i \tau) R(\mu + i \tau) d \mu \in \mathcal{O} \left(\C_{\lambda + i \tau}; C^\infty( X \times X) \right).
\end{align}
Here, $R(\mu + i \tau,x,y)$ is the Schwartz kernel of $(- \Delta - (\mu + i \tau)^2)^{-1}$ with $\tau > 0$ and $\mathcal{O}\left(\C_{\lambda + i \tau}; C^\infty( X \times X) \right)$ is the space of $C^\infty(X,X)$-kernel valued holomorphic functions. The formula \eqref{resolvent} follows from the resolvent identity
\begin{align*}
    R(\mu+i \tau) = \frac{i}{\mu+i \tau} \int_0^\infty e^{i t \mu - t \tau} \cos t \sqrt{-\Delta} dt
\end{align*}
and the standard convolution formula for inverse Fourier transforms of products. Since $\wh \rho(t)$ is compactly supported, $\I(\lambda)$ extends holomorphically to the entire complex plane. At the level of operators, we have
\begin{align*}
    \int_0^\infty e^{i t \lambda} \wh \rho(t) \cos t \sqrt{-\Delta} dt = \pi \left( \rho(\lambda + \sqrt{-\Delta})  + \rho(\lambda - \sqrt{-\Delta}) \right),
\end{align*}
which is a smoothing operator when $\rho \in \mathscr{S}(\R)$. Note that $w(t)$ is $C^\infty$ on the interior of $\supp \wh \rho$ if and only if $\I(\lambda) = O(\lambda^{-\infty})$. In order to show that $w(t)$ is singular, i.e., that there are no cancellations, we will show that $\I(\lambda) \neq O(\lambda^{-\infty})$.

\begin{lemm}\label{lemma:int over diag}
    For $I$ as in Theorem~\ref{MM style parametrix} and $\wh{\rho}$ supported near $I$, equal to one there and containing no other lengths of closed geodesics in its support besides those in $I$, we have
    \begin{align}
        \int_0^\infty e^{i t \lambda} \wh{\rho}(t) w(t) dt = \sum_{\gamma \in \Gamma_I} \int_0^\infty\int_{\wt X} \wh \rho(t) e^{i t \lambda} \chi(\wt x) \wt E_{\gamma, N}(t,\wt x, \gamma\wt x) d\wt x dt + O(\lambda^{2 + d - N}), \quad N \geq d + 2,
    \end{align}
    where $\wt E_{\gamma,N}$ is the operator defined by \eqref{eq:Egammasumoverpm}, the index set $\Gamma_I$ consists of homotopy classes which contribute lengths to $I$, i.e., 
    $$
    \Gamma_I = \{\gamma \in \Gamma : \text{CritVal}(\psi_\gamma) \cap I \neq \empty\}
    $$
    and
    $\chi \in C_0^\infty(\wt X)$ is the cutoff from Lemma \ref{CdVlemma}.
\end{lemm}

\begin{proof}
    If $\wt X$ were compact, the lemma would follow immediately from Proposition~\ref{Cr approximation} and its microlocal extension for larger times (cf. Remark \ref{rem:microlocalCrapproximation}). Instead, we use the cutoff $\chi$ from Lemma~\ref{CdVlemma}. We have
    \begin{align*}
        \int_X E(t,x,x) dx =& \sum_{\gamma \in \Gamma} \int_{\wt X} \chi(\wt x) \wt E (t, \wt x, \gamma \wt x) d \wt x\\
        =& \sum_{\gamma \in \Gamma} \int_{\wt X}\Bigg(  \chi(\wt x) \big(\wt E(t, \wt x, \gamma \wt x) - \wt E_{\gamma,N}(t, \wt x, \gamma \wt x)\big)
         + \chi(\wt x) \wt E_{\gamma,N}(t, \wt x, \gamma \wt x)  \Bigg) d \wt x.
    \end{align*}
	For $t \in I$, a bounded interval, it follows from wavefront considerations that all but finitely many terms in the sum over $\Gamma$ are smooth, so we may replace the sum over $\Gamma$ by one over the finite index set $\Gamma_I$, which is defined in the statement of the lemma. By the same reasoning as in Remark \ref{rem:microlocalCrapproximation}, we see that the term
	\begin{align*}
		\sum_{\gamma \in \Gamma_I} \int_{\wt X} \chi(\wt x) \big(\wt E(t, \wt x, \gamma \wt x) - \wt E_{\gamma,N}(t, \wt x, \gamma \wt x)\big) d \wt x
	\end{align*}
	is locally in $C^{N - d - 2}(\R)$ near $I$ and after multiplication by $\wh \rho(t)$, its Fourier transform has decay rate $O(\lambda^{2 + d - N})$.
\end{proof}

\begin{rema}\label{resolvent 2}
    As before, restricting $E(t,x,y)$ to the diagonal $\{x = y\}$ implicitly makes the assumption that its wavefront set does not meet the conormal bundle of the spatial diagonal in $\R \times X \times X$. While this is indeed true for $t \neq 0$, it is more convenient to interpret the above formula by taking the trace \textit{after} evaluating the $t \to \lambda$ Fourier transform, in which case the regularized resolvent is smoothing and its trace is easily defined.
\end{rema}
	
\subsection{Proof of Theorem~\ref{thm:laplacetolength}} \label{subsec:laplacetolength} We are now ready to use the parametrix construction above to prove length isospectrality in Theorem~\ref{thm:laplacetolength}. We begin by applying the method of stationary phase to reduce the number of phase variables from $d+ 2$ to $d$.

\begin{lemm}\label{Regularized resolvent}
    Let $\rho \in \mathscr{S}(\R)$ be such that the support of $\wh \rho$ does not intersect the oscillatory length spectrum. Then the regularized resolvent trace has the form
        \begin{align*}
        \int_0^\infty e^{i t \lambda} \wh \rho(t) w(t) dt = \left(\sum_{\gamma \in \Gamma} \int_{\wt X} e^{i \lambda \psi_{\gamma}(\wt x)} A_\gamma(\wt x) d \wt x \right) \lambda^{\frac{d-1}{2}}  + O(\lambda^{\frac{d-3}{2}}),
    \end{align*}
    where $A_\gamma \in C_0^\infty(\wt X)$ is given by
    \begin{align*}
        A_\gamma(\wt x) =
        \begin{cases}
            \pi \frac{\chi(\wt x) \wh \rho(\psi_\gamma(\wt x))}{\left(2 \psi_\gamma (\wt x)\right)^{\frac{d- 1}{2} }} \wt W_{\gamma, 0,+}(\wt x, \gamma \wt x) & \gamma \in \Gamma_\rho,\\
            0 & \gamma \notin \Gamma_\rho,
        \end{cases}
    \end{align*}
    and $\psi_\gamma(\wt x) = \psi(\wt x, \gamma \wt x)$ is a lift of the $\gamma$-loop function from $X$ to $\wt X$. The index set $\Gamma_\rho$ in the definition of $A_\gamma$ is the set of all (rotational) homotopy classes in $\Gamma$ for which $\text{CritVal} (\psi_\gamma) \cap \supp \wh \rho \neq \emptyset$.
\end{lemm}

\begin{proof}
    Let $\Gamma_\rho$ be as in the statement of the lemma and apply Lemma \ref{lemma:int over diag}:
    \begin{align*}
        \int_0^\infty e^{i t \lambda} \wh{\rho}(t) w(t) dt = & \sum_{\gamma \in \Gamma_\rho} \int_0^\infty\int_{\wt X} \wh \rho(t) e^{i t \lambda} \chi(\wt x) \wt E_{\gamma,N}(t,\wt x, \gamma \wt x) d\wt x dt + O(\lambda^{2 + d - N})\\
        = &
        \sum_\pm \sum_{\gamma \in \Gamma_\rho} \sum_{k = 0}^N \int_{\wt X} \int_0^\infty \int_0^\infty e^{it \lambda \pm i \theta(\psi^2(\wt x, \gamma \wt x) - t^2)} \wh \rho(t) \chi(\wt x) |t| \wt W_{\gamma, k, \pm}(\wt x, \gamma \wt x) \theta_+^{\frac{d-1}{2} - k} d\theta dt d \wt x
        \\
        & + O(\lambda^{2 + d - N}).
    \end{align*}
    Thus, modulo $O(\lambda^{2 + d - N})$, we have reduced our regularized resolvent trace to a finite sum of oscillatory integrals, each of which can be estimated separately. Choose $N \geq \max\left\{d+ 2, \frac{d + 7}{2}\right\}$ so that the remainder above can be absorbed into the $O(\lambda^{\frac{d-3}{2}})$ term in the lemma. Using the change of variable $\theta = \lambda \eta$, we rewrite the above integral as
    \begin{align}
    \label{eq:new phase who this}
        \sum_\pm \sum_{\gamma \in \Gamma_\rho} \sum_{k = 0}^N \int_{\wt X} \int_0^\infty \int_0^\infty e^{i \lambda \Phi_{\gamma, \pm}(t, \eta, \wt x)} \wh \rho(t) |t| \chi(\wt x) \wt W_{\gamma, k, \pm}(\wt x, \gamma \wt x) \lambda^{\frac{d+1}{2} - k} \eta_+^{\frac{d-1}{2} - k} d \eta dt d \wt x,
    \end{align}
    where the new phase function is defined by
    $$
    \Phi_{\gamma, \pm}(t, \eta, \wt x) = t \pm \eta(\psi^2(\wt x, \gamma \wt x) - t^2).
    $$
    We now apply the method of stationary phase with respect to the variables $(t, \eta)$ in \eqref{eq:new phase who this}; the four stationary points of the phases $\Phi_{\gamma, \pm}$ with respect to $(t, \eta)$ are
    \begin{align*}
        d_{t, \eta} \Phi_{\gamma,\pm} = 0
        \iff
        \begin{cases}
            |t| = |\psi(\wt x,\gamma \wt x)|,\\
            \eta = \pm \frac{1}{2 \psi(\wt x, \gamma \wt x)}.
        \end{cases}
    \end{align*}
    The point $\eta = - \frac{1}{2 \psi(\wt x, \gamma \wt x)}$ doesn't belong to the domain of the $\eta$-integral and similarly, since $\wh \rho$ localizes near positive values of $t$, we can also ignore the stationary point where $t = - \psi(\wt x, \gamma \wt x)$. We therefore consider only the stationary points $t, \eta > 0$ of $\Phi_{\gamma,+}$, at which the Hessian is given by
    \begin{align}\label{Hessian}
        D_{t, \eta}^2 \Phi_{\gamma,+} \left(\psi (\wt x, \gamma \wt x), \frac{1}{2 \psi(\wt x, \gamma \wt x)} \right) =
        \begin{pmatrix}
            - \frac{1}{\psi (\wt x, \gamma \wt x)} & - 2 \psi(\wt x, \gamma \wt x)\\
            - 2 \psi(\wt x, \gamma \wt x)& 0
        \end{pmatrix}.
    \end{align}
    Since we are working away from $t = 0$, $\psi$ is nonvanishing and we can replace the integrand by its stationary phase approximation. From~\eqref{Hessian}, one sees that since the $2 \times 2$ Hessian has negative determinant, its signature must be $\mu = 0$ and the usual phase factor $e^{i \pi \mu/4}$ becomes $1$. Retaining only the stationary terms ($0 \leq k \leq N$, $+$) in \eqref{eq:new phase who this}, we have
    \begin{align*}
        \pi \int_{\wt X} e^{i  \lambda \psi(\wt x,\gamma  \wt x)} \wh \rho(\psi(\wt x, \gamma \wt x) )\chi(\wt x) \left(\frac{\lambda}{2 \psi( \wt x, \gamma \wt x)}\right)^{\frac{d- 1}{2} - k} \wt W_{\gamma, k,+}(\wt x, \gamma \wt x)  d \wt x + O\left(\lambda^{\frac{d-1}{2} - k - 1}\right).
    \end{align*}
    The main term comes from $k = 0$, which contributes $\lambda^{\frac{d-1}{2}}$. The formula \eqref{eq:new phase who this} then becomes
    \begin{align*}
        \pi \int_{\wt X} e^{i  \lambda \psi(\wt x,\gamma  \wt x)} \chi(\wt x) \wh \rho(\psi(\wt x, \gamma \wt x)) \left(\frac{\lambda}{2 \psi( \wt x, \gamma \wt x)}\right)^{\frac{d- 1}{2} } \wt W_{\gamma, 0,+}(\wt x, \gamma \wt x)  d \wt x + O\left(\lambda^{\frac{d-3}{2}}\right)
    \end{align*}
    and the lemma follows from summing over $\Gamma_\rho$.
\end{proof}

For the remainder of this section, we will set $(X,g) = (\T^2, g_\eps)$, where $g_\eps$ is a one-parameter family of Riemannian metrics and $g_0$ has integrable geodesic flow, as in Section \ref{sec:dynamics}. In our proof of Theorem \ref{thm:laplacetolength}, it will be convenient to introduce the following notions.

\begin{def1}\label{def:clusters}
    Let $[\gamma] = (m,n) \in \Z^2$ correspond to graph rational tori, i.e., $mn \neq 0$ in~\eqref{osc Lsp}, and consider the $(m,n)$-loop functions $\psi_\gamma^\eps$ (whose existence is guaranteed by Corollary~\ref{cor:loop fn}) for the metric $g_\eps$. Now set
    \begin{align*}
        t_{\gamma, \eps} = \min_{x \in \T^2} \psi_\gamma^\eps(x), \qquad \text{and} \qquad T_{\gamma, \eps} = \max_{x \in \T^2} \psi_\gamma^\eps(x). 
    \end{align*}
    We call $\text{CritVal}(\psi_\gamma^\eps) \cap [t_{\gamma, \eps}, T_{\gamma, \eps}]$ a $\textbf{$\gamma$-cluster}$ in the length spectrum and $[t_{\gamma,\eps}, T_{\gamma,\eps}]$ a \textbf{$\gamma$-cluster interval}.
\end{def1}

We will also need the following theorem of Soga.

\begin{theo}[\cite{Soga}]\label{SogasTheorem}
    Denote by
    \begin{align*}
        \J(\lambda) = \int_{\R^2} e^{i \lambda \phi(x)} a(x) dx,
    \end{align*}
    where $a(x) \ge 0$ on $\mathbb{R}^2$ and $\phi$ is a smooth function of $x$. Assume that there exists at least one degenerate stationary point $x_0$ of $\phi$ satisfying $a(x_0) > 0$. Then $\J(\lambda)$ does not decrease rapidly:
\begin{align*}
    (1 + |\lambda|)^m \J(\lambda) \notin L^2(\mathbb{R}^1)
\end{align*}
for an $m < \frac{1}{2}$.
\end{theo}

We now proceed with the proof of Theorem~\ref{thm:laplacetolength} in the case that $\eps$ is small enough that the noncoincidence condition~\eqref{eq:INCC} persists.

\begin{proof}[Proof of Theorem~\ref{thm:laplacetolength}]
    We set $X = \T^2$ and $g = g_\eps$ as in the statement of Theorem \ref{thm:laplacetolength}. Assume to the contrary that the rotational length spectrum is not constant within some compact interval $I$ on which the noncoincidence condition \eqref{eq:INCC} holds. {By compactness and closedness of $\CL^{\rot}(\T^2, g_\eps), \CL^{\osc}(\T^2,g_\eps)$ (see Proposition \ref{prop:LSP close}), it follows that for each $0 \leq \eps \leq \eps_0$, there is a positive gap between the rotational and oscillatory length spectra.
    
    \medskip 
    
    Without loss of generality, let us shrink the interval $I$ so that when $\eps = 0$, $\CL(\T^2, g_0) \cap I = \{\ell_0\}$, where $\ell_0 \in \CL^{\rot}(\T^2,g_0)$. In other words, we assume that $I$ contains exactly one rotational length for the integrable metric $g_0$ and does not contain any oscillatory lengths. This is possible by the noncoincidence condition when $\eps = 0$ and the fact that $\CL^{\rot}(X, g_0)$ is discrete; it is indexed by $\gamma \in \Gamma$ and unlike $\CL^{\osc}(X,g_0)$, cannot have accumulation points. Note that there might exist several rational tori in distinct homology classes whose corresponding periodic orbits have length $\ell_0$.
    
    \medskip
    
    \noindent \textbf{Claim 1:} There exists $\eps_2 > 0$ such that for $0 \leq \eps \leq \eps_2$, $I \cap \CL^{\osc}(\T^2, g_\eps) = \emptyset$.
    
    \medskip
    
    To see why this is true, assume to the contrary that there is a sequence $\eps_j \to 0$ such that for each $\eps_j$, there exists a periodic oscillatory $g_{\eps_j}$-geodesic $\alpha_j$ whose length belongs to $I$. Since the metrics $(g_\eps)_{0 \leq \eps \leq \eps_0}$ are uniformly comparable to $g_0$, there are only finitely many homology classes (rotational or oscillatory) which contain closed geodesics with lengths belonging to $I$ for any $0 \leq \eps \leq \eps_0$. Thus, among the periodic oscillatory $g_{\eps_j}$-geodesics with lengths in $I$, we can select a subsequence $\alpha_{j_k}$ which belongs to a fixed oscillatory homology class. By compactness, we can choose a further subsequence $\alpha_{j_{k_\ell}}$ which converges to a periodic oscillatory $g_0$-geodesic whose length belongs to $I$, contradicting our assumption on $I$.
    
    \medskip
    
    \noindent \textbf{Claim 2.} Denote by $\Gamma(\ell_0)$ the collection of all homotopy classes for which corresponding periodic $g_0$-geodesics have length $\ell_0$. There exists $\eps_3 > 0$ such that for $0 \leq \eps \leq \eps_3$,
    \begin{enumerate}
        \item if $\gamma \in \Gamma(\ell_0)$, then $\min \psi_{\gamma}^\eps(X), \max \psi_\gamma^\eps(X) \in I$.

        \item if $\gamma \notin \Gamma(\ell_0)$, then $\psi_\gamma^\eps(X) \cap I = \emptyset$.
    \end{enumerate}
}
    \medskip
    
    {Claim $2$ follows directly from the fact that the loop functions $\psi_\gamma$ in Corollary~\ref{cor:loop fn} depend continuously on $\eps$. In other words, Claims $1$ and $2$ establish that as $\eps \geq 0$ increases, any new lengths in $I$ are necessarily rotational and must arise as bifurcations of $\ell_0 \in \CL(\T^2, g_0)$.}

    \medskip
    
    {Set $\eps_1 = \min(\eps_2,\eps_3)$. For $0 \leq \eps \leq \eps_1$, denote by
    \begin{align*}
        t_\eps(\ell_0) = & \min_{\gamma \in \Gamma(\ell_0)} \min_{x \in \T^2} \psi_\gamma^\eps (x)
        \\
        T_\eps(\ell_0) = & \max_{\gamma \in \Gamma(\ell_0)} \max_{x \in \T^2} \psi_\gamma^\eps(x)
    \end{align*}
    the endpoints of a union of such cluster intervals in the length spectrum. We will show that both $t_\eps(\ell_0)$ and $T_\eps(\ell_0)$ belong to the singular support of the wave trace for $\eps$ small. Together with Laplace isospectrality, this will demonstrate that $t_\eps(\ell_0) = T_\eps(\ell_0) = \ell_0$. For brevity, we will write $t_\eps$ and $T_\eps$ moving forward, without specifying $\ell_0$.} Since smoothness of the wave trace is equivalent to rapid decay of its Fourier transform, it suffices to show that
    \begin{align*}
        \I(\lambda) := \int_0^\infty e^{i \lambda t} \wh \rho(t) w(t) dt
    \end{align*}
    is \textit{not} rapidly decreasing in $\lambda$ for some $\rho \in \mathscr{S}(\R)$ such that $\wh \rho \geq 0$ is compactly supported in $I$ and identically equal to one on a neighborhood of either $t_\eps$ or $T_\eps$. We use the formula for $\I(\lambda)$ in Lemma~\ref{Regularized resolvent}:
    \begin{align*}
        \I(\lambda) = \left(\sum_{\gamma \in \Gamma} \int_{\wt X} e^{i \lambda \psi_{\gamma}(\wt x)} A_\gamma(\wt x) d \wt x \right) \lambda^{\frac{d-1}{2}}  + O(\lambda^{\frac{d-3}{2}}).
    \end{align*}
    For $0 \leq \eps < \eps_1$, we have
    \begin{align}
        \Gamma_\rho(\ell_0):= \{\gamma: \psi_\gamma^{\eps}(\wt X) \cap \supp \wh \rho \neq \emptyset \} \subset \Gamma(\ell_0),
    \end{align}
    i.e., all geodesic loops for the metric $g_\eps$ which have length in the support of $\wh \rho$ belong to a homotopy class in $\Gamma(\ell_0)$. In particular, the collection of all such $\gamma$ is finite. Denote the elements of $\Gamma_\rho(\ell_0)$ by $\gamma_1, \gamma_2, \cdots, \gamma_q$. Now choose a fundamental domain $D \subset \wt X$ together with $\beta_i \in \Gamma$, $1 \leq i \leq q$ such that $\supp \beta_i^*\chi \cap \supp \beta_j^* \chi = \emptyset$ when $i \neq j$. The purpose of introducing the elements $\beta_i$ is to separate out the supports of the test functions $\beta_i^* \chi$ on $\wt X$, with the goal of replacing our sum of integrals from Lemma \ref{Regularized resolvent} by a single integral, the integrand of which will contain an amplitude which is multiplied by a \textit{single} oscillatory factor $e^{i \lambda \Psi}$; here $\Psi$ will be patched together out of the loop functions $\psi_\gamma^\eps$ with $\gamma \in \Gamma_\rho(\ell_0)$, each of which will be defined on the disjoint union of fundamental domains $\beta_i D \subset \wt X$. Since each $\beta_i$ is an isometry, we have
    \begin{align*}
        \I(\lambda) = \lambda^{\frac{d-1}{2}} \int_{\wt X} \sum_{j = 1}^q e^{i \lambda \psi_{\gamma_j}(\wt x)} \beta_j^* A_{\gamma_j}(\wt x) d \wt x + O\left(\lambda^{\frac{d-3}{2}}\right).
    \end{align*}
    Note that since each $A_{\gamma_j}$ contains a factor of $\chi(\wt x)$, all summands in the integral above have disjoint support in $\wt X$. By periodicity, we also have that $\psi(\beta_i \wt x, \gamma_i \beta_i \wt x) = \psi ( \wt x, \gamma_i \wt x)$ for every $1 \leq i \leq q$. Let $\phi_i \in C_c^\infty(\wt X, [0,1])$ be smooth cutoffs which are identically equal to one on $\supp \beta_i^* \chi$ and vanish on $\supp \phi_j$ whenever $i \neq j$. Then, we can write
    \begin{align}
    \label{eq:PsiZ integral}
        \I(\lambda) = \lambda^{\frac{d-1}{2}} e^{- i \pi \frac{d-1}{4}} \int_{\wt X} e^{i \lambda \Psi (\wt x)} Z(\wt x)\, d \wt x  + O\left(\lambda^{\frac{d-3}{2}}\right),
    \end{align}
    where
    \begin{align*}
        \Psi(\wt x) = \sum_{j = 1}^q \phi_j(\wt x) \psi(\wt x, \gamma_j \wt x) \qquad \text{and} \qquad
        Z(\wt x) = e^{i\pi \frac{d-1}{4}} \sum_{i = 1}^q \beta_i^* A_{\gamma_i}(\wt x).
    \end{align*}
    The function $\Psi$ is a \textit{single} phase function which coincides with $\psi_{\gamma_i}(\wt x)$ on the support of each $\beta_i^*A_{\gamma_i}$, $1 \leq i \leq q$ and $Z(\wt x) \geq 0$. Each critical point of $\Psi$ in the support of $Z$ corresponds to a periodic orbit. Denote the integral in \eqref{eq:PsiZ integral} by $\I_0(\lambda)$.

    \medskip
    
    {Recall that we are assuming, for the sake of contradiction, that the metrics $g_\eps$ are not length isospectral on the interval $I$. There are several cases to consider.}
    
    \medskip
    
    \noindent \textbf{Case 1.} Suppose that $t_\eps = T_\eps \neq \ell_0$ and choose a test function $\wh \rho$ which is identically equal to one near $t_\eps = T_\eps$ and vanishes on an open neighborhood of the rest of the length spectrum. This is possible as a consequence of Claims $1$ and $2$ above. Collapse of the $\gamma$-length spectra implies constancy of the $\gamma$-loop functions $\psi_\gamma^\eps$, or equivalently, the existence of rational tori for each $\gamma \in \Gamma(\ell_0)$, all of whose orbits have length $t_\eps = T_\eps$. Since the phase function $\Psi$ of \eqref{eq:PsiZ integral} is \textit{constant} on the support of $Z$, {the oscillatory factor can be removed from the integral:
    \begin{align*}
        \I(\lambda) = \lambda^{\frac{d-1}{2}} e^{i t_\eps \lambda } e^{- i \pi \frac{d-1}{4}} \int_{\wt X} Z(\wt x) d \wt x + O(\lambda^{\frac{d-3}{2}}).
    \end{align*}
    Since $Z \geq 0$ and has positive integral, we conclude that $\I(\lambda)$ is not rapidly decreasing.}
    
    \medskip
    
    \noindent \textbf{Case 2.} Now consider the case $t_\eps < T_\eps$. Without loss of generality, assume that $t_\eps \notin \CL(X,g_0)$, or equivalently, that $t_\eps \neq \ell_0$; otherwise, we may instead consider $T_\eps \notin \CL(\T^2,g_0)$. We choose a test function $\rho \in \mathscr{S}(\R)$ such that $\wh{\rho}$ is nonnegative, compactly supported in $I$, equal to one on a neighborhood of $t_\eps$, and vanishes near $T_\eps$. We will again derive a contradiction by showing that the integral $\I(\lambda)$ is not rapidly decreasing in $\lambda$. {Note that the set $\Psi^{-1}(t_\eps) \cap \supp Z$ consists entirely of points which belong to minimal periodic geodesics with homotopy classes in $\Gamma_\rho(\ell_0)$. Hence, $\Psi$ is in fact locally constant along the projection of such an orbit from $S^* \T^2$ to $\T^2$. In particular, $D^2 \Psi$ is degenerate at all such critical points. It is clear that the amplitude $Z$ of $\I_0(\lambda)$ in \eqref{eq:PsiZ integral} is nonnegative everywhere and positive on a nonempty set of points in $\Psi^{-1}(t_\eps)$ (cf. Remark \ref{positivity of chi CdV}).} Hence, Soga's Theorem (Theorem \ref{SogasTheorem}) implies that $|\I_0(\lambda)| (1 + |\lambda|)^{\nu} \notin L^2(\R_+)$ for some $\nu < 1/2$. We claim that \eqref{eq:PsiZ integral} cannot be $O(\lambda^{-\infty})$. If it were, then the integral $\I_0(\lambda)$ would be $O(\lambda^{-1})$, in which case $\I_0(\lambda) (1 + \lambda)^\nu = O(\lambda^{\nu - 1})$. Since $2(\nu - 1) < -1$, $\I_0(\lambda) (1+ \lambda)^{\nu} \in L^2(\R_+)$, contradicting the conclusion of Soga's Theorem. Hence, $\I(\lambda)$ is not rapidly decreasing. (Note that the condition $d = 2$ is needed in order to apply Theorem \ref{SogasTheorem}).

    \medskip

    The preceding argument shows that $w(t)$ is singular at $t_\eps$. By symmetry, the same proof establishes singularity of $w(t)$ at $T_\eps$. From isospectrality and the fact that within $I$, $w(t)$ is singular only at $\ell_0$ when $\eps = 0$, we see that $t_\eps$ must coincide with $\ell_0$, which is a contradiction. We conclude that $g_\eps$ must be rotationally length isospectral on any interval for which the noncoincidence condition holds.
    
\end{proof}
    
Note that in the proof above, we used the fact that for $\eps = 0$, the lengths of rotational periodic orbits at which the wave trace is singular are isolated. In general this may not hold, even in the completely integrable setting, for oscillatory lengths (cf. Definition \ref{def:osc Lsp}). Consider the following example, suggested to us by J. de Simoi:
\begin{exa}
    Take a smooth function $f$ on $\R/\Z$ which has infinitely many critical values accumulating at some finite number $L$. Revolving its graph around the circle, we obtain a metric of revolution on the torus, corresponding to a completely integrable Hamiltonian system. The periodic orbits at which the argument of $f$ is a critical point have lengths which accumulate at $L$, despite the discrete indexing of orbits by homology class. In this case, either (1) there are infinitely many rational tori whose orbits belong to the homology class $(1,0)$, with lengths approaching $L$ or (2) these orbits may in fact belong to the separatrix, a measure zero region in $T^* \T^2$ where the foliation by invariant tori is not smooth. In either case, these orbits contribute only to the oscillatory length spectrum $\CL^{\osc} (\T^2,g)$ in \eqref{osc Lsp} and are excluded by our noncoincidence assumption (Definition \ref{def:NCC}). If there are infinitely many such $(1,0)$-rational tori which belong to a smooth part of the foliation, then these (nongraph) tori project vertically onto $\T^2$ at some point, corresponding to conjugate points. Such conjugate points could produce nontrivial Maslov indices in the trace formula and jeopardize our noncancellation argument. By contrast, Theorem~\ref{prop:geom RT} guarantees that for \textit{graph}-rational tori, there is exactly one torus and all of its foliating orbits are minimal.
\end{exa}

\section{A second variation approach: Proof of Theorem~\ref{thm:RT rig}} \label{sec:secondorder}

In this section, we present the proof of Theorem~\ref{thm:RT rig}, which is the second step in our overall scheme for the proof of Theorem~\ref{thm:main2}.

\medskip

The basic idea is to compute the second variation of the \emph{energy functional} and show that, thanks to the dependence on $\epsilon$ being only linear, the unperturbed geodesic $\gamma^{(0)}$ is a critical point of both the unperturbed energy functional $E^{(0)}$ and the first-order expansion $E^{(1)}$; see \eqref{eq:E explicit} below. Combining the two associated Euler-Lagrange equations then allows us to conclude that $U \equiv 0$.
We refer to the paragraph at the end of this section for a brief review of relevant notions in the calculus of variations used in this proof.
 
\medskip

To start the proof, let us suppose, as in the statement of Theorem~\ref{thm:RT rig}, that a deformation $(g_\epsilon)$ whose line element takes the form
\begin{equation}\label{eq:conformal deform in pf}
    \dif s^2_\epsilon = ( \Lambda(x_1, x_2) + \epsilon U(x_1, x_2)) (\dif x_1^2 + \dif x_2^2)
\end{equation}
preserves a rational torus and moreover the lengths of closed geodesics in that torus are constant in $\epsilon$.
Suppose the preserved rational torus has homology class $(m_1, m_2)$ and the length of its closed geodesics is $$T > 0.$$
We consider the lift of $g_\epsilon$ to the universal cover $\R^2$ and still denote it by $g_\epsilon$.
The \textit{energy} $E_\epsilon(\gamma)$ of any $C^1$ curve $\gamma: [0,T] \to \R^2$ with respect to the metric $g_\epsilon$ is 
\begin{equation}\label{eq:energy}
    E_\epsilon (\gamma)  = \int_0^T [\Lambda(\gamma(t)) + \epsilon U(\gamma(t))] |\dot\gamma(t)|^2 \dif t.
\end{equation}
Here and for the rest of this section, we use $|\cdot|$ to denote the standard Euclidean norm and we use $\dot\gamma(t)$ and $\ddot\gamma(t)$ to denote the first and second derivatives of $\gamma(t)$ with respect to $t$.
The study of geodesics as critical points of the energy functional is classical.
We refer to \cite{milnor1963morse} for more details.

\medskip

Fix $a, b \in \R^2$ such that $b = a + (m_1, m_2)$.
By Corollary~\ref{cor:loop fn} and a lift of geodesics from $\T^2$ to $\R^2$, for each $\epsilon$, the plane $\R^2$ is foliated by a family of unit-speed geodesics of $g_\epsilon$ such that each geodesic $\gamma$ in this foliating family satisfies $\gamma(0) + (m_1, m_2) = \gamma(T)$ and corresponds to a closed geodesic in the $(m_1, m_2)$-rational torus after projection.
Therefore, for each $\epsilon$ we pick the unique unit-speed geodesic $\gamma_\epsilon:[0,T] \to \R^2$ of $(\R^2, g_\epsilon)$ in this foliating family that joins $a$ and $b$.
Note that Corollary~\ref{cor:loop fn} implies $\gamma_\epsilon$ depends smoothly on $\epsilon$.
Since $\gamma_\epsilon$ has unit-speed parametrization, a standard calculation (see, for example,~\cite[\S 12]{milnor1963morse}) shows that $E_\epsilon(\gamma_\epsilon) = T$ and thus the energy of $\gamma_\epsilon$ is constant in $\epsilon$.

\medskip

Let us now write the $\epsilon$-expansion of $E_\epsilon$ and $\gamma_\epsilon$ as follows 
\begin{equation*}    \gamma_\epsilon = \gamma^{(0)} + \epsilon \gamma^{(1)} + \epsilon^2 \gamma^{(2)} + \CO(\epsilon^3), \qquad E_\epsilon = E^{(0)} + \epsilon E^{(1)}.
\end{equation*}
Explicitly, we have 
\begin{equation}\label{eq:E explicit}
    E^{(0)}(\gamma) = \int_0^T \Lambda(\gamma(t)) |\gamma'(t)|^2 \dif t \quad \text{and} \quad  
    E^{(1)}(\gamma) = \int_0^T U(\gamma(t)) |\gamma'(t)|^2 \dif t.
\end{equation}
We now compute the $\epsilon$-expansion of $E_\epsilon(\gamma_\epsilon)$:
\begin{equation}\label{eq:E expansion}
    \begin{aligned}
    E_\epsilon(\gamma_\epsilon) =& E^{(0)}(\gamma^{(0)}) + \epsilon \Big(E^{(1)}(\gamma^{(0)}) + \dif_{\gamma^{(0)}}E^{(0)}(\gamma^{(1)}) \Big) \\
    &+ \epsilon^2 \Big(\dif_{\gamma^{(0)}}E^{(1)}(\gamma^{(1)}) + \frac{1}{2}\dif^2_{\gamma^{(0)}}E^{(0)}(\gamma^{(1)},\gamma^{(1)}) + \dif_{\gamma^{(0)}}E^{(0)}(\gamma^{(2)} )\Big) + \CO(\epsilon^3).
    \end{aligned}
\end{equation}
On the other hand, since $\gamma_\epsilon$ is by construction a geodesic with respect to $g_\epsilon$, Theorem~\ref{thm:variation of geod} implies that $\dif_{\gamma_\epsilon}E_\epsilon(\delta) = 0$ for any variation vector field $\delta$ along $\gamma_\epsilon$ fixing endpoints. Expanding in $\epsilon$ yields
\begin{equation}\label{eq:dE expansion}
    0 = \dif_{\gamma_\epsilon}E_\epsilon(\delta) = \dif_{\gamma^{(0)}}E^{(0)}(\delta) + \epsilon \Big( \dif^2_{\gamma^{(0)}} E^{(0)}(\delta, \gamma^{(1)}) + \dif_{\gamma^{(0)}}E^{(1)}(\delta) \Big) + \CO(\epsilon^2).
\end{equation}
Since the endpoints \(a,b\) are independent of \(\epsilon\), we have
\[
\gamma^{(j)}(0)=\gamma^{(j)}(T)=0,\qquad j\ge1.
\]
Using \eqref{eq:dE expansion} with $\delta = \gamma^{(2)}$ and $\delta = \gamma^{(1)}$, we can simplify \eqref{eq:E expansion} to
\begin{equation}\label{eq:E reduced}
    E_\epsilon(\gamma_\epsilon) = E^{(0)}(\gamma^{(0)}) + \epsilon E^{(1)}(\gamma^{(0)})  - \frac{1}{2} \epsilon^2 \dif^2_{\gamma^{(0)}}E^{(0)}(\gamma^{(1)},\gamma^{(1)}) + \CO(\epsilon^3).
\end{equation}
We arrive at the second-order condition 
\begin{equation}\label{eq:d2E condition}
    \dif^2_{\gamma^{(0)}}E^{(0)}(\gamma^{(1)},\gamma^{(1)}) = 0.   
\end{equation}
By Proposition~\ref{prop:geom RT} part (1), the geodesic $\gamma^{(0)}$ minimizes the $g_0$-distance between its endpoints.
Therefore, since $\gamma^{(0)}(t)$ has unit-speed, it also minimizes the energy functional $E^{(0)}$ (with fixed endpoints).
It follows that the Hessian $\dif^2_{\gamma^{(0)}}E^{(0)}$ is positive semi-definite (see, for example, Theorem~\ref{thm:variation of geod}). 
\footnote{We remark that the proof of Proposition~\ref{prop:geom RT} (3) actually shows that any geodesic lying in a rational torus has no conjugate points over a period. This is consistent with the Morse index theorem, which states that the dimension of the maximal subspace on which $\dif^2_\gamma E$ is negative definite coincides with the number of conjugate points along $\gamma$, counted with multiplicity.} 
In particular, for any variation $\delta$ fixing endpoints and any $c \in \R$, we have 
\begin{equation*}
    0 \leq \dif^2_{\gamma^{(0)}}E^{(0)}(\gamma^{(1)} + c\delta, \gamma^{(1)} + c\delta) = \dif^2_{\gamma^{(0)}}E^{(0)}(\gamma^{(1)}, \gamma^{(1)}) + 2 c  \dif^2_{\gamma^{(0)}}E^{(0)}(\gamma^{(1)},\delta) + c^2  \dif^2_{\gamma^{(0)}}E^{(0)}(\delta, \delta).
\end{equation*}
Combining with \eqref{eq:d2E condition}, the above inequality for arbitrary $c \in \R$ forces $\dif^2_{\gamma^{(0)}}E^{(0)}(\gamma^{(1)},\delta) = 0$ for arbitrary variation $\delta$ fixing endpoints.
Substituting into \eqref{eq:dE expansion}, we deduce that 
\begin{equation*}    
    \dif_{\gamma^{(0)}}E^{(1)} = 0,
\end{equation*} 
 i.e., the undeformed geodesic $\gamma^{(0)}$ is also a critical point of the perturbing term $E^{(1)}$, in addition to being a critical point of the original energy functional $E^{(0)}$.
In these two cases, the Euler-Lagrange equations satisfied by $\gamma^{(0)}$ write respectively
\begin{equation*}
    \frac{1}{2} |\dot \gamma^{(0)}(t)|^2\nabla_{\gamma^{(0)}(t)} U 
    = \big(\nabla_{\gamma^{(0)}(t)}U \cdot  \dot \gamma^{(0)}(t)\big) \dot \gamma^{(0)}(t)   + U(\gamma^{(0)}(t) ) \ddot \gamma^{(0)}(t)
\end{equation*}
and 
\begin{equation*}
    \frac{1}{2} |\dot \gamma^{(0)}(t)|^2\nabla_{\gamma^{(0)}(t)} \Lambda 
    = \big(\nabla_{\gamma^{(0)}(t)}\Lambda \cdot  \dot \gamma^{(0)}(t)\big) \dot \gamma^{(0)}(t)   + \Lambda(\gamma^{(0)}(t) ) \ddot \gamma^{(0)}(t).
\end{equation*}
Taking the Euclidean inner product with $\dot \gamma^{(0)}(t)$ and rearranging, we obtain from the first equation
\begin{equation*}
    0 = \frac{1}{2} |\dot \gamma^{(0)}(t)|^2 \frac{\dif}{\dif t} U(\gamma^{(0)}(t) ) + U(\gamma^{(0)}(t) ) \ddot \gamma^{(0)}(t) \cdot \dot \gamma^{(0)}(t).
\end{equation*}
For $t \in \R$ such that $U(\gamma^{(0)}(t)) \neq 0$, the above equation gives
\begin{equation*}
    -\frac{1}{2} \frac{\dif }{\dif t} \log |U(\gamma^{(0)}(t))| = \frac{\ddot \gamma^{(0)}(t) \cdot \dot \gamma^{(0)}(t)}{|\dot \gamma^{(0)}(t)|^2}.
\end{equation*}
Since the same equation with $U$ replaced by $\Lambda$ also holds, we have $\frac{\dif }{\dif t} \log |U(\gamma^{(0)}(t))| = \frac{\dif }{\dif t} \log |\Lambda(\gamma^{(0)}(t))|$.
It follows that $U(\gamma^{(0)}(t))$ is a constant multiple of $\Lambda(\gamma^{(0)}(t))$ over each connected component of $\{t \mid U(\gamma^{(0)}(t)) \neq 0\}$.
Actually, since $\Lambda > 0$, by continuity of $\Lambda$ and $U$ there exists a constant $c $ such that $U(\gamma^{(0)}(t)) = c \Lambda(\gamma^{(0)}(t))$ for all $t$.
Now, using the first-order condition of \eqref{eq:E reduced}, we have $E^{(1)}(\gamma^{(0)}) = 0$.
Since $\gamma^{(0)}$ has unit-speed with respect to $g_0$, we obtain 
\begin{equation}\label{eq:c=0}
    0 = E^{(1)}(\gamma^{(0)}) = \int_0^T U(\gamma^{(0)}(t)) |\dot \gamma^{(0)}(t)|^2 \dif t = c\int_0^T \Lambda(\gamma^{(0)}(t)) |\dot \gamma^{(0)}(t)|^2 \dif t
        = c T.
\end{equation}
Since $T > 0$, the above equation forces $c = 0$ and therefore $U$ vanishes on the image of $\gamma^{(0)}$.
Finally, since the family of closed geodesics $\gamma^{(0)}$ of a rational torus foliates $\T^2$, we can apply the same argument to every such geodesic $\gamma^{(0)}$ to deduce that $U = 0$ everywhere. 
The proof of Theorem~\ref{thm:RT rig} is finished.

\begin{rema}
    We emphasize that the preceding argument requires the deformation to preserve not only a rational torus itself, but also the \emph{length} of the closed geodesics in the rational torus (guaranteed by the isospectral assumption).
    The condition for preserving the length is necessary: one can of course deform the metric nontrivially within the class of Liouville metrics, however at the cost of nontrivially changing the length spectrum.
\end{rema}

\subsection*{A brief review of the calculus of variations}\label{sec:var calc}

We recall some basic notions in the calculus of variations applied to geodesics. Our exposition loosely follows Milnor's book~\cite{milnor1963morse}.

\medskip

Let $(X,g)$ be a Riemannian manifold.
Let $T>0$ and let $\omega: [0,T] \to X$ be a piecewise smooth curve, i.e., there exist $n \geq 1$ and a subdivision
\[
0 = t_0 < t_1 < \cdots < t_{n-1} < t_n = T
\]
such that $\omega|_{[t_j, t_{j+1}]}$ is smooth for $j = 0,\ldots, n-1$.
We define the \emph{energy} of $\omega$ associated with $g$ by
\begin{equation*}    
    E(\omega) := \int_0^T |\dot\omega(t)|_g^2 \,\dif t.
\end{equation*}
A \emph{variation vector field} along $\omega$ is a continuous and piecewise smooth section
\[
\delta: [0,T] \to \omega^*TX
\]
of the pull-back bundle $\omega^*TX$.
Explicitly, $\delta(t) \in T_{\omega(t)}X$ for each $t \in [0,T]$, and the restriction $\delta|_{[t_j,t_{j+1}]}$ is smooth for every $j$.
We say that $\delta$ \emph{fixes the endpoints} if $\delta(0)=0$ and $\delta(T)=0$.

\medskip

Given such a variation vector field $\delta$, we define the \emph{first variation} of the energy functional at $\omega$ in the direction $\delta$ by
\begin{equation}\label{eq:first-variation}
    \dif_\omega E(\delta)
    :=
    \frac{\dif}{\dif\eta} E\bigl(\alpha(\,\cdot\,,\eta)\bigr)\Big|_{\eta=0},
\end{equation}
where $\alpha : [0,T] \times (-\eta_0,\eta_0) \to X$ is any map such that
\[
    \alpha(t,0)=\omega(t),
    \qquad
    \frac{\partial \alpha}{\partial \eta}(t,0)=\delta(t)
    \quad \text{for all } t \in [0,T].
\]
One checks that the value of \eqref{eq:first-variation} does not depend on the choice of $\alpha$.

\medskip

Suppose \(\omega\) is a critical point of \(E\), i.e., $\dif_\omega E \equiv 0$. Then, given two variation vector fields $\delta_1,\delta_2$ along $\omega$, we define the corresponding \emph{second variation} of $E$ by
\begin{equation}\label{eq:second-variation}
    \dif_\omega^2 E(\delta_1,\delta_2)
    :=
    \frac{\partial^2}{\partial \eta_1\,\partial \eta_2}
    E\bigl(\alpha(\,\cdot\,,\eta_1,\eta_2)\bigr)\Big|_{\eta_1=\eta_2=0},
\end{equation}
where $\alpha : [0,T] \times (-\eta_0,\eta_0)^2 \to X$ is any map satisfying
\[
\alpha(t,0,0)=\omega(t), \qquad
\frac{\partial \alpha}{\partial \eta_1}(t,0,0)=\delta_1(t), \qquad
\frac{\partial \alpha}{\partial \eta_2}(t,0,0)=\delta_2(t).
\]
One checks using the criticality of $\omega$ that~\eqref{eq:second-variation} is independent of the choice of the two-parameter variation $\alpha$.
These notions give the usual characterization of geodesics; see, for example,~\cite[\S\S 12--13]{milnor1963morse}.

\begin{theo}\label{thm:variation of geod}
    The curve $\omega$ is a geodesic if and only if $\dif_\omega E(\delta)=0$ for every variation vector field $\delta$ along $\omega$ fixing the endpoints.
    In addition, if $\omega$ is a length-minimizing geodesic segment, then the quadratic form
    \[
        \delta \mapsto \dif_\omega^2 E(\delta,\delta)
    \]
    is positive semidefinite on the space of variation vector fields fixing the endpoints.
\end{theo}

\section{Integrable deformations of Liouville metrics: Proof of Theorem \ref{thm:RI implies liouv}} \label{sec:RIliouv}
 We start with proving part (i) of Theorem \ref{thm:RI implies liouv}. By length isospectrality and application of the well known Maupertuis principle (discussed in Section~\ref{subsubsec:Maupertuis}) we have that the length functional of perturbed geodesics of any fixed homology class $(m_1, m_2) \in \Z^2$ with $m_1 m_2 \neq 0$ is constant in $\epsilon$, i.e.
 \begin{equation} \label{eq:lengthorderbyorder}
\epsilon \mapsto \ell^{(m_1, m_2)}(\epsilon) := \int_0^{1}   \left( 
V_\epsilon(\gamma_\epsilon(t))\right)^{1/2} |\dot{\gamma_\epsilon}(t)| \, \dif t  = \const 
\end{equation}
Here and in the following, we parametrize all geodesics by a fixed interval $[0,1]$, after rescaling the unit-speed geodesic produced by Maupertuis, so that the relevant quantity is the length functional in fixed-endpoint form.

\medskip

In \eqref{eq:lengthorderbyorder}, $\gamma_\epsilon$ is the perturbed geodesic of homology class $(m_1, m_2) \in \Z^2$ with $m_1 m_2 \neq 0$. By the asumed analyticity of $f_1, f_2, U$ and the persistence of the rational torus, the corresponding family of geodesics depends, by the implicit function theorem, analytically on the deformation parameter $\eps$ and we can hence write
\begin{equation}
    \gamma_\epsilon = \sum_{j \ge 0} \epsilon^j \gamma^{(j)} 
\end{equation}
where each $\gamma^{(j)}$ is a smooth curve in $\T^2$. Moreover, we abbreviated 
\begin{equation*}
    V_\epsilon(x_1, x_2) = 1 + f_1(x_1) + f_2(x_2) + \sum_{j \ge 1} \epsilon^j U^{(j)}(x_1, x_2) \,. 
\end{equation*}

To prove Theorem \ref{thm:RI implies liouv}, we analyze \eqref{eq:lengthorderbyorder} inductively to show that each $U^{(n)}$ is separable as in \eqref{eq:separablemain}. The key to this is showing that, at any order $n$,  we have
\begin{equation}
    \int_{0}^{1} U^{(n)}(\gamma_0(t)) \dif t = \const \,, 
\end{equation}
 independent of the initial position of the unperturbed geodesic $\gamma_0$. 

\medskip

At first order, the base case of our induction, since $\gamma_0$ is a critical point of the length functional corresponding to the \emph{unperturbed} metric, we have the first order relation
\begin{equation} \label{eq:firstvariation}
    \int_0^1 U^{(1)}(\gamma_0(t)) \dif t  = 0 \,, 
\end{equation}
i.e., in particular, the integral of the first order coefficient of the perturbation along a geodesic in an (unperturbed) rational torus vanishes is constant, irrespective of the geodesic.

\medskip

Next, we rewrite \eqref{eq:firstvariation} for a geodesic $\gamma_0$ in the homology class $(m_1,m_2) \in \Z^2$ with $m_1 m_2\neq 0$ and initial condition $x_0 = (x_{0,1},x_{0,2}) \in \T^2$. In order to do so, we again employ the Maupertuis principle and parametrize $\gamma_0$ by $t \in [0,1]$ using \emph{action-angle coordinates} (see \eqref{eq:AA} below) for the associated (integrable) natural mechanical Hamiltonian $\widetilde{H}$. 

\medskip

Then \eqref{eq:firstvariation} takes the form\footnote{With a slight abuse of notation, we will use $x_i$ as both a variable and a function.} 
\begin{equation} \label{eq:firstvariationcanon}
    \int_0^1 U^{(1)}\big(x_1(\theta_{0,1} + m_1t), x_2(\theta_{0,2} + m_2t) \big) \dif t = \const \,, 
\end{equation}
where $\theta_{0,i} \in \T$ is such that $x_i(\theta_{0,i}) = x_{0,i}$ for $i=1,2$, and the functions $x_i \in C^1(\T)$ are defined as follows: First, note that there exists a unique decomposition $e_1 + e_2 = 1$ with $e_i \in (0,1)$ such that 
\begin{equation} \label{eq:length}
    |m_1| \int_{0}^{1} \frac{1}{\sqrt{e_1 + f_1(x_1)}} 		\mathrm{d} x_1 = |m_2| \int_{0}^{1} \frac{1}{\sqrt{e_2+ f_2(x_2)}} 		\mathrm{d} x_2 \,. 
\end{equation}
Then, for $i=1,2$, $x_i \in C^1(\T)$ is defined as the inverse function of
\begin{equation} \label{eq:AA}
    x_i \mapsto \theta_i(x_i) = \mathrm{sgn}(m_i) \frac{\int_0^{x_i} \dif x_i'\big(e_i + f_i(x_i')\big)^{-1/2} }{\int_0^{1} \dif x_i'\big(e_i + f_i(x_i')\big)^{-1/2}} \,. 
\end{equation}
We point out that the $\theta_i$'s are the \emph{angle coordinates} corresponding to the \emph{action coordinates} $I_i = \mathrm{sgn}(m_i) \int_0^1 \sqrt{2(e_i + f_i(x_i))} \dif x_i $ near the rational torus, in which $\gamma_0$ lies. 

\medskip

To show that $U^{(1)}$ is separable (as in \eqref{eq:separablemain}), the rest of the argument closely follows \cite[Proof of Theorem III in Section 4.3]{henheik2025deformational}: This identifies the vanishing of the first-variation integral along all rotational orbits as in \eqref{eq:firstvariation} with the vanishing of the off-diagonal Fourier modes of $U^{(1)}$, and hence with separability. 

\medskip

In fact, our first variation formula \eqref{eq:firstvariationcanon} agrees with \cite[Eq.~(4.15)]{henheik2025deformational} for $\mu_i = 1$. More precisely, in \cite{henheik2025deformational} the functions $f_i$ were replaced as $f_i \to \mu_i f_i$ for some $\mu_i >0$. Then, adopting this rescaling to our setting, part (c) in step (iii) of the proof of Theorem III in \cite{henheik2025deformational} shows the following: Given \eqref{eq:firstvariationcanon} for all initial positions $x_0 \in \T^2$ and homology classes $(m_1, m_2) \in \Z^2$ with $m_1 m_2 \neq 0$, we can conclude that $U^{(1)}$ is separable as in \eqref{eq:separablemain} -- given that $(\mu_1, \mu_2)$ lies in a \emph{generic set}. In fact, the set on which this conclusion might be \emph{invalid}, is the zero set of a non-constant analytic function of $(\mu_1, \mu_2)$. We hence deduce that, for generic analytic $f_1, f_2$, a rationally integrable deformation of a Liouville metric by a trigonometric polynomial preserving lengths of invariant tori as in \eqref{eq:deformLioutrig} is necessarily separable. This concludes the proof of Theorem \ref{thm:RI implies liouv} in the case that $f_1, f_2 \not\equiv \mathrm{const}$.

\medskip

If $f_1 \equiv 0$, say (a non-zero constant can be absorbed into a redefined $f_2$), we follow \cite[Proof of Theorem II in Section 4.2]{henheik2025deformational} instead of \cite[Proof of Theorem III in Section 4.3]{henheik2025deformational}. The reason why one can allow a general dependence of $U$ on $x_1$ is that the functions $x_1 \equiv \vartheta_1$ agree, and hence  $\left\{\vartheta_{0,1} \mapsto \mathrm{e}^{\mathrm{i} 2 \pi k_1 \vartheta_{0,1}} \right\}_{k_1 \in \Z}$ forms an orthonormal basis of $L^2(\T)$, cf.~\cite[Step (iii) of the proof of Theorem II in Section 4.2]{henheik2025deformational}. 

\medskip

We thus conclude that
\begin{equation*}
    U^{(1)}(x_1, x_2) = U^{(1)}_1(x_1) + U^{(1)}_2(x_2) \,, 
\end{equation*}
i.e.~separability of $U^{(1)}$.

\medskip

We now perform the induction step. Take as the induction hypothesis that $U^{(1)}, ... , U^{(n-1)}$ are separable. We now want to prove that 
\begin{equation}     \label{eq:cancel}
   \int_0^1 U^{(n)}(\gamma_0(t)) \dif t = \const
\end{equation}
independent of the initial position of the unperturbed geodesic for any $n \ge 2$, from which, following the argument above for $U^{(1)}$, we immediately deduce separability of $U^{(n)}$.

\medskip

By differentiating the geodesic equation order by order, the $n$-th correction $\gamma^{(n)}$ enters the $n$-th variation only through terms that vanish because the unperturbed curve is critical; therefore the $n$-th coefficient depends on $U^{(n)}$
  plus terms determined entirely by lower-order data.
Hence, expanding \eqref{eq:lengthorderbyorder} in $\eps$ to order $n$, we deduce the (vanishing) $n$-th order coefficient to be given by
\begin{equation} \label{eq:highergen}
    \int_0^1 \dif t \, U^{(n)}(\gamma_0(t)) + (\text{terms involving } U^{(j)} \text{ and } \gamma^{(j)} \text{ for } j \le n-1)= 0. 
\end{equation}
Since the lower-order terms are built from separable potentials (by the induction hypothesis) and their associated averages, they are independent of the initial phase. Hence, we deduce that the first term in \eqref{eq:highergen} is a constant independent of the initial condition of the unperturbed geodesic $\gamma_0$. The same argument as in the base case thus applies verbatim and we obtain separability of $U^{(n)}$. This concludes the induction step and thus the proof of Theorem \ref{thm:RI implies liouv}~(i).

\medskip

For part (ii) of Theorem \ref{thm:RI implies liouv}, we can simply follow the steps in the proof of part (i), noting that, at any fixed order $n$, the argument from \cite{henheik2025deformational} requires only \emph{finitely} many rational invariant tori to be preserved, whose maximal length we denote by $T_n$.\qed

\section{Length-isospectral families of Liouville metrics: Proof of Theorem~\ref{thm:liourearrange}}
\label{sec:length-liouville}
In this section, we give the proof of Theorem~\ref{thm:liourearrange} that characterizes length-isospectral deformations within the class of Liouville metrics using ``rearrangments''.
Theorem~\ref{thm:liourearrange} follows from the more general, non-deformative results in Theorem~\ref{thm:rearrangement-rotMLS}, which asserts that, roughly speaking, Liouville metrics on $\T^2$ are determined by their marked length spectra up to rearrangement. 

\medskip

More precisely, for a Riemannian metric $g$ on $\T^2$, we recall the notation $\CL(\T^2,g)$ for the unmarked length spectrum and $\CL^{m,n}(\T^2,g)$ for the set of lengths of closed geodesics in the homology class $(m,n) \in \Z^2$.
For $m,n$ not both zero, we set 
\[
    \ell^{m,n}_g := \min \CL^{m,n}(\T^2,g),
\]
and define the \textbf{marked length spectrum} to be the map
\[
(m,n)\quad \longmapsto \quad  \ell^{m,n}_g.
\]
We now define the \textbf{rotational marked length spectrum} to be its restriction to
$(m,n)\in(\Z \setminus \{0\})^2$.

\medskip

For a measurable function $f:\T\to\R$, we define its \textbf{superlevel measure} to be
\begin{equation}\label{eq:superlevel-measure}
    L_f(t) := \leb\{x\in\T : f(x)>t\}, \qquad t\in\R,
\end{equation}
and we say that $\tilde f$ is a \emph{rearrangement} of $f$ if
$L_{\tilde f} \equiv L_f$. 
Although the notion of rearrangements applies to measurable functions, we shall nevertheless assume that, in the rest of this section, all functions defined on $\T$ are at least $C^2$.
We now state the central result of this section. 

\begin{theo}\label{thm:rearrangement-rotMLS}
Let $g$ and $\wt g$ be Liouville metrics on $\T^2$ with conformal factors $1+f_1(x_1)+f_2(x_2)$ and $1+\wt f_1(x_1)+\wt f_2(x_2)$, respectively.

Then the following are equivalent:
\begin{enumerate}
  \item The metrics $g$ and $\wt g$ have the same rotational marked length
        spectrum, i.e.
        \[
        \ell^{m,n}_g = \ell^{m,n}_{\wt g} \qquad
        \text{for all } (m,n)\in (\Z \setminus \{0\})^2.
        \]
  \item There exists a constant $c\in\R$ such that
        $\wt f_1$ is a rearrangement of $f_1+c$ and $\wt f_2$ is a
        rearrangement of $f_2-c$.
\end{enumerate}

Moreover, if each of the functions $f_1,f_2,\wt f_1,\wt f_2$ has exactly two
critical points on $\T$, then either (and hence both) of the above conditions
 implies that $\CL(\T^2,g) = \CL(\T^2,\wt g)$, i.e., the metrics $g$ and $\wt g$ are length-isospectral.
\end{theo}

We can now easily conclude the proof of Theorem~\ref{thm:liourearrange}. 

\begin{proof}[Proof of Theorem~\ref{thm:liourearrange}]
    Let $(g_\eps)$ be the family of Liouville metrics in the statement of the theorem. 
The assumption that the lengths of all rational tori are preserved for the family $(g_\eps)$ implies that the rotational marked length spectrum of $g_\eps$ is independent of $\eps$. The conclusion of Theorem~\ref{thm:liourearrange} now follows immediately from the implication (1) $\implies$ (2) in Theorem~\ref{thm:rearrangement-rotMLS}, applied to the pairs $(f_1(0, \cdot), f_2(0, \cdot))$ and $(f_1(\eps, \cdot), f_2(\eps, \cdot))$.
\end{proof}

\begin{remark}[Length spectrum and Laplace spectrum] 
As discussed in Section \ref{subsec:discuss}, we believe that the length isospectral metrics constructed here and in Example \ref{ex:two rivers} are \emph{not} Laplace isospectral.
    In fact, in the Laplace spectral literature, constructions like that in Example~\ref{ex:two rivers} are not new. For example, Zelditch presents a construction in \cite{Zelditch3}: Consider the round metric on $\mathbb{S}^2$ and add two disjoint bumps. Translating these bumps around $\mathbb{S}^2$ clearly results in nonisometric metrics. The Laplace spectrum may very well change, but the \textit{heat invariants} do not. The heat invariants are coefficients in the expansion
\begin{align}\label{eq:heat trace}
    \tr e^{t \Delta} = \sum_j e^{- t \lambda_j^2} \sim t^{-d/2} \sum_{k = 0}^\infty a_k t^{k}, \quad \text{as}\,\, t \to 0.
\end{align}
It is known that these heat invariants have the form
\begin{align}
    a_k = \int_{\mathbb{S}^2} \Theta_k[g] \,\dif  \vol_g,
\end{align}
for some universal polynomials in the curvature of the metric $g$, together with finitely many (depending on $k$) of its covariant derivatives. It might seem strange that identical heat invariants do not necessarily yield identical Laplace spectra $\lambda_j^2$, but the key point is that the right-hand side of \eqref{eq:heat trace} is not in general a convergent sum, even after renormalizing by a power of $t$. Equality of $\tr e^{t \Delta}$ for all $t$ would indeed imply isospectrality, but its Taylor coefficients at $t = 0$ do not.

\end{remark}

\subsection{Length formulae of closed geodesics of Liouville metrics.}

Before starting the proof of Theorem~\ref{thm:rearrangement-rotMLS} proper, let us record some useful formulae of closed geodesics of Liouville metrics.
These elementary formulae, collected in Lemma~\ref{lem:length-liouville}, will be used both in the proof of Theorem~\ref{thm:rearrangement-rotMLS} and later on in the proof of Proposition~\ref{prop:NCC gen} in the appendix.

\medskip

We shall begin by translating the geodesic equations for a Liouville metric into a separable mechanical Hamiltonian~\eqref{eq:H separable} and the associated one-dimensional Hamiltonian equations via the \emph{Maupertuis principle} (cf.~Section \ref{subsubsec:Maupertuis}).

\medskip

Recall that the Hamiltonian~\eqref{eq:H separable} obtained from the Maupertuis principle (with total energy $= 1$) is the sum of two Hamiltonians $H_1 + H_2$ with $H_i(x_i, \xi_i) = \xi_i^2/2 - f_i(x_i)$.
The Hamiltonian flow on $\T \times \R$ generated by an individual Hamiltonian $H_i$ is described by the following ODE.
\begin{equation} \label{eq:separableODE}
    \begin{cases}
        x_i'(t) = \xi_i(t) \\
        \xi_i'(t) = f_i'(x_i(t))
    \end{cases} \qquad \text{for} \quad i = 1,2.
\end{equation}
Hence, the projections of orbits of \eqref{eq:H separable} onto $\T^2$ are solutions of the uncoupled ODEs
\begin{equation}\label{eq:Maup ODE}
    \begin{cases}
        x_1'(t) = \pm \sqrt{2(e + f_1(x_1(t)))} \\
        x_2'(t) = \pm \sqrt{2(1 - e + f_2(x_2(t)))}
    \end{cases} 
\end{equation}
where $e \in \R$ is a conserved quantity (with $\xi_1^2/2 - f_1(x_1) = e$).

\medskip

In particular, for geodesics in a $(m_1, m_2)$-rational torus, the corresponding periodic solution $(x_1(t), x_2(t))$ of \eqref{eq:Maup ODE} satisfies $x_j'(t) \neq 0$ and $x_j(\wt T) - x_j(0) = m_j$ (for any lift to the universal cover), where
\begin{equation}\label{eq:Maupertuis period}
    \wt T = |m_1|\int_{\T} \frac{\dif x_1}{\sqrt{2 (e + f_1(x_1))}} = |m_2|\int_{\T} \frac{\dif x_2}{\sqrt{2 ( 1 - e + f_2(x_2))}} 
\end{equation}
is the period of the solution $(x_1(t), x_2(t))$ to the ODE~\eqref{eq:Maup ODE}.
The relation~\eqref{eq:Maupertuis period} uniquely determines $e \in \R$ in terms of $(m_1, m_2)$.

\medskip

On the other hand, for the periodic solution $(x_1(t), x_2(t))$ of \eqref{eq:Maup ODE} corresponding to a geodesic in the homology class $(0, m)$, the coordinate $x_1(t)$ oscillates through a connected component $[x_1^{\min}, x_1^{\max}]$ of $\{x \in \T \mid e + f_1(x) > 0\}$.
In this case, either $x_1^{\min} = x_1^{\max}$, forcing $x_1'(t) = 0$ and thus $f_1'(x_1)=0$ by~\eqref{eq:separableODE}; or else $x_1^{\min} < x_1^{\max}$ and we have, similar to~\eqref{eq:Maupertuis period}, the time-matching condition
\[
\wt T = 2n \int_{x_1^{\min}}^{x_1^{\max}} \frac{\mathrm{d}x_1}{\sqrt{e+f_1(x_1)}}=|m|
 \int_{\T} \frac{\mathrm{d}x_2}{\sqrt{1-e+f_2(x_2)}}
\]
for the number $n\in\Z_{>0}$ of oscillations of $x_1(t)$ between $x_1^{\min}$ and $ x_1^{\max}$ over one period $\wt T$.
An analogous formula holds, of course, for homology class $(m,0)$.

\medskip

Next, we calculate the length and the energy of these geodesics.
Let us start with a general solution $(x_1(t), x_2(t))$ of \eqref{eq:Maup ODE} with some choice of $e$.
From the definition and~\eqref{eq:Maup ODE}, the Riemannian length of such orbits over an arbitrary time interval $[0,\wt T]$ is
\begin{equation*}\begin{aligned}
    &\int_{0}^{\wt T} \sqrt{(1 + f_1(x_1(t)) + f_2(x_2(t))) (x_1'(t)^2 + x_2'(t)^2)} \dif t\\
    =&\int_{0}^{\wt T} \sqrt{2}(1 + f_1(x_1(t)) + f_2(x_2(t)))  \dif t\\
    =&\int_{0}^{\wt T} \sqrt{2}(e + f_1(x_1(t)))  \dif t+ 
    \int_{0}^{\wt T} \sqrt{2}(1 - e + f_2(x_2(t)))  \dif t.
\end{aligned}
\end{equation*}
Suppose for example $x_j'(t) > 0$. 
Then the changes of variables
\[
\frac{\dif x_1}{\sqrt{2(e+f_1(x_1))}}=\dif t,
\qquad
\frac{\dif x_2}{\sqrt{2(1-e+f_2(x_2))}}=\dif t
\]
give
\begin{equation}\label{eq:Maup T-length 1}
    \int_{0}^{\wt T}\sqrt{2}\,\bigl(e+f_1(x_1(t))\bigr)\,\dif t
=\sqrt{2}\int_{x_1(0)}^{x_1(\wt T)} \frac{e+f_1(x_1)}{\sqrt{2(e+f_1(x_1))}}\,\dif x_1
=\int_{x_1(0)}^{x_1(\wt T)} \sqrt{e+f_1(x_1)}\,\dif x_1.
\end{equation}
Similarly,
\begin{equation}\label{eq:Maup T-length 2}
    \int_{0}^{\wt T}\sqrt{2}\,\bigl(1-e+f_2(x_2(t))\bigr)\,\dif t=\int_{x_2(0)}^{x_2(\wt T)}\sqrt{1-e+f_2(x_2)}\,\dif x_2.
\end{equation}
It follows that the length of the orbit over $[0, \wt T]$ is given by 
\begin{equation*}
    \int_{x_1(0)}^{x_1(\wt T)} \sqrt{e+f_1(x_1)}\,\dif x_1 + \int_{x_2(0)}^{x_2(\wt T)}\sqrt{1-e+f_2(x_2)}\,\dif x_2.
\end{equation*}
If $(x_1(t), x_2(t))$ is a periodic solution lying in the homology class $(m_1, m_2) \in (\Z\setminus\{0\})^2$ with period $\wt T$, then its Riemannian length over one period is simply 
\begin{equation*}    |m_1|\int_{\T}\sqrt{e+f_1(x_1)}\,\dif x_1 + |m_2|\int_{\T}\sqrt{1-e+f_2(x_2)}\,\dif x_2
\end{equation*}
where $e$ is determined by~\eqref{eq:Maupertuis period}.
Analogous statements hold, of course, for homology $(m,0)$ and $(0,m)$.
We now summarize.

\begin{lemm}\label{lem:length-liouville}
    Let $g = (1 + f_1(x_1) + f_2(x_2)) (\dif x_1^2 + \dif x_2^2)$ be a Liouville metric. Then:
    \begin{itemize}
        \item For each $(m_1, m_2) \in (\Z\setminus\{0\})^2$, there exists a unique 
        $e \in (-\min f_1, 1 + \min f_2)$  satisfying 
        \begin{equation} \label{eq:m1m2 rot}
            |m_1|\int_0^1 \frac{\dif x_1}{\sqrt{e + f_1(x_1)}} = 
            |m_2|\int_0^1 \frac{\dif x_2}{\sqrt{1 - e + f_2(x_2)}}
        \end{equation}
        such that the length of every closed geodesic in an $(m_1, m_2)$-rational torus is 
        \begin{equation} \label{eq:rot length}
            |m_1| \int_\T \sqrt{e + f_1(x_1)} \dif x_1 + |m_2| \int_\T \sqrt{1 - e + f_2(x_2)} \dif x_2.
        \end{equation}
        \item For $m \in \Z\setminus\{0\}$, every closed geodesic in the homology class $(0,m)$ belongs to one of the two following types:
        \begin{itemize}
            \item There exists a critical point $x_1 \in \T$ of $f_1$ such that the closed geodesic is contained in $\{x_1\} \times \T$. 
            In this case, the length of the closed geodesic is given by 
            \begin{equation}  \label{eq:stat length}
                |m|\int_\T \sqrt{1 + f_1(x_1) + f_2(x_2)} \dif x_2
            \end{equation}
            \item There exists $n \in \Z_{>0}$ and $e \in (- \max f_1, -\min f_1)$ satisfying
            \begin{equation} \label{eq:m1m2 osc}
                2n\int_{x_1^{\min}}^{x_1^{\max}} \frac{\dif x_1}{\sqrt{e + f_1(x_1)}} = 
                |m|\int_\T \frac{\dif x_2}{\sqrt{1 - e + f_2(x_2)}}
            \end{equation}
            where $(x_1^{\min}, x_1^{\max})$ is a connected component of $\{x \mid e + f_1(x) > 0\}$.
            In particular, $e$ is not a critical value of $-f_1$.
            In this case, the closed geodesic is contained in $[x_1^{\min}, x_1^{\max}] \times \T$ and its length is given by 
            \begin{equation} \label{eq:osc length}
                2n\int_{x_1^{\min}}^{x_1^{\max}} \sqrt{e + f_1(x_1)} \dif x_1 + |m| \int_\T \sqrt{1 - e + f_2(x_2)} \dif x_2.
            \end{equation}
        \end{itemize}
        \item The symmetric statement holds for closed geodesics in the homology class $(m,0)$.
    \end{itemize}
\end{lemm}

\begin{remark}
    We also observe from above that there are no contractible closed geodesics for a Liouville metric, other than the trivial cases of `geodesics with length $0$'.
\end{remark}

The rest of the section is devoted to proving
Theorem~\ref{thm:rearrangement-rotMLS}.

\subsection{Proof of Theorem~\ref{thm:rearrangement-rotMLS}.}

For a $C^2$ function $f: \T \to \R$, define 
\begin{equation}\label{eq:J def}
    J_f(e) := \int_{\{x \in \T \ : \ e + f(x) > 0\}} \sqrt{e + f(x)} \dif x, \qquad (e \in \R).
\end{equation}
Clearly $J_f$ is strictly increasing, strictly concave and analytic over the interval $(-\min f, \infty)$. 
We emphasize that the analyticity holds even if $f$ may not be analytic.
Consider $C^2$ functions $f_j: \T \to \R$ for $j = 1,2$ such that $\min f_1 + \min f_2 > -1$ and define 
\begin{equation}\label{eq:Gam def}
\begin{aligned}
    \Gamma_{f_1, f_2} =& \{(J_{f_1}(e), J_{f_2}(1 - e)) \ : \  -\min f_1 < e < 1 + \min f_2\}\\
    \Gamma^+_{f_1, f_2} =& \{(J_{f_1}(e), J_{f_2}(1 - e)) \ : \  1 + \min f_2 < e < 1 + \max f_2\}\\
    \Gamma^-_{f_1, f_2} =& \{(J_{f_1}(e), J_{f_2}(1 - e)) \ : \  -\max f_1 < e < - \min f_1\}.
\end{aligned}
\end{equation}
Then $\Gamma_{f_1, f_2}$ is a real-analytically embedded curve in the first quadrant of $\R^2$ and it is the graph of a strictly decreasing concave function. 
The end points of $\Gamma_{f_1, f_2}$ are 
\begin{equation}\label{eq:Gam ends}
\begin{gathered}
    \Big(\int_\T \sqrt{f_1(x) - \min f_1} \dif x,\ \int_\T \sqrt{1 + \min f_1 + f_2(x)} \dif x\Big) \quad \text{and}\\
    \Big(\int_\T \sqrt{1 + \min f_2 + f_1(x)} \dif x,\ \int_\T \sqrt{f_2(x) - \min f_2} \dif x \Big).
\end{gathered}
\end{equation}
We refer to Figure~\ref{fig:Gamma} for an illustration.

\medskip

Let $\mu_f$ denote the measure on $\R$ given by the push-forward $f_* \leb_{\T}$ of the Lebesgue measure on $\T$ by $f: \T \to \R$. Naturally $\mu_f$ is supported within the interval $[\min f, \max f]$. 
Recall that $L_f$ has been defined in \eqref{eq:superlevel-measure}.
The distribution function $L_f$ determines the push-forward measure $\mu_f$, and conversely. Clearly, if $\wt f$ is a rearrangement of $f$ then $\mu_f = \mu_{\wt f}$. 

\medskip

By the Layer-cake representation (see for example~\cite[Ch. 1.13]{LiebLoss}) and the definition of $\mu_f$,
\begin{equation}\label{eq:layer-cake J}
	J_f(e) = \frac{1}{2} \int_{-e}^\infty \frac{L_f(t)}{\sqrt{t + e}} \dif t = \int_{-e}^\infty \sqrt{t + e} \;\dif \mu_f(t).
\end{equation}
Clearly, if $\wt f$ is a rearrangement of $f$, then \eqref{eq:layer-cake J} shows that $J_f = J_{\wt f}$ as functions. 
It follows that the curves \eqref{eq:Gam def} are invariant under rearrangements of $f_1$ and $f_2$.
Since $J_{f+c}(e) = J_f(e + c)$ for $c \in \R$ by definition of $J_f$, it is clear that the curves \eqref{eq:Gam def} are also invariant (as subsets of $\R^2$) under the transformation $(f_1, f_2) \mapsto (f_1 + c, f_2 - c)$.
We show a converse of this observation.
\begin{lemm}\label{lem:Gam determines f12}
	Let $(\wt f_1, \wt f_2)$ and $(f_1, f_2)$ be two pairs of $C^2$ functions $\T \to \R$ with $\min f_1 + \min f_2 > -1$ and $\min \wt f_1 + \min \wt f_2 > -1$.
	Then $\Gamma_{f_1, f_2} = \Gamma_{\wt f_1, \wt f_2}$ if, and only if there exists $c \in \R$ such that $\wt f_1$ is a rearrangement of $f_1 + c$ and $\wt f_2$ is a rearrangement of $f_2 - c$.
\end{lemm}
\begin{proof}
	We have observed the ``if'' part. 
	For the ``only if'' part, suppose $\Gamma_{\wt f_1, \wt f_2} = \Gamma_{f_1, f_2}$ as subsets of $\R^2$ for some $C^2$ functions $f_1, f_2, \wt f_1$ and $\wt f_2$ satisfying the conditions of the lemma.
	Modulo transformations of the type $(f_1, f_2) \mapsto (f_1 + c, f_2 - c)$, we may assume that $\min f_1 = \min \wt f_1 = 0$.
	To show the lemma it is enough to prove that $\wt f_j$ is a rearrangement of $f_j$ for $j = 1,2$.

\medskip

    We begin by matching the endpoints \eqref{eq:Gam ends} of the curves $\Gamma_{\wt f_1, \wt f_2}$ and $\Gamma_{f_1, f_2}$, which in particular implies
    \begin{equation*}        \int_{\T} \sqrt{f_j(x) - \min f_j}\ \dif x = 
        \int_{\T} \sqrt{\wt f_j(x) - \min \wt f_j}\ \dif x \qquad (j = 1,2).
    \end{equation*}
	Since $J_f(e)$ is real-analytic and strictly increasing for $e > -\min f$, it follows that 
    \begin{equation*}        J_{\wt f_j}\inv \circ J_{f_j} : (-\min f_j, \infty) \to  (-\min \wt f_j, \infty) \qquad (j = 1,2)
    \end{equation*}
    are real-analytic diffeomorphisms.
    Define 
    \begin{equation*}    \varphi(e) :=  \begin{cases}
            J_{\wt f_1}\inv \circ J_{f_1} (e) & (e > -\min f_1) \\
            1 - J_{\wt f_2}\inv \circ J_{f_2} (1 - e)  & (e < 1 + \min f_2).
        \end{cases}
    \end{equation*}
    Since $\Gamma_{f_1, f_2}$ projects injectively to $x_j$-coordinate for $j = 1,2$, we can verify that the two definitions of $\varphi(e)$ are consistent for $-\min f_1 < e < 1 + \min f_2$ and
    \begin{equation}\label{eq:J reparam}
        (J_{f_1}(e), J_{f_2}(1- e)) = (J_{\wt f_1}(\varphi(e)), J_{\wt f_2}(1- \varphi(e))) \qquad  (-\min f_1 < e < 1 + \min f_2).
    \end{equation}
    By analyticity of $J_{f_j}$ and $J_{\wt f_j}$ over their respective domains, we deduce that $\varphi: \R \to \R$ is a real-analytic diffeomorphism.
    We remark that $\varphi(1 + \min f_2) = 1 + \min \wt f_2$ and 
    \begin{equation*}    
        \varphi(0) = \varphi(-\min f_1) = - \min \wt f_1 = 0
    \end{equation*} 
 since we have assumed $\min f_1 = \min \wt f_1 = 0$.

\medskip

	Let us denote $f = f_1$ and $\wt f = \wt f_1$. Consider
	\begin{equation*}
		\wh J_{f}(e) = \int_0^\infty \sqrt{t + e} \; \dif \mu_f(t).
	\end{equation*}
	Choosing the principal branch of the square root, we define $\wh J_f$ analytically on the domain $\C \setminus (-\infty, 0]$.
	Observe that $J_f|_{(0, \infty)}$ is exactly the restriction of $\wh J_f$ on $(0, \infty)$.
	Hence $J_f$ determines $\wh J_f$  by analytic continuation.

\medskip

	Let $\mathbb{H}^\pm$ denote the upper/lower complex half-plane. 
	If $e \in \mathbb{H}^\pm$, then $\sqrt{t + e}$ lies in the first/fourth quadrant of $\C$ for any $t > 0$. It follows from convexity that $\wh J_f$ maps $\mathbb{H}^\pm$ into the first/fourth quadrant respectively.
	In fact, we show that $\wh J_f$ restricts to biholomorphisms
	\begin{equation}\label{eq:hat J bihol}
		\wh J_f: \mathbb{H}^\pm \xrightarrow{\sim} \CD^\pm,
	\end{equation}
	where $\CD^\pm$ are the unbounded domains in the first/fourth quadrant of $\C$ bounded by the curve 
	\begin{equation}\label{eq:D boundary}
		e \mapsto \lim_{\epsilon \to 0^\pm}\wh J_f (e + i \epsilon) = \int_{-e}^{\max f} \sqrt{t + e} \; \dif \mu_f(t) \pm i \int_{0}^{-e} \sqrt{|t + e|}\; \dif \mu_f(t).
	\end{equation}
	Notice that \eqref{eq:D boundary} parametrizes the image of the boundary of $\mathbb{H}^\pm$ under $\wh J_f$, and it is understood that the first integral is zero if $-e > \max f$ and similarly the second integral is zero for $-e < 0$.

\medskip

	To show \eqref{eq:hat J bihol}, we first observe that, since $\mu_f$ is a finite measure supported in the interval $[0, \max f]$, we have 
	\begin{equation}\label{eq:hat J asymptotic}
		|z|^{-1/2} |\wh J_f(z)|\to 1 \quad \text{uniformly as $|z| \to \infty$}
	\end{equation}
	Let $z_0 \in \CD^+$.
	Then, for $R$ sufficiently large, the point $z_0$ is enclosed by a closed loop formed by the image of the semicircle $\{z \in \mathbb{H}^+ : |z| = R\}$ under the function $\wh J_f$ and the boundary \eqref{eq:D boundary}.
	By shrinking $R$, the usual argument using contractible loops shows that $z_0 \in \wh J_f(\mathbb{H}^+)$. Hence $\wh J_f(\mathbb{H}^+) = \CD^+$.
	Similarly $\wh J_f(\mathbb{H}^-) = \CD^-$.
	To see the injectivity of $\wh J_f$, suppose $\wh J_f(z_1) = \wh J_f(z_2)$. Then
	\begin{equation*}
		0 = \wh J_f(z_1) -  \wh J_f(z_2) = \int_0^\infty \sqrt{t + z_1} - \sqrt{t + z_2} \dif \mu_f(t) = (z_1 - z_2) \int_0^\infty \frac{\dif \mu_f(t)}{\sqrt{t + z_1} + \sqrt{t + z_2}}.
	\end{equation*}
	Since $\Re \sqrt{t + z} > 0$ for all $t > 0$ and $z \in \C \setminus (-\infty, 0]$, it is easy to see that the last integral above does not vanish. It follows that $z_1 = z_2$ and hence $\wh J_f$ is injective.
	It is also clear that $\wh J_f$ has nonvanishing derivative everywhere. It follows that \eqref{eq:hat J bihol} are indeed biholomorphisms.
	Similarly, we have $\wh J_{\wt f}: \mathbb{H}^\pm \xrightarrow{\sim} \wt{\CD}^\pm$ where $\wt{\CD}^\pm$ are domains defined analogously with boundaries given by \eqref{eq:D boundary} with $f$ replaced by $\wt f$.
	
\medskip

	Recall that we have $J_f(e) = J_{\wt f}(\varphi(e))$ for $e > 0$ and $\varphi: \R \to \R$ real-analytic.
	Extending $\varphi$ analytically over a complex neighborhood $\CU \supset \R$, we have $\wh J_f = \wh J_{\wt f} \circ \varphi$ over $\CU \setminus (-\infty, 0]$.
	Shrinking $\CU$ if necessary we may assume $\varphi(\CU \cap \mathbb{H}^\pm) \subset \mathbb{H}^\pm$.
	It follows by taking the limit $\varphi(e + i \epsilon)\to \varphi(e)$ for $e \in \R$ that 
	\begin{equation*}		\lim_{\epsilon \to 0^\pm} \wh J_f(e + i \epsilon) = \lim_{\epsilon \to 0^\pm} \wh J_{\wt f}(\varphi(e) + i \epsilon).
	\end{equation*}
	Consequently, the boundary curves \eqref{eq:D boundary} defined by $f$ and by $\wt f$ are identical up to a reparametrization by $\varphi$ and hence both $\wh J_{\wt f}$ and $\wh J_f$ are biholomorphisms $\mathbb{H}^\pm \xrightarrow{\sim} \CD^\pm$ onto the same image.
	We can therefore extend the original definition of $\varphi$ to an analytic function $\varphi = \wh J_{\wt f}\inv \circ \wh J_f$  over $\C\setminus (-\infty, 0]$.
	Since $\varphi$ is real-analytic over $\R$, by the Schwarz reflection principle, it is an entire function.

\medskip

	Since $\varphi(0) = 0$, the function $z \mapsto \varphi(z) /z$ is entire and bounded by the asymptotics \eqref{eq:hat J asymptotic}.
	It then follows from Liouville's theorem that $\varphi(z) = az$ for some $a \in \R$.
	But \eqref{eq:hat J asymptotic} forces $a = 1$. Thus $\varphi(z) = z$ and we obtain accordingly $J_f = J_{\wt f}$ and $J_{f_2}(1 - e) = J_{\wt f_2}(1 - e)$. Since $L_f(e) = L_{\wt f}(e) = 1$ for $e < 0$, Equation \eqref{eq:layer-cake J} implies 
	\begin{equation*}		\int_{0}^\infty \frac{L_f(t)}{\sqrt{t + e}} \dif t = \int_{0}^\infty \frac{L_{\wt f}(t)}{\sqrt{t + e}} \dif t \quad \text{for all $e > 0$}.
	\end{equation*}
	Using the identity $\int_0^\infty s^{-1/2}\eu^{-(t + e) s} \dif s = \pi^{1/2} (e + t)^{-1/2}$, we have 
	\begin{equation*}
		\int_{0}^\infty \frac{L_f(t)}{\sqrt{t + e}} \dif t = \frac{1}{\sqrt{\pi}} \int_{0}^\infty \int_0^\infty  L_f(t) \frac{\eu^{-(e + t)s}}{\sqrt{s}} \dif s \dif t = \frac{1}{\sqrt{\pi}} \mathscr{L}_s \bigg[  \frac{\mathscr{L}_t[L_f(t)](s)}{\sqrt{s}}\bigg](e)
	\end{equation*}
	where $\mathscr{L}_t$  and  $\mathscr{L}_s$  denote the Laplace transform in the variable $t$ and $s$ respectively.
	By injectivity of Laplace transform, we have $L_f = L_{\wt f}$.
	By the same argument applied to $J_{f_2}(1 - e) = J_{\wt f_2}(1 - e)$ for $e < 1 + \min f_2$, we deduce $L_{f_2} = L_{\wt f_2}$. 
	Hence $\wt f_j$ is a rearrangement of $f_j$ as desired. 
	The proof is complete.
\end{proof}

We now prove Theorem~\ref{thm:rearrangement-rotMLS}.

\begin{proof}[Proof of Theorem~\ref{thm:rearrangement-rotMLS}]
	Consider the Liouville metrics $g$ and $\wt g$ given in the statement.
	The lengths of geodesics in an $(m_1, m_2)$-rational torus are described by Lemma~\ref{lem:length-liouville}. 
	We prove (1) $\iff$ (2) by reformulating Lemma~\ref{lem:length-liouville} using the geometric objects $\Gamma_{f_1, f_2}$ and $\Gamma_{f_1, f_2}^\pm$.
\\[2mm]
	\noindent \underline{\textit{Proof of (1) $\implies$ (2).}}
	By definition of $J_f$ and Lemma~\ref{lem:length-liouville}, for each pair of nonnegative integers $(m_1, m_2) \neq (0,0)$, there exists a unique $e \in [-\min f_1, 1 + \min f_2]$ such that 
    \begin{equation}\label{eq:J-tangent}
        m_1 J_{f_1}'(e) - m_2 J_{f_2}'(1-e) = 0
    \end{equation}
    and
    \begin{equation}\label{eq:J-length}
        \ell_g^{m_1, m_2} = m_1 J_{f_1}(e) + m_2 J_{f_2}(1-e).
    \end{equation}
    We note in particular that when $m_1 = 0$, then~\eqref{eq:J-length} follows from the formula~\eqref{eq:stat length} and vice versa.
    Therefore, the marked length spectrum $(m_1, m_2) \mapsto \ell_g^{m_1, m_2}$ completely determines the tangent lines of the curve $\Gamma_{f_1, f_2}$ with rational slope. See Figure~\ref{fig:Gamma}.
    By density of $\Q$ in $\R$ and strict concavity, these tangent lines determine $\Gamma_{f_1, f_2}$, which is just the envelope of these lines.
	It follows from (1) that $\Gamma_{f_1, f_2} = \Gamma_{\wt f_1, \wt f_2}$, which in turn implies (2) via Lemma~\ref{lem:Gam determines f12}.
\\[2mm]
	\noindent \underline{\textit{Proof of (2) $\implies$ (1).}}
    The strict concavity of $\Gamma_{f_1, f_2}$ implies that for each coprime pair of positive integers $(m_1, m_2)$ there is a unique tangent line orthogonal to $(m_1, m_2)$ so that \eqref{eq:J-tangent} holds. 
	The distance of this tangent line to the origin determines $\ell_g^{\pm m_1, \pm m_2}$ via \eqref{eq:J-length}.
    Hence $\ell^{m_1, m_2}_g$ is determined by the curve $\Gamma_{f_1, f_2}$. 
    For the cases $m_1 = 0$ and $m_2 = 0$, formula~\eqref{eq:stat length} from Lemma~\ref{lem:length-liouville} implies $\ell^{0,m_2}_g$ and $\ell^{m_1, 0}_g$ are given by the $y$-coordinate and the $x$-coordinate of the two endpoints of $\Gamma_{f_1, f_2}$, respectively.
	Since (2) implies $\Gamma_{f_1, f_2} = \Gamma_{\wt f_1, \wt f_2}$, it follows that $g$ and $\wt g$ have the same rotational marked length spectrum.
	We have proved (2) $\implies$ (1).
\\[2mm]
	\noindent \underline{\textit{Proof of length-isospectrality.}}
	Assume that each of the functions $f_1$ and $f_2$ has exactly two critical points on $\T$.
	Then, by part (2) of Lemma~\ref{lem:length-liouville}, the lengths of closed geodesics in the homology class $(0, m)$ are given by one of the following values:
	\begin{equation}\label{eq:osc - J}
		|m| J_{f_2}(1 + \max f_1), \qquad 	|m| J_{f_2}(1 + \min f_1) \qquad \text{or} \qquad 2n J_{f_1}(e) + |m| J_{f_2}(1- e) 
	\end{equation}
	where $n \in \Z_{> 0}$ is such that $2n J_{f_1}'(e) - |m| J_{f_2}'(1-e) = 0$.
	Indeed, the first two values correspond to \eqref{eq:stat length} whereas the last value corresponds to \eqref{eq:osc length}.
	We have used the assumption that $f_1$ has exactly 2 critical points, which implies that $\{x \in \T \; : \; f_1(x) + e > 0\}$ is connected for $ - \max f_1 \leq e \leq - \min f_1$.
	Note that the quantities \eqref{eq:osc - J} are all determined by the curve $\Gamma_{f_1, f_2}^-$: the first two are integer multiples of the $y$-coordinates of its endpoints, whereas the values of the last type are determined by tangent lines of $\Gamma_{f_1, f_2}^-$ with rational slope in a way similar to the preceding proof of (1) $\iff$ (2).
	A symmetric argument applies to the homology class $(m, 0)$.
	Therefore $\bigcup_{m_1 m_2 = 0}\CL^{m_1,m_2}(\T^2, g)$ is completely determined by $\Gamma_{f_1, f_2}^\pm$.
	It follows that, if $\wt f_1$ and $\wt f_2$ each have exactly two critical points, as do $f_1$ and $f_2$, and satisfy statement (2), then $\Gamma_{f_1, f_2}^\pm = \Gamma_{\wt f_1, \wt f_2}^\pm$ by the Layer-cake representation~\eqref{eq:layer-cake J} and thus $\bigcup_{m_1 m_2 = 0}\CL^{m_1,m_2}(\T^2, g) = \bigcup_{m_1 m_2 = 0}\CL^{m_1,m_2}(\T^2, \wt g)$.
	Since $\wt g$ and $g$ also have the same marked length spectrum by part (1), and $\CL^{m_1, m_2}(\T^2, g) = \{\ell_g^{m_1, m_2}\}$ for $m_1 m_2 \neq 0$ by rational integrability and Proposition~\ref{prop:geom RT} part 1, we deduce that $\wt g$ and $g$ are length-isospectral.
	The proof is complete.
\end{proof}

\begin{figure}[h]
	\centering
	\includegraphics[width=0.7\textwidth, page=3]{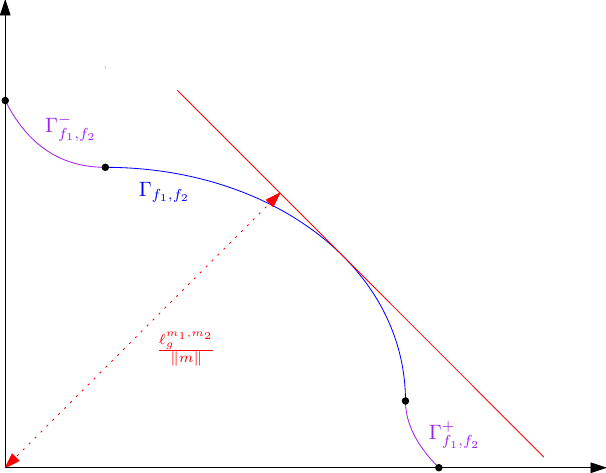} 
	\caption{The curves $\Gamma_{f_1, f_2}$ and $\Gamma_{f_1, f_2}^\pm$ in the first quadrant of $\R^2$. The red solid line is the unique tangent line to $\Gamma_{f_1, f_2}$ with slope $-m_1/m_2$ for some fixed $m_1, m_2 \in \Z_{>0}$. The value of $\ell^{m_1, m_2}_g$ is equal to the distance of this tangent line to the origin times $\|m\| = \sqrt{m_1^2 + m_2^2}$.}
	\label{fig:Gamma}
\end{figure}

\begin{remark}
	It is clear from the proof that the assumption that $f_1$, $f_2$, $\wt f_1$ and $\wt f_2$ each has exactly two critical points on $\T$ can be weakened to  the assumption that $\{x \in \T \; : \; f(x) + e > 0\}$ is connected for $f \in \{f_1, f_2, \wt f_1, \wt f_2\}$.
\end{remark}

\begin{remark}
	The argument can be modified to show that, more generally, a deformation of a Liouville metric is length-isospectral if and only if it is a rearrangement-type deformation preserving a `branched' version of the superlevel measures \eqref{eq:superlevel-measure}.
	More precisely, the requirement is that  the deformation should preserve the cyclic order of the connected components of $\{x \in \T \; : \; e + f_j(x) > 0\}$ for $j = 1,2$ as well as the measure of each individual connected component.
\end{remark}

\appendix

\section{On the Noncoincidence Condition (NCC)} \label{app:NCC}

In this appendix, we prove the genericity of NCC within the class of Liouville metrics.
Therefore, the conditions on $f_1$ and $f_2$ for Theorems~\ref{thm:main2}--\ref{thm:mainsecond} to hold are generically satisfied.
We emphasize that a property holds \emph{generically} if it holds on a residual set, i.e., a countable intersection of open dense sets.

\medskip

Recall that a Riemannian metric $g$ on $\T^2$ is said to satisfy the \emph{noncoincidence condition} (NCC) on a closed subset $I \subset \R_{\geq 0}$ if
\begin{equation*}
    I \cap \CL^{\osc}(\T^2, g) \cap \CL^{\rot}(\T^2, g) = \varnothing.
\end{equation*}
We refer to Section~\ref{sec:dynamics} for definitions of $\CL^{\osc}$ and of $\CL^{\rot}$.
The following proposition is the main result of this appendix.

\begin{prop}\label{prop:NCC gen}
    Let $I \subset \R_{\geq 0}$ be a \emph{compact} subset.
    Then the following statements hold with respect to the $C^\infty$ or $C^r$ ($r \geq 2$) topology on the space of Riemannian metrics on $\T^2$.
    \begin{enumerate}
        \item The set of metrics satisfying NCC on $I$ is open.
        \item The set of Liouville metrics satisfying NCC on $I$ is dense within the class of Liouville metrics.
    \end{enumerate}
    In particular, NCC up to $T=\infty$ (resp.\ $T<\infty$) holds for a generic set (resp.\ an open dense set) of Liouville metrics.
\end{prop}

It is important for the proof of Theorem~\ref{thm:laplacetolength} to see that NCC implies that $\CL^{\osc}$ and $\CL^{\rot}$ are not only disjoint, but locally isolated from each other by a positive distance.
This is a simple consequence of the following fact on $\CL^{\osc}$ and $\CL^{\rot}$, which we shall also establish in this appendix.

\begin{prop}\label{prop:LSP close}
    Let $g$ be a $C^2$ Riemannian metric on $\T^2$. Then both $\CL^{\rot}(\T^2, g)$ and $\CL^{\osc}(\T^2, g)$ are closed subsets of $\R_{\geq 0}$.
\end{prop}

The proofs of part (1) of Proposition~\ref{prop:NCC gen} and of~\ref{prop:LSP close} rely on the following two elementary lemmas about the length spectrum $\CL^{m,n}(\T^2, g)$ within a fixed homology class $(m,n) \in \Z^2$.
The first lemma is a simple \emph{a priori} lower bound on $\min \CL^{m,n}(\T^2, g)$. 
This allows us to reduce the analysis of $\CL^{\rot}(\T^2, g)$ and $\CL^{\osc}(\T^2, g)$ over bounded intervals to the length spectra of finitely many homology classes.
The result of this lemma is rather elementary but we shall supply a complete proof for the sake of clarity.

\begin{lemm}\label{lem:NCC mn lower bd}
Let $g$ be a $C^2$ Riemannian metric on $\T^2$ and let $T>0$.
Then there exist $N>0$ and a $C^2$-neighborhood $\CU$ of $g$ such that for all $g'\in\CU$ and all $(m,n)\in\Z^2$ satisfying $|m|+|n|>N$, we have
\begin{equation}\label{eq:mn lower bd}
\CL^{m,n}(\T^2,g')\cap [0,T]=\varnothing.    
\end{equation}
\end{lemm}

\begin{proof}
Let $g_0$ denote the flat metric on $\T^2$ induced from the Euclidean metric on $\R^2$.
By compactness, there exists a constant $c>0$ such that, for all tangent vectors $\xi \in T\T^2$,
\[
|\xi|_g \geq c |\xi|_{g_0}.
\]
After shrinking to a sufficiently small $C^2$-neighborhood $\CU$ of $g$, we may assume that
\[
|\xi|_{g'} \geq \frac{c}{2}|\xi|_{g_0}
\]
for all $g' \in \CU$.
Let $\gamma$ be any closed curve in the homology class $(m,n)$. Lift $\gamma$ to a curve $\widetilde\gamma$ in $\R^2$. Then the endpoints of $\widetilde\gamma$ differ by the translation vector $(m,n)$, and therefore, the Pythagorean theorem implies
\[
    |\gamma|_{g_0}
    \geq \sqrt{m^2+n^2}.
\]
Hence, for every $g'\in\CU$,
\[
    |\gamma|_{g'}
    \geq
    \frac{c}{2}|\gamma|_{g_0}
    \geq
    \frac{c}{2}\sqrt{m^2+n^2}
    \geq
    \frac{c}{2\sqrt 2}(|m|+|n|).
\]
Choose $N>2\sqrt 2\,T/c$. Then whenever $|m|+|n|>N$, every closed curve, and hence every closed geodesic, in the homology class $(m,n)$ has $g'$-length greater than $T$. We have proved~\eqref{eq:mn lower bd}.
\end{proof}

The second lemma establishes that the length spectrum $\CL^{m,n}(\T^2, g)$ is a closed subset of $\R$ and has a stability property under $C^2$ perturbations: roughly speaking, new points in $\CL^{m,n}(\T^2, g)$ can only arise either by `bifurcating' from some existing points in $\CL^{m,n}(\T^2, g)$ or by `descending' from $+\infty$.

\begin{lemm}\label{lem:LSP stability}
    Let $(m,n) \in \Z^2$. 
    Then the following holds.
    \begin{enumerate}
        \item For any $C^2$ metric $g$, the length spectrum $\CL^{m,n}(\T^2, g)$ is a closed subset of $\R_{\geq 0}$.
        \item Let $K \subset \R_{\geq 0}$ be a compact subset. The set of metrics $g$ satisfying
        \begin{equation*}
            \CL^{m,n}(\T^2, g) \cap K = \varnothing
        \end{equation*}
        is open in the $C^2$ topology.
    \end{enumerate}
\end{lemm}

\begin{proof}
    Let $g$ be a $C^2$ metric on $\T^2$ and let $(m,n) \in \Z^2$.
    We shall reformulate the conclusions in terms of intersection properties of certain compact sets in the phase space.
    We begin by lifting the system to the universal cover and consider the equivalent setting of $\Z^2$-periodic metrics $g$ on $\R^2$. We still denote by $\Phi_g^t$ the associated geodesic flow on the unit tangent bundle $S\R^2$.
    We denote by $S\R^2$ the unit tangent bundle on $\R^2$ with respect to $g$ and let $\pi: S\R^2 \to \R^2$ be the canonical projection.
    Write
	\begin{equation}\label{eq:NCC AB def}
    \begin{aligned}
		A_g(t) &:= \{(\xi, \Phi^t_g(\xi))\ : \ \xi \in S\R^2,\ \pi(\xi) \in [0,1]^2\}, \\ 
        B_{m,n} &:= \{(\xi, \xi') \ : \ \xi \in S\R^2,\ \pi(\xi) \in [0,1]^2, \ \xi' = \tau_{(m,n)} (\xi)\}
    \end{aligned}
	\end{equation}
    where $\tau_{(m,n)}$ is the deck translation on $S\R^2$ associated with $(m,n)$, i.e., it shifts the basepoint $\pi(\xi)$ by $(m,n)$ and leaves the direction unchanged.
    Note that $[0,1]^2$ is a fundamental domain for $\T^2$.
    Thus, the existence of a closed geodesic of $(\T^2, g)$ of length $L \geq 0$ in the homology class $(m,n)$ is equivalent to the existence of $\xi$ with $\pi (\xi) \in [0,1]^2$ such that $\Phi^L_g(\xi) = \tau_{(m,n)} (\xi)$.
    It now follows from the definitions~\eqref{eq:NCC AB def} that, for $t \geq 0$,
    \footnote{We include $0$ as a point in $\CL^{0,0}(\T^2, g)$ (by regarding a geodesic of length $0$ as a closed geodesic). Hence the equivalence~\eqref{eq:NCC LSP closedness condition} holds, in particular, for $m = n = 0$.}
    \begin{equation}\label{eq:NCC LSP closedness condition}
        t \notin \CL^{m,n}(\T^2, g) \quad \iff \quad  A_g(t) \cap B_{m,n} = \varnothing.
    \end{equation}
    On the other hand, let $K \subset \R_{\geq 0}$ be a compact subset and put 
    \begin{equation}
        A_g(K) := \bigcup_{t \in K} A_g(t).
    \end{equation}
    Then we have 
    \begin{equation}\label{eq:NCC LSP stability condition}
        \CL^{m,n}(\T^2, g) \cap K = \varnothing \quad  \iff  \quad A_g(K) \cap  B_{m,n} = \varnothing.
    \end{equation}
    Observe that $A_g(t)$, $A_g(K)$ and $B_{m,n}$ are all compact subsets of $S\R^2 \times S\R^2$. Since the geodesic flow $\Phi^t_g$ depends continuously on $t \in \R$ and on $g$ in the $C^2$ topology, the sets $A_g(t)$ and $A_g(K)$ vary continuously with respect to $t$ and $g$. 
    We also remark that, since the unit tangent bundle $S\R^2$ varies continuously with $g$, the set $B_{m,n}$ does as well. 
    It now follows that the conditions on the right-hand sides of~\eqref{eq:NCC LSP closedness condition} and~\eqref{eq:NCC LSP stability condition} are open with respect to $t$ and $g$.
    By these equivalences, we deduce that $\CL^{m,n}(\T^2, g)$ is closed and that $\CL^{m,n}(\T^2, g) \cap K = \varnothing$ holds for a $C^2$-open set of metrics $g$, as desired.
\end{proof}

Having Lemma~\ref{lem:NCC mn lower bd} and Lemma~\ref{lem:LSP stability} at hand, we may now easily prove Proposition~\ref{prop:LSP close}.

\begin{proof}
We prove the statement for $\CL^{\rot}$; the proof for $\CL^{\osc}$ is identical.
Fix $T>0$. By Lemma~\ref{lem:NCC mn lower bd}, there exists $N>0$ such that
\[
    \CL^{\rot}(\T^2,g)\cap[0,T]
    =
    \bigcup_{\substack{mn\neq 0\\ |m|+|n|\leq N}}
    \left(
        \CL^{m,n}(\T^2,g)\cap[0,T]
    \right).
\]
The right-hand side is a finite union of closed subsets of $[0,T]$ by Lemma~\ref{lem:LSP stability} part (1). 
Therefore, $\CL^{\rot}(\T^2,g)\cap[0,T]$ is closed for every $T > 0$.
Now let $L_j\in\CL^{\rot}(\T^2,g)$ and suppose $L_j\to L\in\R_{\geq 0}$. Choose $T>L$. Then $L_j\in[0,T]\cap\CL^{\rot}(\T^2,g)$ for all sufficiently large $j$. By closedness, we obtain
\[
    L\in [0,T]\cap\CL^{\rot}(\T^2,g).
\]
Hence $L\in\CL^{\rot}(\T^2,g)$, so $\CL^{\rot}(\T^2,g)$ is closed.
\end{proof}

Finally, we complete the proof of Proposition~\ref{prop:NCC gen}.

\begin{proof}[Proof of Proposition~\ref{prop:NCC gen}]
    Let us fix a compact subset $I \subset \R_{\geq 0}$. We first show that NCC on $I$ is an open property in the general space of $C^2$ Riemannian metrics on $\T^2$.

\medskip

    \noindent \underline{\textit{Proof of part (1).}} 
    Fix a $C^2$ metric $g$ and assume that $g$ satisfies NCC on $I$, i.e.,
    \begin{equation}\label{eq:I-NCC proof starting condition}
        I \cap \CL^{\osc}(\T^2, g) \cap \CL^{\rot}(\T^2, g) = \varnothing.
    \end{equation}
    Our goal is to show that NCC on $I$ holds for every metric in some $C^2$ neighborhood of $g$.

\medskip

    Since $I$ is bounded, by Lemma~\ref{lem:NCC mn lower bd}, there exist $N > 0$ and a $C^2$ neighborhood $\CU'$ of $g$ such that for all $g' \in \CU'$ we have 
    \begin{equation}\label{eq:NCC mn lower bd}
        |m| + |n| > N \implies \CL^{m,n}(\T^2, g') \cap I = \varnothing.
    \end{equation}
    Therefore, by definition of the oscillatory length spectrum,
    \begin{equation}\label{eq:NCC finite mn}
        I \cap \CL^{\osc}(\T^2, g') = I \cap \bigg( \bigcup_{\substack{|m| + |n| \leq N \\ mn = 0}} \CL^{m,n}(\T^2, g') \bigg),
    \end{equation}
    where $\CL^{m,n}(\T^2, g')$ is closed for each $(m,n)$ by Lemma~\ref{lem:LSP stability} part (1). 
    It follows that $I \cap \CL^{\osc}(\T^2, g')$ is compact.
    Similarly, $I \cap \CL^{\rot}(\T^2, g')$ is also compact. 
    Moreover, by assumption~\eqref{eq:I-NCC proof starting condition} the two compact sets are disjoint when $g' = g$.
    We may pick open neighborhoods $V_1, V_2 \subset \R_{\geq 0}$ such that 
    \begin{equation}\label{eq:V12 NCC proof}
        V_1 \cap V_2 = \varnothing, \quad I \cap \CL^{\osc}(\T^2, g) \subset V_1 \quad \text{and} \quad I \cap \CL^{\rot}(\T^2, g) \subset V_2.
    \end{equation}
    In view of~\eqref{eq:NCC finite mn}, we have 
    \begin{equation}
        \begin{cases}
            I \cap \CL^{m,n}(\T^2, g) \subset V_1  & \text{if $mn = 0$}\\
            I \cap \CL^{m,n}(\T^2, g) \subset V_2  & \text{if $mn \neq 0$}.
        \end{cases}
    \end{equation}
    Note that, by~\eqref{eq:NCC mn lower bd}, the above inclusions trivially hold when $|m|+ |n| > N$.
    Next, we apply Lemma~\ref{lem:LSP stability} part (2) to $K = I \setminus V_1$ when $mn = 0$ and to $K = I \setminus V_2$ when $mn \neq 0$. We thereby obtain a $C^2$ neighborhood $\CU'_{m,n}$ of $g$ such that, for all $g' \in \CU'_{m,n}$, we have 
    \begin{equation}\label{eq:NCC perturbed separation}
        \begin{cases}
            I \cap \CL^{m,n}(\T^2, g') \subset V_1 &\text{in the case $mn = 0$}\\
            I \cap \CL^{m,n}(\T^2, g') \subset V_2  &\text{in the case $mn \neq 0$}.
        \end{cases}
    \end{equation}
    Finally, let us define $\CU$ by the finite intersection 
    \begin{equation*}
        \CU = \CU' \cap \bigg( \bigcap_{\substack{(m,n) \in \Z^2\\ |m| + |n| \leq N}} \CU'_{m,n}\bigg).
    \end{equation*}
    By~\eqref{eq:NCC mn lower bd} and~\eqref{eq:NCC perturbed separation}, for all $g' \in \CU$ we have
    \begin{equation}
        I \cap \CL^{\osc}(\T^2, g') \subset V_1 \quad \text{and} \quad I \cap \CL^{\rot}(\T^2, g') \subset V_2.    
    \end{equation}
    Since $V_1 \cap V_2 = \varnothing$, we have proved that NCC holds on $I$ for all $g'$ in the neighborhood $\CU$, as desired.

\medskip

    The preceding proof gives $C^2$-openness. This obviously implies the openness in the finer topologies considered here. 
    The proof of part (1) is now complete.

    \medskip

	\noindent \underline{\textit{Proof of part (2).}}
    We now restrict to the class of Liouville metrics and show that a dense set of Liouville metrics satisfies NCC on $I$.

\medskip

    Let $g$ be a Liouville metric whose line element takes the form~\eqref{eq:deform}.
    Recall that, by Proposition~\ref{prop:geom RT} part (1) and the rational integrability of $g$, when $mn \neq 0$, the set $\CL^{m,n}(\T^2, g)$ is a singleton. 
    Let us denote $\CL^{m,n}(\T^2, g) = \{\ell^{m,n}_g\}$.
    We first observe that the desired density result follows from the density of the condition
    \begin{equation}\label{eq:NCC density goal 1}
        \ell^{m,n}_g \notin \CL^{0,k}(\T^2, g)
    \end{equation}
    and the analogous condition 
    \begin{equation}\label{eq:NCC density goal 2}
        \ell^{m,n}_g \notin \CL^{k,0}(\T^2, g)
    \end{equation}
    on the Liouville metric $g$ for \emph{fixed} nonzero integers $m, n$ and $k$, with respect to the considered topology.
    Indeed, an argument similar to part (1) shows that, with $m$, $n$ and $k$ fixed, both~\eqref{eq:NCC density goal 1} and~\eqref{eq:NCC density goal 2} are open conditions on $g$.
    The openness follows from the continuity of $g\mapsto \ell^{m,n}_g$ in the Liouville class, together with Lemma~\ref{lem:LSP stability} applied to the fixed class $(0,k)$ or $(k,0)$.
    Once both~\eqref{eq:NCC density goal 1} and~\eqref{eq:NCC density goal 2} are also shown to be dense conditions, we can find a $C^2$ neighborhood $\CU'$ of $g$ and $N > 0$ as in part (1) so that, for metrics in $\CU'$, NCC on $I$ reduces to finitely many conditions of these two types.
    Taking a finite intersection of these conditions gives an open and dense subset of $\CU'$ consisting of Liouville metrics satisfying NCC on $I$.

\medskip

    Hence, we shall focus on proving the density of metrics $g$ satisfying~\eqref{eq:NCC density goal 1} and~\eqref{eq:NCC density goal 2}.
    We shall consider only~\eqref{eq:NCC density goal 1}; the proof of the other case is entirely symmetric.
    More concretely, we start with the case in which~\eqref{eq:NCC density goal 1} does \emph{not} hold, and construct arbitrarily small perturbations of $f_1$ appearing in the conformal factor of $g$ that break the coincidence.

\medskip

	The key idea is to choose an interval $J$ such that the annulus $J \times \T \subset \T^2$ does not intersect orbits in the homology class $(0,k)$, and then perturb $f_1$ on $J$ to change $\ell^{m,n}_g$ without affecting $\CL^{0,k}(\T^2, g)$.
    By Lemma~\ref{lem:length-liouville}, there are two types of closed geodesics contributing to $\CL^{0,k}(\T^2, g)$: those contained in a circle $\{x_1\}\times\T$ for some critical point $x_1$ of $f_1$, and those not contained in such circles. 
    In the first case, the geodesics are $|k|$-fold iterates of a $(0,1)$ closed geodesic lying on a critical torus.\footnote{If $x_1$ is a local minimum of $f_1$, then this geodesic is hyperbolic and corresponds to the separatrix in the phase diagram.}
    In the second case, the geodesics oscillate around geodesics of the first type that correspond to a local maximum of $f_1$.
    We now focus on this second case. The heuristic is that, for bounded length, such geodesics cannot get too close to a geodesic of the first type that corresponds to a local minimum of $f_1$.

    \medskip
    
    Let us fix nonzero $m$, $n$ and $k$ from now.
    To obtain this perturbation interval $J$, let us assume, without loss of generality, that $f_1$ belongs to the open and dense set of Morse functions on $\T$; that is, all critical points of $f_1$ are nondegenerate.
	Recall that, by Lemma~\ref{lem:length-liouville}, such geodesics are associated with $j \in \Z_{>0}$ and satisfy 
	\begin{equation}\label{eq:NCC osc interval}
		2j \int_{x_1^{\min}}^{x_1^{\max}} \frac{\dif x_1}{\sqrt{e + f_1(x_1)}} = |k| \int_{0}^{1} \frac{\dif x_2}{\sqrt{1 - e + f_2(x_2)}} \quad \text{for some } e \in (- \max f_1, -\min f_1),
	\end{equation}
	where $(x_1^{\min}, x_1^{\max})$ is a connected component of $\{x_1 \in \T \, : \, e + f_1(x_1) > 0\}$ and the $x_1$-coordinate of the geodesic oscillates within $[x_1^{\min}, x_1^{\max}]$.
	Since $k$ is fixed and $1 - e + f_2(x_2) > 1 + \min f_1 + f_2(x_2) > 0$, the right-hand side of~\eqref{eq:NCC osc interval} is bounded from above uniformly in $e$. 
    On the other hand, by the assumption on $f_1$, for all $e < -\min f_1$ sufficiently close to $-\min f_1$, the set $\{x_1 \, : \, e + f_1(x_1) > 0\}$ consists of a fixed number of intervals whose endpoints vary continuously in $e$ and converge to global minima of $f_1$.
    In particular, for any choice of the interval $(x_1^{\min}, x_1^{\max})$, the integral on the left-hand side diverges as $e \nearrow -\min f_1$.

\medskip

    Since there are only finitely many such intervals and the divergence is uniform over this finite collection, while the right-hand side is uniformly bounded in $e$, we may choose $c>0$ such that the equality~\eqref{eq:NCC osc interval} has no solution for
    \[
    e\in[-\min f_1-c,-\min f_1).
    \]
    More precisely, for any $e$ in the interval $[-\min f_1 - c,\, -\min f_1)$, the left-hand side of~\eqref{eq:NCC osc interval} is strictly larger than the right-hand side. 
    Let us now pick a nontrivial open interval
    \begin{equation}\label{eq:NCC J choice}
        J \subset \{x_1 \, : \, f_1(x_1) < \min f_1 + c\}
    \end{equation}
    with the requirement that $J$ does not contain any critical point of $f_1$. By construction, the annulus $J \times \T$ does not intersect \textit{any} closed geodesics in the homology class $(0,k)$.

    \medskip
    
	Next, choose a perturbation $f_1\mapsto f_1+\eps\phi$ (with $f_2$ left unperturbed), where $\phi$ is a smooth, nonzero, nonnegative function compactly supported in $J$, and denote by $g_\eps$ the corresponding perturbed Liouville metric.
    For sufficiently small $|\eps|$, we can assume that the perturbed function $f_1 + \eps \phi$ creates no new critical points on $J$, and hence that this perturbation does not create additional geodesics of the first type.
    Furthermore, by~\eqref{eq:NCC J choice}, for all sufficiently small $|\eps|$ and all $x \in \T$ we have
    \begin{equation}\label{eq:NCC level set inv}
        f_1(x) + \eps \phi(x) > \min f_1 + c \iff f_1(x) > \min f_1 + c.
    \end{equation}
	Since $\phi$ vanishes on critical points of $f_1$ by the choice of $J$, the lengths of existing closed geodesics of the first type remain unchanged. 
	On the other hand, we consider the geodesics of the second type, which are described by the condition~\eqref{eq:NCC osc interval}.
    Note that, for $e < -\min f_1 - c$, it follows from~\eqref{eq:NCC level set inv} and the support condition on $\phi$ that both sides of the equation in~\eqref{eq:NCC osc interval} are independent of $\eps$.
    In particular, the geodesics whose oscillation intervals do not intersect $J$ are unaffected by the perturbation.
    By the choice of $c$, for $e \geq -\min f_1 - c$ the left-hand side of the equation in~\eqref{eq:NCC osc interval} is strictly larger than the right-hand side.
    By the compact-support condition of $\phi$, it is not hard to see that the left-hand side of~\eqref{eq:NCC osc interval} varies uniformly continuously in $\eps$.
    When $|\eps|$ is sufficiently small, the equality in~\eqref{eq:NCC osc interval} therefore does not hold for any $-\min f_1 - c \leq e < - \min f_1$.
    In this case, this perturbation does not create additional geodesics of the second type.
	In summary, we deduce that 
    \[
        \CL^{0,k}(\T^2,g_\eps) = \CL^{0,k}(\T^2, g)
    \] 
    for all sufficiently small $|\eps|$.

\medskip

    On the other hand, taking the derivative of the length functional, we find that the first-order variation of $\ell^{m,n}_{g_\eps}$ in $\eps$ is given by the integral
	\begin{equation}
		\frac{\dif}{\dif \eps} \ell^{m,n}_{g_\eps}\bigg|_{\eps = 0} = \frac12 \int \frac{\phi(\gamma_1(t))}{\sqrt{1 + f_1(\gamma_1(t)) + f_2(\gamma_2(t))}} \sqrt{\dot\gamma_1(t)^2 + \dot\gamma_2(t)^2} \,\dif t,
	\end{equation}
	over one period, where $\gamma$ is any closed geodesic in the homology class $(m,n)$.
	Since $m \neq 0$, the coordinate $\gamma_1$ winds nontrivially around $\T$ and therefore intersects the support of $\phi$. Thus, the above integral is strictly positive.
	By the inverse function theorem, the quantity $\ell^{m,n}_{g_\eps}$ covers an open neighborhood of $\ell^{m,n}_g$ as $\eps$ varies in a small interval around $0$.
    
    \medskip
	
    Finally, by the standard argument using Sard's Lemma (see for example \cite{Hirsch}), the spectrum $\CL^{0,k}(\T^2, g)$ has zero Lebesgue measure and hence does not contain an interval. 
    It follows that there exist arbitrarily small $\eps$ such that
    \begin{equation*}
        \ell^{m,n}_{g_\eps} \notin \CL^{0,k}(\T^2, g_\eps).
    \end{equation*}
	Since $g_\eps \to g$ in the $C^\infty$ topology, this proves the desired density in the $C^r$ category for $r \in \{2,3,\ldots,\infty\}$.

    \medskip

    \noindent \underline{\textit{Proof of genericity.}} By statements (1) and (2), for every $T < \infty$, NCC up to $T$ is an open and dense condition on the class of Liouville metrics. Taking a countable intersection along a sequence $T \to \infty$, we deduce that NCC at all lengths ($T = \infty$) holds for a generic set within the class of Liouville metrics. 
    The proof is complete.
\end{proof}

\bibliographystyle{alpha}
\bibliography{ref}

\end{document}